\DocumentMetadata{uncompress}
\documentclass[11pt]{amsart}

 \usepackage{amsfonts,graphics,amsmath,amsthm,amsfonts,amscd,amssymb,amsmath,latexsym,multicol,
 mathrsfs}
\usepackage{epsfig,url}
\usepackage{flafter}
\usepackage{fancyhdr}
\usepackage{hyperref}
\usepackage[dvipsnames]{xcolor}
\usepackage{hyperref}
\hypersetup{
	colorlinks=true,
	linkcolor=black,
	citecolor=ForestGreen,
	filecolor=cyan,
	urlcolor=purple
}

\usepackage{graphicx}
\usepackage{xcolor}
\usepackage{float}
\usepackage{bm}
\usepackage{enumitem}
\usepackage{multirow}
\usepackage{scalerel,stackengine}
\usepackage{stmaryrd}
\usepackage{datetime}
\usepackage[title]{appendix}
\usepackage{url}
\usepackage{tabulary}
\usepackage{booktabs}
\usepackage{tikz}
\usetikzlibrary{positioning}
\usepackage{verbatim}
\usepackage[utf8]{inputenc}
\usepackage[english]{babel}
\usepackage{csquotes}
\usepackage{comment}
\usepackage[
backend=biber,
style=alphabetic,
sorting=anyt,
doi=false,isbn=false,url=false,
maxalphanames=99,maxbibnames=99,backref=true,
]{biblatex}
\DefineBibliographyStrings{english}{
  backrefpage  = {\hspace{-0.35em}},
  backrefpages = {\hspace{-0.35em}},
}

\DeclareFieldFormat{bracketswithperiod}{\mkbibbrackets{#1}}

\renewbibmacro*{pageref}{%
  \iflistundef{pageref}
    {}
    {\printtext[bracketswithperiod]{%
       \ifnumgreater{\value{pageref}}{1}
         {\bibstring{backrefpages}\ppspace}
         {\bibstring{backrefpage}\ppspace}%
       \printlist[pageref][-\value{listtotal}]{pageref}}%
     }}

 \theoremstyle{plain}
 \newtheorem{theorem}[subsection]{Theorem}
 \newtheorem{lemma}[subsection]{Lemma}

 \newtheorem{conjecture}[subsection]{Conjecture}

 \theoremstyle{definition}
 \newtheorem{definition}[subsection]{Definition}
 \newtheorem{definitionnotation}[subsection]{Notation}

 \newtheorem{exa}[subsection]{Example}
 \newtheorem{remark}[subsection]{Remark}

 \theoremstyle{remark}

 \theoremstyle{plain} % just in case the style had changed
\newcommand{\thistheoremname}{}
\newtheorem{genericthm}[subsection]{\thistheoremname}

  \newtheorem*{genericthm*}{\thistheoremname}
\newenvironment{namedthm*}[1]
  {\renewcommand{\thistheoremname}{#1}%
   \begin{genericthm*}}
  {\end{genericthm*}}
 
\newcommand{\abs}[1]{\left\vert#1\right\vert}

\newcommand{\spec}{\operatorname{Spec}}

\newcommand{\codim}{\operatorname{codim}}

\newcommand{\hg}{\operatorname{H}}
\newcommand{\chara}{\operatorname{char}}

\newcommand{\supp}{\operatorname{Supp}}
\newcommand{\cl}{\operatorname{Cl}}
\newcommand{\sdeg}{\operatorname{sdeg}}
\newcommand{\ssdeg}{\operatorname{ssdeg}}
\newcommand{\centre}{\operatorname{centre}}

\newcommand{\CY}{\operatorname{\mathcal{CY}}}

\newcommand{\Spair}{\operatorname{\mathcal{CY}_{klt}}}

\newcommand{\ex}{\operatorname{Ex}}
\newcommand{\red}{\operatorname{red}}

\newcommand{\coeff}{\operatorname{coeff}}
\newcommand{\nor}{\operatorname{nor}}
\newcommand{\vol}{\operatorname{vol}}
\newcommand{\mult}{\operatorname{mult}}
\newcommand{\tr}{\operatorname{tr}}

\newcommand{\bdiv}{\operatorname{b}}
\newcommand{\cur}{\operatorname{cur}}

\newcommand{\mni}{\medskip\noindent}

\newcommand{\mbb}{\mathbb}
\newcommand{\QQ}{\mbb{Q}}
\newcommand{\NN}{\mbb{N}}
\newcommand{\ZZ}{\mbb{Z}}
\newcommand{\CC}{\mbb{C}}
\newcommand{\RR}{\mbb{R}}

\newcommand{\PP}{\mbb{P}}

\newcommand{\mc}{\mathcal}

\newcommand{\mcP}{\mc{P}}

\newcommand{\mf}{\mathfrak}

\newcommand{\mb}{\mathbf}

\newcommand{\wh}{\widehat}
\newcommand{\wt}{\widetilde}
\newcommand{\ol}{\overline}

\usepackage{graphicx}
\usepackage{xcolor}
\usepackage{float}
\usepackage{amssymb}
\usepackage{amsmath}
\usepackage{mathrsfs}
\usepackage{bm}
\usepackage[all,cmtip]{xy}
\usepackage{enumitem}
\usepackage[utf8]{inputenc}
\usepackage[english]{babel}
\usepackage{multirow}
\usepackage{mathtools}
\usepackage{scalerel,stackengine}
\usepackage{stmaryrd}
\usepackage{datetime}
\usepackage[title]{appendix}
\usepackage{url}
\usepackage{tabulary}
\usepackage{booktabs}
\usepackage{tikz}
\usetikzlibrary{positioning}
\usepackage{verbatim}
\usepackage{braket}

\DeclarePairedDelimiter\floor{\lfloor}{\rfloor}
\DeclareFieldFormat{postnote}{#1}
\DeclareFieldFormat{multipostnote}{#1}

\makeatletter
\newcommand*{\da@rightarrow}{\mathchar"0\hexnumber@\symAMSa 4B }
\newcommand*{\da@leftarrow}{\mathchar"0\hexnumber@\symAMSa 4C }
\newcommand*{\xdashrightarrow}[2][]{%
  \mathrel{%
    \mathpalette{\da@xarrow{#1}{#2}{}\da@rightarrow{\,}{}}{}%
  }%
}
\newcommand{\xdashleftarrow}[2][]{%
  \mathrel{%
    \mathpalette{\da@xarrow{#1}{#2}\da@leftarrow{}{}{\,}}{}%
  }%
}
\newcommand*{\da@xarrow}[7]{%
  \sbox0{$\ifx#7\scriptstyle\scriptscriptstyle\else\scriptstyle\fi#5#1#6\m@th$}%
  \sbox2{$\ifx#7\scriptstyle\scriptscriptstyle\else\scriptstyle\fi#5#2#6\m@th$}%
  \sbox4{$#7\dabar@\m@th$}%
  \dimen@=\wd0 %
  \ifdim\wd2 >\dimen@
    \dimen@=\wd2 %   
  \fi
  \count@=2 %
  \def\da@bars{\dabar@\dabar@}%
  \@whiledim\count@\wd4<\dimen@\do{%
    \advance\count@\@ne
    \expandafter\def\expandafter\da@bars\expandafter{%
      \da@bars
      \dabar@ 
    }%
  }%  
  \mathrel{#3}%
  \mathrel{%   
    \mathop{\da@bars}\limits
    \ifx\\#1\\%
    \else
      _{\copy0}%
    \fi
    \ifx\\#2\\%
    \else
      ^{\copy2}%
    \fi
  }%   
  \mathrel{#4}%
}
\makeatother

\title{\large S\MakeLowercase{tein degree on log} C\MakeLowercase{alabi-}Y\MakeLowercase{au fibrations}}
\author{\large C\MakeLowercase{aucher} B\MakeLowercase{irkar} \MakeLowercase{and} \large S\MakeLowercase{antai} Q\MakeLowercase{u}}
\date{\today}

\begin{document}

\begin{abstract}
    We prove a conjecture proposed by the first author on 
    boundedness of Stein degree of divisors on log Calabi-Yau fibrations.
    More precisely, for $d\in \NN$ and $t\in (0,1]$, 
    let $(X, B)\to Z$ be a log Calabi-Yau fibration of relative dimension $d$, 
    and let $S$ be a horizontal$/Z$ irreducible component
    of $B$ whose coefficient in $B$ is $\ge t$.  
    We show that the number of irreducible components of a general fibre of $S\to Z$ is bounded
    from above depending only on $d,t$.
\end{abstract}

\maketitle

\tableofcontents

\addtocontents{toc}{\protect\setcounter{tocdepth}{1}}

%%%%%%%%%%%%%%%%%%%%%%%%%%
%%%%%%%%%%%%%%%%%%%%%%%%%%

\section{Introduction}

We work over an algebraically closed field $\mbb K$ of characteristic zero unless stated otherwise.

%%%%%%%%%%%%%%%%%%%%%%%%%%%%%%%%%%%%%%%

\mni
{\textbf{\sffamily{Boundedness of Stein degree of divisors on log Calabi-Yau fibrations.}}}
Stein degree is first introduced in the work \cite{B-moduli} by the first author,
which measures the number of connected components of general fibres of a projective morphism.
More specifically, let $S\to Z$ be a projective morphism between varieties, and
let $S\to V\to Z$ be the Stein factorisation.  
The \emph{Stein degree} (or \emph{S-degree} for short) of $S$ over $Z$ is defined to be 
\[ \sdeg (S/Z) := \deg (V/Z). \]
If $S\to Z$ is not surjective, $\sdeg (S/Z)$ is defined to be zero by convention.

One of the crucial ingredients in \cite{B-moduli} to establish the existence of moduli space of 
semi-log canonical (slc) stable minimal models (see \cite[\S 1]{B-moduli}) is the boundedness of Stein degree
of log canonical (lc) centres on log Calabi-Yau fibrations.  Recall that a \emph{log Calabi-Yau fibration}
$(X, B)\to Z$ consists of an lc pair $(X, B)$ and a contraction $X\to Z$ of varieties
such that $K_X + B \sim_{\RR} 0/Z$; see \cite[\S 1]{B-moduli}.

\begin{theorem}[\protect{\cite[Theorem 1.15]{B-moduli}}]\label{B-moduli-bnd-S-degree}
    Let $d\in \NN$.  Let $(X, B)\to Z$ be a log Calabi-Yau fibration of dimension $d=\dim X$.
    Then $\sdeg (I/Z)$ is bounded from above depending only on $d$ for every lc centre 
    $I$ of $(X, B)$.
\end{theorem}

Based on Theorem~\ref{B-moduli-bnd-S-degree}, the first author conjectures
a more general form of boundedness of Stein degree of divisors on 
log Calabi-Yau fibrations.

\begin{conjecture}[\protect{\cite[Conjecture 1.16]{B-moduli}}]\label{B-ori-conj}
    Let $d\in \NN$ and let $t\in \RR^{>0}$.
	Let $(X, B)\to Z$ be a log Calabi-Yau fibration 
	of dimension $d = \dim X$, and let $S$ be a horizontal$/Z$ irreducible component of $B$ 
	whose coefficient in $B$ is $\ge t$.
	Then $\sdeg (S/Z)$ is bounded from above depending only on $d,t$.
\end{conjecture}

Note that Theorem~\ref{B-moduli-bnd-S-degree} is stated for lc centres $I$
which may not be divisors, but the theorem can easily be reduced to the case when $I$ is a divisor.  
Moreover, we can also ask a similar question when $S$ is vertical$/Z$ in Conjecture~\ref{B-ori-conj}
in which case we consider $\sdeg (S/T)$, where $T$ is the image of $S$ in $Z$; cf. \cite[page 8]{B-moduli}.
More generally, let $S\to Z$ be a projective morphism between varieties,
and let $T$ be the image of $S$ in $Z$.
We call the Stein degree $\sdeg (S/T)$ the \emph{strong Stein degree} (or \emph{SS-degree} for short)
of $S$ over $Z$, and we denote it by $\ssdeg (S/Z)$.

\mni
{\textbf{\sffamily{Main results $\mc{H}(d, t)$ and $\mc{V}(d, t)$.}}}
The main result of this paper is a stronger version of Conjecture~\ref{B-ori-conj} 
on boundedness of Stein degree of divisors on log Calabi-Yau fibrations.
Indeed, we show the boundedness of Stein degree in the context of generalised pairs.
A \emph{generalised log Calabi-Yau fibration} $(X, B + \mb{M}) \to Z$ consists of a generalised log canonical 
generalised pair $(X, B + \mb{M})$ and a contraction of varieties $X\to Z$ 
such that $K_X + B + \mb{M}_X \sim_{\RR}0/Z$; see \S \ref{preliminaries}.

\begin{theorem}[$\mc{H}(d,t)$]\label{dim-d-t}
    Let $d\in \NN$ and let $t\in (0,1]$.
	Let $(X, B + \mb{M})\to Z$ be a generalised log Calabi-Yau fibration 
	of relative dimension $d$, that is, a general fibre of $X\to Z$ has dimension $d$.
    Let $S$ be a horizontal$/Z$ irreducible component of $B$ 
	whose coefficient in $B$ is $\ge t$.
	Then $\sdeg (S^{\nor}/Z)$ is bounded from above depending only on $d,t$, 
    where $S^{\nor}$ is the normalisation of $S$.
\end{theorem}

Conjecture~\ref{B-ori-conj} is derived immediately by setting 
$\mb{M} = 0$ in Theorem~\ref{dim-d-t}.
Note that the conclusion in Theorem~\ref{dim-d-t} is stronger than Conjecture~\ref{B-ori-conj}
as $\sdeg (S^{\nor}/Z)$ is equal to the number of irreducible components of a general fibre of $S\to Z$,
while $\sdeg (S/Z)$ is equal to the number of connected components of a general fibre of $S\to Z$.

\medskip

On the other hand, when $S$ is vertical$/Z$, we show the boundedness of $\ssdeg (S^{\nor}/Z)$
by assuming that the image of $S$ in $Z$ is a divisor and that
$(X, B + \mb{M})\to Z$ is a \emph{Fano type generalised log Calabi-Yau fibration} 
(or a \emph{Fano type fibration} for short), that is, $(X, B + \mb{M})\to Z$
is a generalised log Calabi-Yau fibration and $X$ is of Fano type over $Z$.
Recall that a contraction of varieties $X\to Z$ is called \emph{of Fano type} 
if there exists a big$/Z$ $\RR$-divisor $\Gamma$ such that $(X, \Gamma)$ is Kawamata log terminal (klt)
and $K_X + \Gamma \sim_{\RR}0/Z$; see \S\ref{fano-type-defn}.

\begin{theorem}[$\mc{V}(d, t)$]\label{vertical-bnd-intro}
    Let $d\in \NN$ and let $t \in (0,1]$.  Let 
    $(X, B + \mb{M})\to Z$ 
    be a Fano type generalised log Calabi-Yau fibration
    of relative dimension $d$.  
    Let $S$ be a vertical$/Z$ irreducible component of $B$ such that
    \begin{itemize} 
        \item [\emph{(1)}] the coefficient of $S$ in $B$ is $\ge t$, and
        \item [\emph{(2)}] the image of $S$ in $Z$ has codimension one in $Z$.
    \end{itemize}
    Then $\ssdeg (S^{\nor}/Z)$ is bounded from above depending only on $d, t$,
    where $S^{\nor}$ is the normalisation of $S$.
\end{theorem}

In a forthcoming paper, we give examples showing that 
the strong Stein degree is unbounded if $X\to Z$ is not of Fano type in
Theorem~\ref{vertical-bnd-intro}, even for $d=t=1$.

Let $(X, B + \mb M)\to Z$ be a generalised log Calabi-Yau fibration
(respectively, a Fano type fibration).
Then Theorem~\ref{dim-d-t} (respectively, Theorem~\ref{vertical-bnd-intro}) can also be stated 
for prime divisors $S$ over $X$ with generalised log discrepancy $a(S, X, B + \mb M) \le 1-t$.  
We can extract such $S$, then the situation is reduced to the case when $S$ is a divisor on $X$.

The proofs of $\mc{H}(d, t)$ and $\mc{V}(d, t)$ are proceeded by the Minimal Model Program \cite{BCHM, B-Zhang},
boundedness of complements \cite{B-Fano}, boundedness of $\epsilon$-lc Fano varieties \cite{B-BAB},
toroidal geometry methods \cite{birkar2024irrationality}, and interweaving inductions
on the relative dimension $d$ in $\mc{H}(d, t)$ and $\mc{V}(d, t)$.
In particular, Theorem~\ref{vertical-bnd-intro} is one of the crucial ingredients
in the proof of Theorem~\ref{dim-d-t}.

%%%%%%%%%%%%%%%%%%%%%%%%%%%%%%%%%%%%%%%%%%%%

\mni
{\textbf{\sffamily{Relative $\RR$-linear system of Fano fibrations.}}}
Log Calabi-Yau fibrations arise naturally in the study 
of Fano fibrations.  Recall that a contraction $X\to Z$ of normal varieties
is called a \emph{Fano fibration} if $X$ has klt singularities and $-K_X$ is ample$/Z$.
Let $\abs{-K_X}_{/Z, \RR}$ be the relative $\RR$-linear system of the ample$/Z$ divisor $-K_X$, that is,
the set of all effective $\RR$-divisors that are $\RR$-linearly equivalent to $-K_X$ over $Z$.
Pick $B\in \abs{-K_X}_{/Z, \RR}$ such that $(X, B)$ is an lc pair, 
then $(X, B)\to Z$ is a log Calabi-Yau fibration.

Let $X\to Z$ be a Fano fibration of relative dimension $d$.  If $X$ is $\epsilon$-lc 
over the generic point of $Z$ for some fixed $\epsilon > 0$, then a general fibre $F$ of $X\to Z$ 
is an $\epsilon$-lc Fano variety of dimension $d$.
By \cite[Theorem 1.1]{B-BAB}, $F$ belongs to a bounded family of projective varieties.
Denote by $\mc{F}(d, \epsilon)$ the set of all log Calabi-Yau fibrations $(X, B)\to Z$
of relative dimension $d$ such that 
$X\to Z$ is a Fano fibration and $X$ is $\epsilon$-lc over the generic point of $Z$.
Denote by $\mc{F}(d)$ the union of the sets $\mc{F}(d, \epsilon)$ for all $\epsilon > 0$.

For an arbitrary fixed $\epsilon>0$, the set of general fibres of 
the fibrations $X\to Z$ in $\mc{F}(d, \epsilon)$
form a bounded family.  However, it is well-known that the set of general fibres
of fibrations in $\mc{F}(d)$ is not bounded even for $d=2$; see \cite[Example 1.2]{B-BAB}.

Let $t\in (0,1]$.  Pick an arbitrary $(X, B)\to Z$ in $\mc{F}(d)$,
and assume that $S$ is a horizontal$/Z$ irreducible component of $B$ 
whose coefficient in $B$ is $\ge t$.
Then $B$ belongs to the relative $\RR$-linear system $\abs{-K_X}_{/Z, \RR}$.
Now Theorem~\ref{dim-d-t} demonstrates the phenomenon that the number of irreducible components of general fibres of $S\to Z$ is bounded
although $\mc{F}(d)$ is not a relatively bounded family of fibrations; cf. \S\ref{r-bnd-def}.

%%%%%%%%%%%%%%%%%%%%%%%%%%%%%%%%%%%%%%%%%%%%

\mni
{\textbf{\sffamily{Local and global geometry of log Calabi-Yau fibrations.}}}
Let $(X, B)\to Z$ be a log Calabi-Yau fibration with $\dim X = 3$ and $\dim Z = 1$,
and let $S$ be a horizontal$/Z$ irreducible component of $B$ whose coefficient in $B$ is $\ge t$.
Denote by $F$ a general fibre of $X\to Z$.
Denote by $B_F$ (respectively, by $S_F$) the restriction of $B$ to $F$ (respectively, of $S$ to $F$).  
We can write
\[ K_F + B_F \sim_{\RR} 0. \]
Then $\sdeg (S^{\nor}/Z)$ is equal to the number of irreducible components of $S_F$.

Let $\mc{C}_2$ be the set of log Calabi-Yau projective couples of dimension two,
that is, $\mc{C}_2$ consists of lc pairs $(Y, D)$, where $Y$
is a normal projective surface and
$D$ is a reduced divisor on $Y$ such that $K_Y + D \sim_{\RR} 0$.
Let $\Lambda$ be the set of integers 
\[ \Lambda := \Set{ n\in \NN \mid n\text{ is the number of irreducible components of }D \text{ for some }(Y, D)\in \mc{C}_2 }. \]
It turns out that $\Lambda$ is not a bounded set of integers
as shown in the following example.

\begin{exa}\label{local-counter}
    Let $Y_0 := \PP^2$.  Let $L_0$ be a transversal union of three lines in $Y$, 
    which is a reduced simple normal crossing (or \emph{snc} for short) divisor on $Y_0$.  Now assume that 
    $(Y_i, L_i)$ is already defined, where $Y_i$ is nonsingular and $L_i$
    is a reduced snc divisor with $i+3$ irreducible components.
    Let $Y_{i+1}$ be the blowing up of $Y_i$ at a singular point of $L_i$,
    and let $L_{i+1}$ be the sum of the birational transform of $L_i$
    and the reduced exceptional divisor of $Y_{i+1}\to Y_i$.  Then $(Y_{i+1}, L_{i+1})$
    is a log Calabi-Yau couple of dimension two. 
    Inductively, for every $n\in \NN$, the log Calabi-Yau couple $(Y_n, L_n)\in \mc{C}_2$
    is defined, where $L_n$ is a cycle of $n+3$ nonsingular rational curves.
\end{exa}

Now let $(X, B)\to Z$ be a log Calabi-Yau fibration with $\dim X = 3$ and $\dim Z = 1$, 
let $S$ be an irreducible component of $B$ whose coefficient in $B$ is $\ge t$,
and let $F$ be a general fibre of $X\to Z$ as above.
By Theorem~\ref{dim-d-t}, there is an $r\in \NN$ depending only on $t$ such that 
\[ \mf{n}(S_F) = \sdeg (S^{\nor}/Z)\le r, \]
where $\mf{n}(S_F)$ denotes the number of irreducible components of $S_F$; cf. \S\ref{pre-stein-deg-divisors}.
Let $(Y_n, L_n)\in \mc{C}_2$ be a couple constructed in Example~\ref{local-counter} with $n > r$.
Then the divisor $L_n$ can not be realised as a general fibre
of an irreducible component $S$ of the boundary $B$ of a three-dimensional
log Calabi-Yau fibration $(X/Z, B)$ over a smooth curve; otherwise, if 
$L_n = S_F$ for such an $S$ as above, then $L_n$
has at most $r$ irreducible components, which contradicts the construction of $L_n$ as $n > r$.
This reveals the fact that certain boundedness for log Calabi-Yau
fibrations is dominated by global geometry of the fibrations, 
and it can not be detected by local geometry of the general fibres, which are log Calabi-Yau pairs of lower dimension.

%%%%%%%%%%%%%%%%%%%%%%%%%%%%%%%%%%%%%%%%

\mni
{\textbf{\sffamily{Structure of the paper.}}}
Let $d\in \NN$ and let $t\in (0,1]$.  The proofs of Theorem~\ref{dim-d-t} and Theorem~\ref{vertical-bnd-intro}
proceed along the path
\begin{equation}
    \mc{H}(1, t) \Longrightarrow \mc{V}(1, t) \Longrightarrow \mc{V} (d, t) \Longrightarrow \mc{H} (d, t). \label{path} \tag{$\star$}
\end{equation}
The rest of this paper is organised along this path and is divided into the following sections.

Section~\ref{preliminaries} clarifies the notation and conventions in this paper,
such as divisors, singularities of pairs, generalised pairs, relatively bounded family 
of projective couples, and some preliminary results that will be used in subsequent sections.

Section~\ref{toroidal-section} provides the necessary materials on toroidal geometry to be used in this paper.

Section~\ref{threefolds-rel-dim-one} gives the proofs of $\mc{H}(1, t)$ and $\mc{V}(1,t)$
in \S\ref{pf-S-1-t} and \S\ref{pf-V-1-t} respectively.  In particular,
\S\ref{pf-pre-V-1-t} presents a result (Theorem~\ref{fir-one-u-t}) that is more general than $\mc{V}(1, t)$.  

Section~\ref{pf-V-d-t} gives the proof of $\mc{V}(d,t)$.
The proof proceeds by induction on the relative dimension $d$.
Note that $\mc{V}(1,t)$ holds as proved in \S\ref{pf-V-1-t}.
Then the proof of $\mc{V}(d,t)$ relies on the proof of the main result of \cite{birkar2024irrationality}.

Section~\ref{pf-main-results} gives the proof of $\mc{H}(d, t)$,
which is also accomplished by induction on the relative dimension $d$.
The first step of the induction $\mc{H}(1,t)$ is established in \S\ref{pf-S-1-t},
then $\mc{V}(k, t)$ for $1\le k\le d-1$ is involved as a part of the inductive process.

%%%%%%%%%%%%%%%%%%%%%%%%%%%%%%%%%%%%%%%%%

\mni
{\textbf{\sffamily{Sketch of some proofs.}}}
We sketch the main ideas in the proofs of $\mc{H}(d,t)$ and $\mc{V}(d,t)$.
An easy observation is that the number of irreducible components of a general fibre is birationally invariant.
More precisely, let $X\to Z$ be a contraction of varieties, and let $S$ be a horizontal$/Z$ prime divisor on $X$.
Assume that $X\dashrightarrow X'$ is a birational map over $Z$ such that
the centre $S'$ of $S$ on $X'$ is also a prime divisor, then we have (cf. \S\ref{pre-stein-deg-divisors})
\[ \sdeg (S^{\nor}/Z) = \sdeg \big( (S')^{\nor} /Z\big). \]
This enables us to modify $X$ to some other birational model $X'$ on which certain boundedness 
results can be applied to deduce the boundedness of Stein degree.  

Now we explain how to prove $\mc{H}(d, t)$ and $\mc{V}(d,t)$ as in \eqref{path}.
Pick an arbitrary generalised log Calabi-Yau fibration $(X, B + \mb{M})\to Z$
satisfying all the assumptions in Theorem~\ref{dim-d-t}.  To prove 
$\mc{H}(d, t)$ for $(X/Z, B + \mb{M})$,
it is easy to see that we can assume the base variety $Z$ is a nonsingular curve; see Lemma~\ref{red-to-base-curve}. 
Moreover, decreasing $t$ if necessary, we can assume that $t\in \QQ^{>0}$. 

\medskip

%%%%%%%%%%%%%%%%%%%%%%%%%%%%%%%%%%%%%%%%%%%%%%%%

\emph{Step 1} $(\mc{H}(1,t) \Rightarrow \mc{V}(1,t))$.
The statement $\mc{H}(1, t)$ can be easily proved as a general fibre of $X\to Z$ 
is isomorphic to $\PP^1$; see Lemma~\ref{base-case-r-dim=1}.
For $\mc{V}(1, t)$, let $(X/Z, B + \mb{M})$ be a Fano type fibration 
of relative dimension one with a vertical$/Z$ prime divisor $S$ 
satisfying all the assumptions in Theorem~\ref{vertical-bnd-intro}.
By Bertini's theorems, we can assume that
$\dim X = 3$, $\dim Z = 2$ and $\dim T = 1$, where $T$ is the image of $S$ in $Z$.

By running a sequence of MMPs over $Z$ and by applying the boundedness of complements \cite[Theorem 1.8]{B-Fano},
the problem is reduced to consider an lc pair $(X', C')$ such that
\begin{itemize}
    \item [(1)] there is a rational map $X\dashrightarrow X'$ over $Z$ (which possibly extracts some divisors),
    \item [(2)] the centre $S'$ of $S$ on $X'$ is a horizontal$/T$ prime divisor,
    \item [(3)] there is a horizontal$/Z$ component $H'$ of $C'$ with coefficient $\ge u$ for some $u\in (0, 1]$ depending only on $t$, and
    \item [(4)] every irreducible component of a general fibre of $S'\to T$ intersects $H'$.
\end{itemize}
By $\mc{H}(1, t)$, the number of irreducible components
of a general fibre of $H'\to Z$ is bounded from above depending only on $t$.
Let $z\in Z$ be a general closed point of $T$.  Since $H'\to Z$ is generically finite,
and since $Z$ is normal, shrinking $Z$ near the generic point $\eta_T$ of $T$, 
we can assume that $H'\to Z$ is a finite morphism; see Lemma~\ref{finite-fibre}.
Thus, the number of irreducible components of the fibre $H'_z$ is $\le \sdeg \big( (H')^{\nor}/Z \big)$.

The reduction $\red (H_z')$ consists of finitely many closed points $x_1, \dots, x_r$, 
where $r$ is bounded from above by $\sdeg \big( (H')^{\nor}/Z \big)$.  
By (4), every irreducible component of $S_z'$ contains at least one of the $x_i$'s.
However, for a fixed $x_i$, there can not be too many irreducible components of $S_z'$
passing through $x_i$ as otherwise $(X', C')$ is not lc at $x_i$.
Indeed, by Lemma~\ref{surface-comp-bnd}, for every $x_i$, the number of irreducible components of $S_z'$
passing through $x_i$ is bounded from above depending only on $t$.  
Then we can conclude that $\mc{V}(1, t)$ holds. 

\medskip

%%%%%%%%%%%%%%%%%%%%%%%%%%%%%%%%%%%%%%%%%%%%%%%%%%

\emph{Step 2} $(\mc{V}(1,t)\Rightarrow \mc{V}(d, t))$.  
As $\mc{V}(1,t)$ holds, we assume that $\mc{V}(k, t)$ holds for every $1\le k\le d-1$ and then prove
$\mc{V}(d, t)$.  Part of the proof of this step relies on
the proof of the main result of \cite{birkar2024irrationality}.
From the technical point of view, the Fano type property of $X\to Z$
is essential in the proof as it allows us to run MMP$/Z$ on any $\RR$-divisor \cite{BCHM},
and it is preserved by MMP$/Z$ and by the inductive process; see \S\ref{fano-type-defn}.

Let $(X, B + \mb{M})\to Z$ be a Fano type (generalised log Calabi-Yau) fibration
of relative dimension $d$ that satisfies all the assumptions in Theorem~\ref{vertical-bnd-intro}.
By Bertini's theorems, we can assume that $\dim Z = 2$ and $\dim T = 1$,
where $T$ is the image of $S$ in $Z$.
Running MMP$/Z$ on $-S$ and applying boundedness of complements \cite[Theorem 1.8]{B-Fano},
we arrive at an lc pair $(X', C')$ such that
\begin{itemize}
    \item $(X', C')$ is $\QQ$-factorial and dlt,
    \item $X'$ is of Fano type over $Z$,
    \item $K_{X'} + C' \sim_{\QQ} 0/Z$,
    \item the coefficients of $C'$ belong to the set of rational numbers $\mb{T}_n := \Set{m/n}_{1\le m\le n}$ for some $n\in \NN$ depending only on $d,t$, 
    \item the centre $S'$ of $S$ on $X'$ supports on the whole fibre of $X'\to Z$ over $T$, and
    \item $\frac{t}{2} S' \le C'$.
\end{itemize}
By ACC of lc thresholds, we can assume that $(X', 0)$ is $\epsilon$-lc over $Z\setminus T$
for some $\epsilon > 0$ depending only on $d,n$ (whence only on $d,t$).
Run $\phi\colon X'\dashrightarrow X''$
a $K_{X'}$-MMP over $Z$ that ends with a Mori fibre space $g\colon X''\to Z''$ over $Z$.
Note that $S'$ is not contracted by $\phi$ as it supports on the whole fibre of $X'\to Z$ over $T$.
\[\xymatrix{
X\ar[d]\ar@{-->}[r] & X'\ar@{-->}[r]^{\phi} & X''\ar[d]^g \\
Z & & Z''\ar[ll]
}\]
The rest of the proof of $\mc{V}(d, t)$ is divided into two cases depending
on the relative dimension $\dim (Z''/Z)$ of the contraction $Z''\to Z$. 

\medskip

\emph{Case 1}.
If $\dim (Z''/Z) \ge 1$, then by adjunction for fibre spaces (see \cite[\S3.4]{B-Fano}), we can write
\[ K_{X''} + C'' \sim_{\QQ} g^*(K_{Z''} + C_{Z''} + \mb{N}_{Z''}), \]
where $C'' := \phi_* C'$ and $(Z'', C_{Z''} + \mb{N})$ is also a g-lc g-pair.
Note that both the contractions $X''\to Z''$ and $Z''\to Z$ are of Fano type; see \S\ref{fano-type-defn}.
Let $T''$ be the image of $S'' := \phi_* S'$ on $Z''$.
Then the coefficient of $T''$ in $C_{Z''}$ is also $\ge t/2$ by \cite[Lemma 3.7]{B-Fano}.
Applying $\mc{V}(k, t)$ for $1\le k\le d-1$ to the Fano type fibrations
\[ (X'', C'') \to Z'' \text{ and } (Z'', C_{Z''} + \mb{N}) \to Z, \]
we can conclude that $\ssdeg (S^{\nor}/Z) = \ssdeg \big( (S'')^{\nor}/Z \big)$
is bounded from above depending only on $d,t$; cf. Lemma~\ref{mul-g-fib}. 

\medskip

\emph{Case 2}.
If $\dim (Z''/Z) = 0$, we can assume that $Z''\to Z$ is an isomorphism by shrinking
$Z$ near the generic point $\eta_T$ of $T$.  In particular, $-K_{X''}$ is ample$/Z$,
hence the general fibres of $X''\to Z$ are bounded by \cite[Theorem 1.1]{B-BAB}.
By the construction in \cite{birkar2024irrationality}, there is a commutative diagram
\[\xymatrix{
(U_{Y'}\subset Y')\ar[d]^{\mu'}\ar[r]^-{m_Y} & Y\ar[d]^{\mu} & X''\ar@{-->}[l]_-{\varphi}\ar[d]^g \\
(U_Z\subset Z)\ar[r]^-{m_Z} & Z\ar@{=}[r] & Z
}\]
such that
\begin{itemize}
    \item $\varphi$ is a birational map that does not contract any horizontal$/Z$ divisors,
    \item $\mu\colon Y\to Z$ belongs to a relatively bounded family of varieties; see \S\ref{r-bnd-def},
    \item $m_Y$ is a birational projective morphism,
    \item $m_Z$ is an isomorphism (up to shrinking $Z$ near $\eta_T$),
    \item $\mu'\colon (U_{Y'}\subset Y')\to (U_Z\subset Z)$ is a toroidal morphism of strict toroidal embeddings; see \S\ref{toroidal-section},
    \item $(Y', D')\to Z$ is relatively bounded over an open subset of $Z$, where $D'$ is the reduced divisor that is the reduction of $Y'\setminus U_{Y'}$,
    \item there is an ample$/Z$, base point free, Cartier divisor $A$ on $Y'$ such that $\deg_{A/Z} A$ (see \S\ref{r-bnd-def}) is bounded from above depending only on $d,t$, and
    \item $K_{Y'} + D' + A$ is ample$/Z$.
\end{itemize}
Denote by $C$ the centre of $S$ on $Y'$.  Then $C$ is an lc centre of the strict toroidal couple $(Y', D')$.
Moreover, the general fibre of $S\dashrightarrow C$ is irreducible by \cite[Proposition 3.5]{birkar2024irrationality}.
Thus, it suffices to show that the number of irreducible components of a general fibre of $C\to T$ is bounded from above.
The proof of this fact consists of two parts.

(i). Pick a general closed point $z\in T$.
As $Z$ is nonsingular at $\eta_T$, we can assume that $\mu'$ is flat, hence
$\mu'\colon Y'\to Z$ is a well-defined family of algebraic cycles by \cite[Corollary I.3.15]{Kol-rc}.
Then the constancy of degree with respect to $A$ (see \cite[Proposition I.3.12]{Kol-rc})
shows that the number of irreducible components of the fibre $Y_z'$ is bounded from above 
depending only on $d,t$.

(ii). Let $V'$ be an arbitrary irreducible component of $C_z$,
and let $G'$ be an irreducible component of $Y_z'$ containing $V'$.
Let $G$ be the normalisation of $G'$.  Applying adjunction, we can write
\[ K_G + D_G \sim_{\QQ} (K_{Y'} + D' + A + Y_L')|_G, \]
where $Y_L' := L\times_Z Y'$ is the pullback of a general hyperplane section $L$ (containing $z\in T$) of $Z$ to $Y'$.
As $K_G + D_G$ is an ample Cartier divisor with bounded volume,
then \cite[Theorem 1.1]{HMX14} shows that $(G, D_G)$ belongs to a bounded family of pairs.
Furthermore, adjunction shows that $(G, D_G)$ admits an lc centre $V$ mapping onto $V'$.
Thus, the number of irreducible components of $C_z$ contained in $G'$ is bounded from above.

Now combining (i) and (ii), we see that the number of irreducible components of 
a general fibre of $C\to T$ is bounded from above depending only on $d,t$. 

\medskip

\emph{Step 3} $(\mc{V}(d,t)\Rightarrow \mc{H}(d, t))$.
Assuming $\mc{H}(k, t)$ holds for every $1\le k\le d-1$,
we show that $\mc{H}(d, t)$ holds.
Indeed, it suffices to show that $\mc{H}(d, t)$ holds for Fano type fibrations
$(X, B + \mb{M})\to Z$; see \S\ref{main-proof-S-d-t}.
Running MMP$/Z$ on $-S$, applying induction hypothesis, and taking the $n$-complement
by \cite[Theorem 1.8]{B-Fano} for some $n\in \NN$ depending only on $d,t$, we are reduced to consider
a pair $(X'/Z, C')$ such that
\begin{itemize}
    \item $(X', C')$ is a $\QQ$-factorial dlt pair,
    \item $X'$ is of Fano type over $Z$,
    \item $n(K_{X'} + C')\sim 0/Z$, and
    \item $t S' \le C'$,
\end{itemize}
where $S'$ is the centre of $S$ on $X'$.
Note that $S'$ is horizontal$/Z$.
Shrinking $Z$ to an open subset, we can assume that $(X', 0)$ is $\epsilon$-lc
for some $\epsilon>0$ depending only $d,n$ (hence depending only on $d,t$).
Running $\psi\colon X'\dashrightarrow X''$ a $K_{X'}$-MMP over $Z$, 
we get a Mori fibre space $f''\colon X''\to Z''$ over $Z$.
\[\xymatrix{
X\ar@{-->}[r]\ar[d] & X'\ar@{-->}[r]^{\psi} & X''\ar[d]^{f''} \\
Z &  & Z''\ar[ll]
}\]
The rest of the proof is divided into two cases depending on $\dim (Z''/Z)$. 

\medskip

\emph{Case 1}.
If $\dim (Z''/Z)\ge 1$, it is natural to apply induction to the contractions $X''\to Z''$
and $Z''\to Z$.  However, there are several possibilities for the centre $S''$ of $S'$ on $X''$.
Note that $S''$ may not be a divisor on $X''$.
Denote by $C''$ the pushdown of $C'$ to $X''$.
Replacing $(X'', C'')$ by a pair extracting only the valuation of $S$,
we can assume that $S''$ is a prime divisor on $X''$.
The Fano type property of $X''\to Z$ is preserved in this process by Lemma~\ref{fano-type-at-dlt-model}.

(i). If $S''$ is horizontal$/Z''$, we can apply induction and Lemma~\ref{mul-g-fib} to conclude 
the boundedness of $\sdeg (S^{\nor}/Z)$.

(ii). If $S''$ is vertical$/Z''$, we run $\upsilon\colon X''\dashrightarrow X'''$ a $(-S'')$-MMP over $Z''$
that ends with a good minimal model of $-S''$ (over $Z''$).  Let $S''' := \upsilon_* S''$, and
let $X'''\to Z'''$ be the contraction induced by $-S'''$.  Then
the image $T'''$ of $S'''$ is a prime divisor on $Z'''$.
Now we can apply $\mc{V}(k, t)$ to $X'''\to Z'''$ with $k = \dim (X'''/Z''')$
and apply $\mc{H}(\ell, t)$ to $Z'''\to Z$ with $\ell = \dim (Z'''/Z)$
to conclude the boundedness of $\sdeg (S^{\nor}/Z)$. 

\medskip

\emph{Case 2}.
If $\dim (Z''/Z) = 0$, then $Z''\to Z$ is an isomorphism of nonsingular curves.
At this stage, a general fibre of $X''\to Z$ is an $\epsilon$-lc Fano variety, so it belongs
to a bounded family of projective varieties by \cite[Theorem 1.1]{B-BAB}.
Again, note that the centre of $S'$ on $X''$ may not be a divisor.
However, up to shrinking $Z$ to an open subset,
there is a relatively bounded family of projective couples $(Y, D)\to Z$ such that
\begin{itemize}
    \item there is a birational contraction $Y\dashrightarrow X$ over $Z$,
    \item $(Y, D)\to Z$ is log smooth,
    \item $\centre_Y S$ is an lc centre of $(Y, D)$, and
    \item the general fibre of $S\dashrightarrow \centre_Y S$ is irreducible.
\end{itemize}
Then by the properties listed above, we can conclude that $\sdeg(S^{\nor}/Z)$
is bounded from above depending only on $d,t$.

%%%%%%%%%%%%%%%%%%%%%%%%%%%%%%%%%%%%%%%%

\mni
{\textbf{\sffamily{Acknowledgements.}}}
The first author was supported by a grant from Tsinghua University and 
a grant of the National Program of Overseas High Level Talent. 
We thank Hexu Liu for reading carefully a draft version of this paper,
and for his many suggestions to improve the paper.
We would also like to thank Bingyi Chen and Xiaowei Jiang for helpful comments.

%%%%%%%%%%%%%%%%%%%%%%%%%%%%%%%%%%%%%%%%%%%%%%%%%%%%%
%%%%%%%%%%%%%%%%%%%%%%%%%%%%%%%%%%%%%%%%%%%%%%%%%%%%%

\section{Preliminaries}\label{preliminaries}

Recall that we work over an algebraically closed field $\mbb K$ of characteristic zero.
The set of natural numbers $\NN$ is the set of all positive integers, 
so it does not contain 0; cf. \cite[\S 2]{B-Fano}.  Denote by $\QQ^{>0}$, respectively, by $\RR^{>0}$,
the set of all positive rational numbers, respectively, the set of all positive real numbers.

By a \emph{$\mbb K$-scheme} (or simply, a \emph{scheme}), we mean a scheme that is of finite type and separated over $\mbb K$.
For a scheme $X$, denote by $\red X$ the maximal reduced closed subscheme of $X$, 
called the \emph{reduction} of $X$.
For an integral scheme $X$, we denote by $X^{\nor}$ the \emph{normalisation} of $X$.
A \emph{variety} is a quasi-projective, reduced, and irreducible $\mbb K$-scheme;
when emphasising the quasi-projectivity, we also call it a \emph{quasi-projective variety}.

Let $f\colon X\to Z$ be a morphism between schemes, where $Z$ is irreducible.
If $X$ has a unique irreducible component $Y$ that dominates $Z$, we call
$Y$ the \emph{main component of $X$}.
If this is the case, we usually just say that $Y$ is the main component of $X$
without mentioning that it is the unique irreducible component of $X$ dominating $Z$.

Let $X$ be a scheme.  A \emph{prime divisor} on $X$ is a reduced and irreducible closed subscheme of $X$
of codimension one.  By a \emph{divisor}, we mean a \emph{Weil divisor} on $X$, that is,
a finite $\ZZ$-linear combination of prime divisors on $X$, 
so a divisor is not necessarily Cartier.

%%%%%%%%%%%%%%%%%%%%%%%%%%%%%%%%%%%%%%%%%

\subsection{Contractions}
 
A \emph{contraction} is a projective morphism of schemes $f\colon X\to Y$
such that $f_*\mc{O}_X = \mc{O}_Y$; $f$ is not necessarily birational (see \cite[\S 2.1]{B-Fano}).
In particular, $f$ has connected fibres.  Moreover, if $X$ is normal, then $Y$ is also normal.

We say that a birational map $\phi\colon X\dashrightarrow Y$ of schemes is a 
\emph{birational contraction} if $\phi$ is proper
and $\phi^{-1}$ does not contract any divisors; see \cite[page 424]{BCHM}.

%%%%%%%%%%%%%%%%%%%%%%%%%%%%%%

\subsection{Divisors}\label{divisors-pre}

Let $X$ be a scheme.  
By a \emph{$\QQ$-divisor} (respectively, an \emph{$\RR$-divisor}), we mean a finite linear combination
$\sum_i b_i B_i$, where every $B_i$ is a prime divisor on $X$ and $b_i\in \QQ$ 
(respectively, $b_i\in \RR$). 
Let $D$ be an $\RR$-divisor on $X$.
For any prime divisor $S$ on $X$, we denote by $\coeff_S D$ the coefficient of $S$ in $D$. 
Moreover, we denote by $\supp D$ the support of $D$, which is a closed subset of $X$.
If there is no confusion in the context, we view $\supp D$ as the reduced divisor whose 
set of irreducible components is equal to that of $D$.

A $\QQ$-divisor (respectively, an $\RR$-divisor) is called \emph{$\QQ$-Cartier}
(respectively, \emph{$\RR$-Cartier})
if it is a $\QQ$-linear (respectively, an $\RR$-linear) combination of Cartier divisors.
Let $B_1$ and $B_2$ be two $\RR$-divisors on $X$.  We say that $B_1\sim B_2$ (respectively, 
$B_1\sim_{\QQ} B_2$, respectively, $B_1\sim_{\RR} B_2$) if $B_1 - B_2$ is a $\ZZ$-linear 
(respectively, a $\QQ$-linear, respectively, an $\RR$-linear) combination of principal divisors;
see \cite[\S 2.3]{B-Fano}.

Let $f\colon X\to Z$ be a morphism of schemes, and let $L$ and $M$ be $\RR$-divisors
on $X$.  We say that \emph{$L\sim M$ over $Z$} 
(respectively, \emph{$L\sim_{\QQ} M$ over $Z$}, respectively, \emph{$L\sim_{\RR} M$ over $Z$})
if there is a Cartier (respectively, a $\QQ$-Cartier, respectively, an $\RR$-Cartier) divisor $N$ on $Z$ 
such that $L-M\sim f^*N$ (respectively, $L-M \sim_{\QQ} f^*N$, respectively, $L-M \sim_{\RR} f^*N$); cf. \cite[\S 2.3]{B-Fano}.

Let $f\colon X\to Z$ be a morphism of schemes, and let $D$ be a nonzero $\RR$-divisor on $X$.
We say that $D$ is \emph{vertical$/Z$} if $f(\supp D)$ does not contain
any generic point of $Z$.
If $D$ does not have any vertical$/Z$ irreducible components,
we say that $D$ is \emph{horizontal$/Z$}.

%%%%%%%%%%%%%%%%%%%%%%%%%%%%%%%%%%%%%%%%%

We take the definition of (\emph{relatively}) \emph{big $\RR$-divisors} from \cite[\S3.1]{BCHM}. 
The following result can be derived immediately from
the definition of (relatively) big $\RR$-divisors.

\begin{lemma}\label{no-van-big-divisor}
	Let $\phi\colon X\dasharrow Y/Z$ be a birational contraction of
	normal varieties that are projective over $Z$, and let $D$ be a big$/Z$ $\RR$-Cartier $\RR$-divisor
	on $X$.  Then $\phi_*D$ is also a big$/Z$ $\RR$-divisor on $Y$.
	In particular, if $D$ is an $\RR$-Cartier big$/Z$ prime divisor on $X$, then $\phi_* D \not=0$.
\end{lemma}

%%%%%%%%%%%%%%%%%%%%%%%%%%%%%%%%%%%%%%%%%

\subsection{Irreducible components of fibres}\label{pre-stein-deg-divisors}

Now we introduce some notions for the number of irreducible components of schemes
to simplify the statements in this paper.

%%%%%%%%%%%%%%%%%%%%%%%%%%%%%%%%%%%%%%%%%

\begin{definitionnotation}
	Let $X$ be a $\mbb K$-scheme.  Denote by $\mf{n}(X)$ 
	the number of irreducible components of $X$.  
	Let $X\to Y$ be a dominant morphism of $\mbb K$-schemes with $Y$ irreducible.
    Denote by $\dim (X/Y)$ the dimension of irreducible components of a general fibre of $X\to Y$.
    Note that $\dim (X/Y)$ is well-defined by generic flatness;
    see \cite[Th\'eor\`eme (6.9.1)]{EGA-IV-II} and \cite[Corollary III.9.6]{Hart}.
    In particular, if $X\to Y$ is a dominant morphism of varieties,
    every irreducible component of a general fibre of $X\to Y$ has dimension $\dim X - \dim Y$.
	Furthermore, denote by $\mf{n}(X/Y)$ the number of irreducible components of a general fibre of $X\to Y$,
    which is well-defined by \cite[Proposition (9.7.8)]{EGA-IV-3}.
\end{definitionnotation}

%%%%%%%%%%%%%%%%%%%%%%%%%%%%%%%%%%%%%%%%

\begin{lemma}\label{num-g-fibre}
    Let $f\colon X\to Y$ be a dominant morphism of varieties.  Let $K(X)$ and $K(Y)$
    be the function fields of $X$ and $Y$ respectively.
    Then for an arbitrary nonempty open subset $U\subseteq X$, we have
    \begin{equation}
    \mf{n} (X/Y) = \mf{n}(U/Y) = \mf{n} \big( \spec \big( K(X)\otimes_{K(Y)} \ol{K(Y)} \big) \big). \label{dim-equa}
    \end{equation}
    In particular, $K(X)\otimes_{K(Y)} \ol{K(Y)}$ is the direct sum of $\mf{n}(X/Y)$ fields.
\end{lemma}

\begin{proof}
    Let $T := X\setminus U$ be the closed subscheme of $X$
    equipped with the reduced scheme structure.
    By generic flatness, shrinking $Y$ around its generic point, 
    we can assume that the morphisms $X\to Y$, $U\to Y$,
    and $T\to Y$ are all flat.  In particular, $\dim (T/Y)$
    is strictly less than $\dim (X/Y) = \dim (U/Y)$, hence
    for a general closed point $y\in Y$, the fibres $X_y$ and $U_y$
    have the same set of generic points.
    Thus, we have $\mf{n}(X/Y) = \mf{n}(U/Y)$.
    By \cite[Chap. I, Proposition (3.4.9)]{EGA-I}, the generic points of $X\otimes_{K(Y)}\ol{K(Y)}$
    are the points of $\spec \big( K(X)\otimes_{K(Y)}\ol{K(Y)} \big)$.
    Moreover, as $\chara \mbb K = 0$, $K(X)$ is separable over $K(Y)$; see \cite[\S 26]{CA-Mat}.
    Then the scheme $K(X)\otimes_{K(Y)}\ol{K(Y)}$ is reduced and has dimension zero, hence
    it is a direct sum of $\mf{n}(X/Y)$ fields; cf. \cite[\S 8, Proposition 3]{Bourbaki-alg-ch8}.
    Thus, we can conclude that the second equality of \eqref{dim-equa} holds.
\end{proof}

%%%%%%%%%%%%%%%%%%%%%%%%%%%%%%%%%%%%%%%%

\begin{definition}\label{defn-rational-num-g-fib}
    Let $f\colon X\dashrightarrow Y$ be a dominant rational map of varieties.  
    Let $U\subseteq X$ be an arbitrary open subset of $X$ on which 
    $f$ is well-defined.  Define 
    \[ \mf{n}(X/Y) := \mf{n}(U/Y). \]
    Note that $\mf{n}(X/Y)$ does not depend on the choice of $U$ by Lemma~\ref{num-g-fibre}.
\end{definition}

%%%%%%%%%%%%%%%%%%%%%%%%%%%%%%%%%%%%%%%

\begin{remark}\label{bir-map-irr-comp}
    Let $f\colon S\dashrightarrow S'/Z$ be a birational map of varieties 
    equipped with dominant morphisms to a variety $Z$.
    As $K(S)\cong K(S')$, Lemma~\ref{num-g-fibre} gives the equality
    \begin{equation}
        \mf{n}(S/Z) = \mf{n}(S'/Z). \label{bir-dim-equa}
    \end{equation}
    Let $X\to Z$ be a contraction of normal varieties, and let $S$ be a horizontal$/Z$ prime divisor on $X$.
    Then \eqref{bir-dim-equa} enables us to compute $\mf{n}(S/Z)$
    after running an MMP $\phi\colon X\dashrightarrow X'$ over $Z$ on some $\RR$-divisors (see \S\ref{pre-MMP})
    such that $S$ is not contracted by $\phi$. 
\end{remark}

%%%%%%%%%%%%%%%%%%%%%%%%%%%%%%%%%%%%%%%%

\begin{lemma}\label{mul-g-fib}
    Let $f\colon X\dashrightarrow Y$ and $g\colon Y\dashrightarrow Z$
    be dominant rational maps of varieties.  Then
    \[ \mf{n}(X/Z) \le \mf{n}(X/Y)\cdot \mf{n}(Y/Z). \]
    Moreover, if $\mf{n}(X/Y) = 1$, i.e., a general fibre of $f$ is irreducible, then $\mf{n}(X/Z) = \mf{n}(Y/Z)$.
\end{lemma}

\begin{proof}
    This follows from Lemma~\ref{num-g-fibre}.
    Here we present a more geometric proof.
    Taking open subsets on which the rational maps are well-defined, we can assume that
    $f,g$ are dominant morphisms of varieties.
    Let $F, G$ be fibres of $g\circ f, g$ over a common general closed point of $Z$.
    Then $G$ has $\mf{n}(Y/Z)$ irreducible components.
    By \cite[Th\'eor\`eme (1.8.4)]{EGA-I}, the image $f(X)$ is a constructible subset containing
    the generic point of $Y$, hence it contains an open subset $U$ of $Y$.
    Shrinking $U$ if necessary, we can assume that the fibres of $f$ over $U$ are all equi-dimensional 
    and have the same number of geometric irreducible components.
    By generic flatness, we can assume that $G\cap U$ contains all the generic points of $G$.
    Then $F\to G$ is a dominant morphism of reduced schemes.
    Let $G'$ be an irreducible component of $G$.
    It is evident that the number of irreducible components of $F$ mapped into $G'$ 
    is $\le \mf{n}(X/Y)$.  Thus, we can conclude that $\mf{n}(F)$ is less than
    or equal to $\mf{n}(X/Y)\cdot \mf{n}(Y/Z)$.
    
    Moreover, if $\mf{n}(X/Y) = 1$, then for every irreducible component $G'$ of $G$,
    there is a unique irreducible component of $F$ dominating $G'$,
    hence we have $\mf{n}(X/Z) = \mf{n}(Y/Z)$ as desired.
\end{proof}

%%%%%%%%%%%%%%%%%%%%%%%%%%%%%%%%%%%%%%%%

\subsection{Generically finite morphisms}

Let $f\colon X\to Y$ be a morphism of schemes.
Denote by $X^0$, respectively, by $Y^0$, the set of generic points of
irreducible components of $X$, respectively, of $Y$.
We say that $f$ is \emph{generically finite} if $f^{-1}(Y^0) = X^0$.
If $f\colon X\to Y$ is a generically finite morphism of schemes,
there exists a dense open subset $V\subset Y$ such that $f^{-1}(V)\to V$ is finite;
we denote the degree of $f^{-1}(V)\to V$ by $\deg (X/Y)$.  
If $f\colon X\to Y$ is a generically finite morphism of integral schemes,
then $\deg (X/Y)$ is equal to the degree of the fields extension $K(Y)\to K(X)$.

%%%%%%%%%%%%%%%%%%%%%%%%%%%%%%%%%%%%

\begin{lemma}\label{finite-fibre}
	Let $\pi\colon X\to Y$ be a proper, generically finite morphism of varieties.
	Assume that $Y$ is normal.  Then there exists an open subset $U\subseteq Y$
	satisfying 
    \[ \codim (Y\setminus U, Y) \ge 2 \] 
    such that $\dim X_y = 0$ for every scheme-theoretic point $y\in U$.
	Moreover, 
	\[ \mf{n}(X_{\ol{y}}) \le \mf{n}(X/Y) \]
	for every scheme-theoretic point $y\in U$, where $X_{\ol{y}}$ is the geometric fibre of $X\to Y$ over $y$.
\end{lemma}

\begin{proof}
	As $Y$ is normal, every codimension one point $\xi\in Y$ is regular, so $\mc{O}_{Y, \xi}$
	is a DVR.  Then $\pi$ is flat over $\xi$ since $X$ is integral.
	Thus, there exists an open subset $U\subseteq Y$ whose complement $Y\setminus U$
	has codimension at least two such that $\pi$ is flat over $U$.
    By flatness, $\dim X_y = 0$ for every scheme-theoretic $y\in U$ as $\pi\colon X\to Y$ is generically finite.
    In particular, $\pi^{-1}(U) \to U$ is a finite morphism
    as it is proper and quasi-finite.  Since every finite flat module over a local ring is free,
	the remained result follows immediately.
\end{proof}

%%%%%%%%%%%%%%%%%%%%%%%%%%%%%%%%%%%%%

\subsection{Stein degree of normal varieties}

Note that Stein degree is not a birational invariant; see \cite[Example 3.1]{B-moduli}.
However, by taking Stein factorisations, it is easy to see 
that Stein degree of normal varieties is birationally invariant.

\begin{lemma}\label{num-irr-comp}
	Let $X\to Y$ be a surjective and projective morphism of varieties.  Then
    \begin{equation}
        \sdeg (X/Y) \le \mf{n}(X/Y) = \mf{n}(X^{\nor}/Y) = \sdeg (X^{\nor}/Y),\label{normal-sdeg}
    \end{equation}
	where $X^{\nor}$ is the normalisation of $X$.
	
	Moreover, let $S\dasharrow T/Z$ be a birational map of normal quasi-projective varieties
	that are surjective and projective over a variety $Z$.  Then $\sdeg (S/Z) = \sdeg (T/Z)$.
\end{lemma}

\begin{proof}
    Let $X^{\nor}\to V\to Y$ be the Stein factorisation of $X^{\nor}\to Y$.
	A general fibre of $X^{\nor}\to V$ is normal and connected, hence irreducible.
	Thus, we have $\mf{n}(X^{\nor}/Y) = \sdeg (X^{\nor}/Y)$ by Lemma~\ref{mul-g-fib}.
    On the other hand, as $X$ is a variety, the normalisation morphism $X^{\nor}\to X$
    is an isomorphism over a dense open subset of $X$.
	Let $F$ be a general fibre of $X\to Y$, and let $G$ be a general fibre of $X^{\nor}\to Y$.
	By generic flatness, we can assume that $F$ and $G$ are equi-dimensional.
    We can also assume that
	the induced finite morphism $G\to F$ is an isomorphism over a dense open subset of $F$
    containing all the generic points of $F$; cf. the proof of Lemma~\ref{mul-g-fib}.
	In particular, $G\to F$ induces a bijection on generic points of $G$ and $F$, 
	so $\mf{n} (X/Y) = \mf{n}(X^{\nor}/Y)$.
	
	Now let $S\dasharrow T/Z$ be a birational map of normal varieties
	that are surjective and projective over a variety $Z$.
    Then Remark~\ref{bir-map-irr-comp} shows that general fibres of 
	$S\to Z$ and $T\to Z$ have the same number of irreducible components,
	which is the Stein degree of both $S\to Z$ and $T\to Z$ by~\eqref{normal-sdeg}.
\end{proof}

%%%%%%%%%%%%%%%%%%%%%%%%%%%%%%%%%%%%
%%%%%%%%%%%%%%%%%%%%%%%%%%%%%%%%%%%%

\subsection{Pairs and singularities}
We will use standard notions from the Minimal Model Program \cite{km98, BCHM}.
Here we collect some of the most fundamental definitions for clarification.
A \emph{pair} $(X, B)$ consists of a normal quasi-projective variety $X$ and an
$\RR$-divisor $B\ge 0$ such that $K_X + B$ is $\RR$-Cartier; in this case, we call $B$
the \emph{boundary divisor} (or just \emph{boundary}, for simplicity) of $(X,B)$.

For \emph{log resolutions}, see \cite[Definition 1.12]{Kol_singularities_of_MMP}.
Let $W\to X$ be a log resolution of a pair $(X,B)$, and let 
$K_W + B_W$ be the pullback of $K_X + B$.
Denote by $\coeff_D B_W$ the coefficient of $B_W$ at a prime divisor $D$ on $W$,
then the \emph{log discrepancy} of $D$ with respect to $(X, B)$
is defined as
\[ a(D, X, B) := 1-\coeff_D B_W. \]
We say that a pair $(X, B)$ is \emph{lc} (respectively, \emph{klt}, respectively, \emph{$\epsilon$-lc})
if $a(D, X, B)$ is $\ge 0$ (respectively, $>0$, respectively, $\ge \epsilon$) 
for every divisor $D$ on an arbitrary log resolution $W\to X$ of $(X, B)$.
Let $(X, B)$ be an lc pair.  An \emph{lc place} of $(X, B)$ is a prime divisor $D$ on 
some birational model of $X$ such that $a(D, X, B) = 0$.
An \emph{lc centre} is the closure of the image of an lc place in $X$.

For \emph{dlt pairs}, we refer the readers to \cite[\S 2.3]{km98}. 
For \emph{dlt models}, see \cite{KK10}; in particular, see \cite[Theorem 3.1]{KK10} 
for the existence of $\QQ$-factorial dlt models.

%%%%%%%%%%%%%%%%%%%%%%%%%%%%%%%%%%%%%%%%%%%%

\subsection{$\text{B}$-divisors and generalised pairs}
We take the conventions of $\bdiv$-divisors and generalised pairs from
\cite[\S 2.3]{Corti-3-folds}, \cite{B-Zhang} and \cite[\S2.7, \S2.13]{B-Fano}.

Let $X$ be a variety.
A \emph{$\bdiv$-$\RR$-Cartier $\bdiv$-divisor over $X$} is the choice of a projective birational morphism
$Y\to X$ from a normal variety and an $\RR$-Cartier divisor $M$ on $Y$ up to the following equivalence:
another projective birational morphism $Y'\to X$ from a normal variety and an $\RR$-Cartier divisor $M'$
defines the same $\bdiv$-$\RR$-Cartier $\bdiv$-divisor if there is a common resolution $W\to Y$
and $W\to Y'$ on which the pullbacks of $M$ and $M'$ coincide.
Given an equivalence class of some $Y\to X$ and $M$ that represents a $\bdiv$-$\RR$-Cartier $\bdiv$-divisor over $X$,
we denote by $\mb{M}$ the equivalence class. 
We also say that $\mb{M}$ is a $\bdiv$-$\RR$-Cartier $\bdiv$-divisor over $X$; cf. \cite[2.3.8]{Corti-3-folds}.

Let $X$ be a variety, and let $\mb{M}$ be a $\bdiv$-$\RR$-Cartier $\bdiv$-divisor over $X$.
If $\mb{M}$ is represented by $Y\to X$ and an $\RR$-Cartier divisor $M$ on $Y$,
we say that \emph{$\mb{M}$ descends to $Y$}.  Moreover,
let $X'\to X$ be an arbitrary projective birational morphism from a normal variety $X'$.
Let $\pi\colon Y\to X'$ be a sufficiently high resolution so that $\mb{M}$ descends to $Y$, that is,
$\mb{M}$ is represented by $Y\to X$ and an $\RR$-Cartier divisor $M$ on $Y$.
The \emph{trace} of $\mb{M}$ on $X'$ is the $\RR$-divisor (cf. \cite[2.3.10]{Corti-3-folds})
\[ \mb{M}_{X'} := \tr_{X'} \mb{M} := \pi_* M.  \]
Note that $\mb{M}_{X'}$ is not necessarily $\RR$-Cartier.
On the other hand, let $X\dashrightarrow X''$ be an arbitrary birational map.
Let $Y\to X''$ and $Y\to X$ be a sufficiently high resolution of $X\dashrightarrow X''$
such that the $\bdiv$-$\RR$-Cartier $\bdiv$-divisor $\mb{M}$ over $X$ is represented by an $\RR$-Cartier
divisor $M$ on $Y$.  Then the equivalence class of $Y\to X''$ and $M$
also defines a $\bdiv$-$\RR$-Cartier $\bdiv$-divisor over $X''$,
which we also denote by $\mb{M}$.  
This should not lead to confusion in the context as a $\bdiv$-$\RR$-Cartier $\bdiv$-divisor
over $X$ can be viewed as an $\RR$-valued function on the set of all geometric valuations
of the function field $K(X)$, having finite support on some (hence any) birational model of $X$;
see \cite{Sh-3-fold-log-models} and \cite[\S 1.2]{Ambro04-Sh-boundary}.

A \emph{generalised pair} (or a \emph{g-pair} for short) 
$(X/Z, B + \mb{M})$ consists of a normal variety $X$ equipped with a projective morphism $X\to Z$
to a variety $Z$, an $\RR$-divisor $B \ge 0$ on $X$, and a $\bdiv$-$\RR$-Cartier $\bdiv$-divisor $\mb{M}$
over $X$ such that $\mb{M}$ descends to a nef$/Z$ $\RR$-divisor on some normal birational model of $X$
and that $K_X + B + \mb{M}_X$ is $\RR$-Cartier.  
When $Z$ is not relevant or $Z$ is $\spec \mbb K$, we usually drop it and do not mention it.

Let $(X, B + \mb{M})$ be a generalised pair.  Let $\phi\colon Y\to X$ be a sufficiently
high log resolution of $(X, B)$ so that $\mb{M}$ descends to $Y$.  We can write
\[ K_Y + B_Y + \mb{M}_Y = \phi^*(K_X + B + \mb{M}_X) \]
for some uniquely determined $\RR$-divisor $B_Y$ on $Y$.  For a prime divisor $D$ on $Y$,
the \emph{generalised log discrepancy} $a(D, X, B + \mb{M})$ is defined to be
\[ a(D, X, B + \mb{M}) := 1 - \coeff_D B_Y. \]
We say that $(X, B + \mb{M})$ is \emph{generalised lc} (respectively, \emph{generalised klt}, respectively,
\emph{generalised $\epsilon$-lc}) if for each prime divisor $D$ over $X$, the 
generalised log discrepancy $a(D, X, B + \mb{M})$ is $\ge 0$ (respectively, $>0$, respectively, $\ge \epsilon$).
For simplicity, we also say that $(X, B + \mb{M})$ is \emph{g-lc}
(respectively, \emph{g-klt}, respectively, \emph{g-$\epsilon$-lc}).

We adopt the notion of \emph{generalised dlt generalised pairs} (or \emph{g-dlt g-pairs} for short)
from \cite[\S 2.13 (2)]{B-Fano}.
In particular, for any g-lc g-pair $(X, B + \mb{M})$, \cite[Lemma 4.5]{B-Zhang} shows the existence of 
a projective birational morphism $\phi\colon Y\to X$ from a normal variety $Y$ such that
\begin{itemize}
    \item $Y$ is $\QQ$-factorial,
    \item $K_Y + B_Y + \mb{M}_Y = \phi^*(K_X + B + \mb{M}_X)$,
    \item each exceptional divisor of $\phi$ is a component of $\floor{B_Y}$, and
    \item $(Y, B_Y + \mb{M})$ is a g-dlt g-pair.
\end{itemize}
We call $(Y, B_Y + \mb{M})$ a \emph{$\QQ$-factorial g-dlt model} of $(X, B + \mb{M})$.
This is consistent with the notion of $\QQ$-factorial dlt model in \cite{KK10} when $\mb{M} = 0$.

%%%%%%%%%%%%%%%%%%%%%%%%%%%%%%%%%

\subsection{Generalised log Calabi-Yau fibrations}\label{g-log-CY-fib-preliminaries}

Recall that a \emph{generalised log Calabi-Yau fibration}
$(X, B + \mb{M})\to Z$ consists of a g-lc pair $(X, B + \mb{M})$ and a contraction $X\to Z$ of
quasi-projective varieties such that $K_X + B + \mb{M}_X \sim_{\RR} 0/Z$.

\begin{lemma}\label{CY-log-discre-no-change}
	Let $(X, B + \mb{M})\to Z$ be a generalised log Calabi-Yau fibration, and let $\phi\colon X\dashrightarrow Y$
	be a birational contraction$/Z$.  Denote by $B_Y$ the pushdown of $B$ to $Y$.
	Then for any divisor $E$ over $Y$, we have
	\[ a(E, X, B + \mb{M}) = a(E, Y, B_Y + \mb{M}). \]
	In particular, for any real number $\epsilon\ge 0$, if $(X, B + \mb{M})$ is g-$\epsilon$-lc,
	then $(Y, B_Y + \mb{M})$ is also g-$\epsilon$-lc.
\end{lemma}

\begin{proof}
	Let $p\colon W\to X$ and $q\colon W\to Y$ be a sufficiently high resolution of the birational contraction $\phi$
    so that $\mb{M}$ descends to a nef $\RR$-Cartier divisor on $W$.
    Note that $K_Y + B_Y + \mb{M}_Y$ is also $\RR$-Cartier since $K_X + B + \mb{M}_X\sim_{\RR}0/Z$.
	Write
	\[ L := p^*(K_X + B + \mb{M}_X) - q^*(K_Y + B_Y + \mb{M}_Y), \]
	which is $q$-exceptional as $\phi$ is a birational contraction.  Moreover, since
	\[ K_X + B + \mb{M}_X \sim_{\RR}0/Z \text{ and }K_Y + B_Y + \mb{M}_Y \sim_{\RR} 0/Z, \]
	the $\RR$-Cartier $\RR$-divisor $L$ is numerically trivial over $Z$, hence
	$L$ is also numerically trivial over $Y$.
	By negativity lemma, $L=0$.
\end{proof}

%%%%%%%%%%%%%%%%%%%%%%%%%%%%%%%%%%%

\subsection{Fano type varieties}\label{fano-type-defn}

Assume that $X$ is a variety and that $X\to Z$ is a contraction.
We say that $X$ is \emph{of Fano type over $Z$}
if there is a boundary $B$ such that $(X, B)$ is klt and
that $-(K_X + B)$ is ample over $Z$; cf. \cite[\S 2.10]{B-Fano}.
This is equivalent to the existence of a \emph{big}$/Z$ $\QQ$-divisor $\Gamma$ such that $(X, \Gamma)$
is klt and $K_X + \Gamma \sim_{\QQ} 0/Z$.
This is also equivalent to having a \emph{big}$/Z$ $\RR$-divisor $B$ such that
$(X, B + \mb{M})$ is a g-klt g-pair and $K_X + B + \mb{M}_X \sim_{\RR} 0/Z$;
see the proof of \cite[\S 2.13 (6)]{B-Fano}.
We call a generalised log Calabi-Yau fibration $(X, B+\mb{M})\to Z$ 
a \emph{Fano type generalised log Calabi-Yau fibration}
(or just a \emph{Fano type fibration} for short)
if $X$ is of Fano type over $Z$ in addition.

In this subsection, we collect some basic properties of Fano type varieties
that are used repeatedly throughout this paper.

\begin{lemma}\label{Fano-type-pushdown}
	Let $X\to Z$ be a contraction of varieties such that $X$ is of Fano type over $Z$,
	and let $\phi\colon X\dashrightarrow Y$ be a birational contraction.
	Then $Y$ is also of Fano type over $Z$.
\end{lemma}

\begin{proof}
    This is the relative version of \cite[Lemma 2.4]{birkar-sing-base-fano}.
	As $X$ is of Fano type over $Z$, there exists a big$/Z$ $\RR$-boundary $\Gamma$
	so that $(X, \Gamma)$ is klt and $K_X + \Gamma\sim_{\RR}0/Z$.
	Then the result follows immediately from Lemma~\ref{no-van-big-divisor} and Lemma~\ref{CY-log-discre-no-change}.
\end{proof}

%%%%%%%%%%%%%%%%%%%%%%%%%%%%%%%%%%%%%%%%

\begin{lemma}\label{fano-type-at-dlt-model}
    Let $f\colon X\to Z$ be a contraction of varieties such that $X$ is of Fano type over $Z$.
    Let $(X, B + \mb{M})$ be a g-lc g-pair such that $-(K_X + B + \mb{M}_X)$ is nef$/Z$.
    Let $\pi\colon X'\to X$ be a birational contraction of normal varieties.  Write
    \[ K_{X'} + B' + \mb{M}_{X'} = \pi^*(K_X + B + \mb{M}_X). \]
    If every exceptional$/X$ irreducible component of $B'$ has positive coefficient, 
    then $X'$ is also of Fano type over $Z$.
\end{lemma}

\begin{proof}
    This is the relative version of \cite[\S 2.13 (7)]{B-Fano} to which 
    we refer the readers for the proof of this result.
\end{proof}

%%%%%%%%%%%%%%%%%%%%%%%%%%%%%%%%%%%%%%%%%

\begin{lemma}\label{fano-type-upper}
    Let $X\to Y$ and $Y\to Z$ be contractions of normal varieties.
    If $X$ is of Fano type over $Z$, then $X$ is also of Fano type over $Y$.
\end{lemma}

\begin{proof}
    As $X$ is of Fano type over $Z$, there is a big$/Z$ $\QQ$-divisor $\Gamma$ such that
    $(X, \Gamma)$ is klt and $K_X + \Gamma\sim_{\QQ}0/Z$.
    By \cite[Remark, page 69]{nakayama_Zariski_decomposition},
    $\Gamma$ is also big$/Y$, hence $X$ is also of Fano type over $Y$ as $K_X + \Gamma \sim_{\QQ}0/Y$.
\end{proof}

%%%%%%%%%%%%%%%%%%%%%%%%%%%%%%%%%%%%%%%%%

The following lemma is well-known to experts, which is the relative version of \cite[Lemma 2.12]{B-Fano}.
For the lack of references in literature,
we sketch the proof here to indicate the necessary modifications on the original proof 
of \cite[Lemma 2.12]{B-Fano}.  The difference is that for a contraction $X\to Z$ of \emph{quasi-projective}
varieties, the result \cite[Theorem 0.2]{Ambro-lc-fibration} can not be applied directly.

\begin{lemma}[cf. \protect{\cite[Lemma 2.12]{B-Fano}}]\label{fano-type-on-base}
    Let $f\colon X\to Y$ be a contraction$/Z$ of quasi-projective varieties $X, Y$
    that are projective over a variety $Z$,
    where $\dim (Y/Z)>0$.  Assume that $X$ is of Fano type over $Z$.
    Then $Y$ is also of Fano type over $Z$.
\end{lemma}

\begin{proof}
    By Lemma~\ref{fano-type-at-dlt-model}, we can assume that 
    $X$ is $\QQ$-factorial.  As $X$ is of Fano type over $Z$,
    there is a big$/Z$ $\QQ$-divisor $\Gamma$ such that $(X, \Gamma)$ is klt
    and $K_X + \Gamma \sim_{\QQ} 0/Z$.  Since $\Gamma$
    is big$/Z$, we can assume that there is an ample$/Z$ $\QQ$-divisor $H\ge 0$ on $X$ 
    and an ample$/Z$ $\QQ$-divisor $A\ge 0$ on $Y$ such that
    $\Gamma \ge H \ge f^* A \ge 0$ and that $\Gamma - f^* A$ is big$/Z$, hence big$/Y$.
    Set $\Delta := \Gamma - f^* A$.
    By \cite[Lemma 3.3]{birkar-sing-base-fano}, there exist normal projective varieties $X\subseteq X'$ 
    and $Y\subseteq Y'$ and a contraction $f'\colon X'\to Y'$ such that 
    \begin{itemize}
        \item $(X', \Delta')$ is $\QQ$-factorial klt and $K_{X'} + \Delta' \sim_{\QQ} 0/Y'$,
        \item $(K_{X'} + \Delta')|_X = K_X + \Delta$, and $f'|_X = f$.
    \end{itemize}
    By \cite[Theorem 0.2]{Ambro-lc-fibration}, there is a $\QQ$-divisor $\Delta_{Y'}$ 
    such that $K_{X'} + \Delta' \sim_{\QQ} (f')^*(K_{Y'} + \Delta_{Y'})$ and that $(Y', \Delta_{Y'})$ is klt.
    Denote by $\Delta_Y$ the restriction of $\Delta_{Y'}$ to $Y$.
    Then we have $K_X + \Delta \sim_{\QQ} f^*(K_Y + \Delta_Y)$ and $(Y, \Delta_Y)$ is klt.
    Let $\Gamma_Y := \Delta_Y + A'$, where $0\le A'\sim_{\QQ} A$ is general, then
    $(Y, \Gamma_Y)$ is also klt and $\Gamma_Y$ is big$/Z$.
    Since $K_Y + \Gamma_Y \sim_{\QQ}0/Z$, we see that $Y$ is also of Fano type over $Z$.
\end{proof}

%%%%%%%%%%%%%%%%%%%%%%%%%%%%%%%%%%%%%%%%%

The following result is a direct consequence of the (generalised) adjunction formula,
we state it specifically as a lemma since it is used repetitively throughout this paper.

\begin{lemma}\label{fano-type-induction}
    Let $(X, B + \mb{M})\to Z$ be a Fano type generalised log Calabi-Yau fibration.
    Let $L$ be a general Cartier hyperplane section of $Z$.  
    Denote by $X_L$ the fibre product $L\times_Z X$.  Then we can write
    \begin{equation}
        K_{X_L} + B_L + \mb{N}_{X_L} \sim_{\RR} (K_X + B + X_L + \mb{M}_X)|_{X_L}, \label{fano-type-adjunction}
    \end{equation}
    where $(X_L, B_L + \mb{N})\to L$ is also a Fano type generalised log Calabi-Yau fibration.
    Moreover, let $S$ be an irreducible component of $B$ whose image $T$ in $Z$ satisfies that
    \begin{itemize}
        \item $\dim T = \dim Z$ if $\dim Z \ge 2$, or
        \item $\dim T = \dim Z - 1$ if $\dim Z \ge 3$.
    \end{itemize}
    Assume that $\coeff_S B\ge t$ for some $t\in \RR^{>0}$,
    and set $S_L := L\times_Z S$.  Then $S_L$ is also an irreducible component of $B_L$
    satisfying $\coeff_{S_L} B_L \ge t$.
\end{lemma}

\begin{proof}
    By Bertini's theorem for normality (see \cite[Corollary 3.4.9]{normal-bertini}),
    both $L$ and $X_L$ are also normal varieties.
    Moreover, for a general $L$, $(X, B + X_L + \mb{M})$ is also g-lc, hence by generalised adjunction
    (see \cite[Definition 4.7]{B-Zhang}), we can write \eqref{fano-type-adjunction},
    where $(X_L, B_L + \mb{N}) \to L$ is also a generalised log Calabi-Yau fibration.
    
    On the other hand, as $X$ is of Fano type over $Z$, there is a big$/Z$ $\QQ$-divisor $\Gamma$
    such that $(X, \Gamma)$ is klt and $K_X + \Gamma \sim_{\QQ}0/Z$.  For a general $L$, we can write
    \[ K_{X_L} + \Gamma_L \sim_{\QQ} (K_X + \Gamma + X_L)|_{X_L}, \]
    where $(X_L, \Gamma_L)$ is also a klt pair (by \cite[Theorem 5.50]{km98}) 
    satisfying that $K_{X_L} + \Gamma_L\sim_{\QQ}0/L$.
    It is clear that $\Gamma_L$ is big$/L$ by \cite[Corollary II.5.17]{nakayama_Zariski_decomposition}.
    Thus, $X_L$ is also of Fano type over $L$.
    Now pick an irreducible component $S$ of $B$ satisfying all the assumptions.
    Then $L\times_Z S$ is also reduced and irreducible by \cite[Th\'eor\`eme 6.10]{jou-bertini}
    and \cite[Corollary 3.4.9]{normal-bertini}.  Finally, the coefficient of $S_L$ 
    in $B_L$ is $\ge t$ by \cite[Definition 4.7]{B-Zhang}.
\end{proof}

%%%%%%%%%%%%%%%%%%%%%%%%%%%%%%%%%%%%%%%%%

\subsection{Minimal models, Mori fibre spaces, and MMP}\label{pre-MMP}

Let $X\to Z$ be a projective morphism of normal varieties, and
let $D$ be an $\RR$-Cartier $\RR$-divisor on $X$.
Let $Y$ be a normal variety, projective over $Z$, and let
$\phi\colon X\dashrightarrow Y/Z$ be a birational contraction.
Assume that $D_Y := \phi_* D$ is also $\RR$-Cartier and that there is a common resolution
$g\colon W\to X$ and $h\colon W\to Y$ such that $E := g^* D - h^* D_Y$ is effective 
and exceptional$/Y$, and $\supp g_* E$ contains all the exceptional divisors of $\phi$.

Under the above assumptions, we call $Y$ a \emph{minimal model} of $D$ over $Z$ 
if $D_Y$ is nef$/Z$;  
we call $Y$ a \emph{good minimal model} of $D$ over $Z$ if $D_Y$ is semiample$/Z$.
On the other hand, we call $Y$ a \emph{Mori fibre space}
of $D$ over $Z$ if there is an extremal contraction $Y\to T/Z$ with
$-D_Y$ ample$/T$ and $\dim Y > \dim T$; cf. \cite[Definition 3.10.7]{BCHM}.

If one can run a \emph{Minimal Model Program} (MMP) on $D$ over $Z$ which terminates
with a model $Y$, then $Y$ is either a minimal model or a Mori fibre space of $D$ over $Z$.
If $X$ is a Mori dream space (see \cite[Definition 1.3.1]{BCHM}), then such an MMP
always exists by \cite{BCHM}.
In particular, if $X$ is of Fano type over $Z$, for any given $\RR$-Cartier $\RR$-divisor $D$,
one can run the $D$-MMP over $Z$, and this ends with either a Mori fibre space of $D$ if
$D$ is non-pseudo-effective$/Z$ by \cite[Corollary 1.3.3]{BCHM} or with a good minimal model of $D$ if
$D$ is pseudo-effective$/Z$ by \cite[Theorem 1.2, Lemma 3.9.3]{BCHM}.
We will freely use these results without specifying citations from \cite{BCHM}
in the rest of this paper.

%%%%%%%%%%%%%%%%%%%%%%%%%%%%%%%%%%%%%%%%%%

\subsection{Complements}\label{complements}

We define complements as in \cite{Sh-surface}; see also \cite[\S 2.18]{B-Fano}.
Let $(X, B)$ be a pair, and let $X\to Z$ be a contraction. 
Let $T=\floor {B}$ and $\Delta = B- T$,
and let $n\in \NN$ be a natural number.
An \emph{$n$-complement} of $K_X + B$ over a point $z\in Z$ is of the form $K_X + B^+$
such that over some neighbourhood of $z$ we have the following properties:
\begin{itemize}
	\item $(X, B^+)$ is lc,
	\item $n(K_X + B^+)\sim 0$, and
	\item $nB^+ \ge nT+ \floor{(n+1)\Delta}$.
\end{itemize}
In particular, $n B^+$ is an integral divisor.  
An $n$-complement is \emph{monotonic} if $B^+\ge B$.

%%%%%%%%%%%%%%%%%%%%%%%%%%%%%%%%%%%%%%%%

\subsection{Couples}\label{couples-defn}

A \emph{couple} $(X,D)$ consists of a quasi-projective variety $X$ and a reduced Weil divisor $D$ on $X$.
This is more general than the definition given in \cite[\S 2.19]{B-Fano} because we are not assuming $X$ to be normal 
nor projective.   We say that a couple $(X, D)$ is \emph{projective} if the underlying variety $X$ 
is projective over $\spec \mbb K$.
Also note that a couple $(X,D)$ is not necessarily a pair in the sense that we are not assuming 
$K_X+D$ to be $\RR$-Cartier. We often consider a couple $(X,D)$ equipped with a 
\emph{surjective} projective morphism to a variety $Z$
in which case we often denote the couple as $(X/Z,D)$ or $(X, D)\to Z$. 
We say that a couple $(X/Z,D)$ is \emph{flat} if both $X\to Z$ and $D\to Z$ are flat.

A \emph{morphism} $(Z,E)\to (V,C)$ between couples is a morphism $f\colon Z\to V$ of varieties such that 
$f^{-1}(C)\subseteq E$. 

%%%%%%%%%%%%%%%%%%%%%%%%%%%%%%%%%%%%%%%%

\subsection{Strata and log smooth morphisms}\label{log-smooth}

We follow the conventions as in \cite[\S 2.1]{HMX14}.
Let $(X, D)$ be a couple.  The \emph{strata} of $(X, D)$ are the irreducible components of the intersections
\[ D_I = \bigcap_{j\in I} D_j = D_{i_1}\cap \cdots \cap D_{i_r} \]
of irreducible components of $D$, where $I = \Set{ i_1, \dots, i_r}$ is a subset of the indices, including the empty
intersection $X = D_{\emptyset}$.  
Every irreducible component in the strata of $(X, D)$ is called a \emph{stratum} of $(X, D)$.
If $(X, B)$ is a pair, then the \emph{strata} of $(X, B)$ are the strata of the underlying couple $(X, D)$,
where $\supp D = \supp B$.

For \emph{simple normal crossing} (or \emph{snc} for short) for pairs or divisors, see \cite[\S 1.1]{Kol_singularities_of_MMP}.
If we are given a couple $(X, D)$ over a variety $T$, then we say that $(X, D)$ is \emph{log smooth over $T$}
if $(X, D)$ has simple normal crossings and the strata of $(X, D)$ are smooth over $T$.
If $(X, B)$ is a pair over a variety $U$, we say that $(X, B)$ is \emph{log smooth over $U$}
if the underlying couple $(X, D)$ is log smooth over $U$, where $\supp D = \supp B$.
We say that a couple (similarly, a pair) is \emph{log smooth} if it is log smooth over $\spec \mbb K$.

%%%%%%%%%%%%%%%%%%%%%%%%%%%%%%%%%%%%%%%%

\subsection{Bounded families of couples}\label{b-bnd-couples}

We say that a set $\mc{Q}$ of projective couples $(X, D)$ is \emph{bounded}
if there exist finitely many projective morphisms $V_i\to T_i$ 
of varieties and reduced divisors $C_i$ on $V_i$ such that, for each
$(X, D)\in \mc{Q}$, there exist an $i$, a closed point $t\in T_i$,
and an isomorphism of couples $\phi\colon (V_{i,t}, C_{i, t})\to (X, D)$,
where $V_{i, t}$ and $C_{i, t}$ are the fibres over $t$ of the morphisms $V_i\to T_i$ and $C_i\to T_i$ respectively.
In particular, if $D=0$ for every $(X, D)\in \mc{Q}$, we say that the family $\mc{Q}$ consisting of projective varieties
is \emph{bounded}.

Let $\mc{Q}$ be a bounded family of projective couples,
and let $(X, D)\in \mc{Q}$.
When there is no confusion in the context, we usually say that $(X, D)$
belongs to a bounded family of projective couples, or 
just $(X, D)$ is bounded.

%%%%%%%%%%%%%%%%%%%%%%%%%%%%%%%%%%%%%%%%%%%%%%%%%%%%

\subsection{Relatively bounded families of couples}\label{r-bnd-def}

Let $f\colon X\to Z$ be a \emph{surjective} projective morphism of varieties,
and let $A$ be a $\QQ$-Cartier divisor on $X$. 
For a Weil divisor $D$ on $X$, we define the \emph{relative degree} of $D$ over $Z$ with respect to $A$ as 
\[ \deg_{A/Z}D:=(D|_F)\cdot (A|_F)^{n-1}, \] 
where $F$ is a general fibre of $f$ and $n=\dim F$. 
It is clear that this is a generic degree, so the vertical$/Z$ irreducible components of $D$ do not contribute 
to the relative degree.  Note that $F$ may not be irreducible.

We define generically relatively bounded families of couples (respectively, varieties) as in \cite[\S 3]{birkar2023singularities}.
Let $\mathcal{P}$ be a family of couples $(X/Z, D)$. We say $\mathcal{P}$ is \emph{generically relatively bounded} if 
there exist natural numbers $d, r$ 
such that for each $(X/Z,D)\in \mathcal{P}$, we have the following:
$\dim X -\dim Z \le d$, and there is a very ample$/Z$ 
divisor $A$ on $X$ such that 
\[ \deg_{A/Z}A\le r \text{ and } \deg_{A/Z}D\le r. \]
If in addition all the $(X/Z, D)\in \mcP$ are flat, we say that $\mcP$ is \emph{relatively bounded}. 
When $D=0$ for every $(X/Z,D)\in \mathcal{P}$, we then refer to $\mathcal{P}$ 
as a generically relatively bounded (respectively, relatively bounded) family of varieties.

Let $\mcP$ be a generically relatively bounded family of couples. 
Assume that $Z$ is a normal variety for each $(X/Z, D)\in \mcP$.  
By \cite[Proposition III.9.7]{Hart}, up to replacing $Z$ with an open subset
whose complement in $Z$ has codimension $\ge 2$, we see that $(X/Z, D)$ is flat.
Then replacing $Z$ for each $(X/Z, D)\in \mcP$ in this way, we can assume that
$\mcP$ is relatively bounded.

As for bounded set of projective couples (see \cite[Lemma 2.21]{B-Fano}), 
there is a universal family of 
varieties and divisors for a relatively bounded family of couples 
(up to shrinking the base of every couple to an open subset containing all the codimension one points).

%%%%%%%%%%%%%%%%%%%%%%%%%%

\begin{lemma}[cf. \protect{\cite[Lemma 3.4]{birkar2023singularities}}]\label{bir-lem-on-Pn}
    Let $\mc{P}$ be a generically relatively bounded family of couples $(X/Z, D)$, where $Z$ 
    is a normal variety.  Then there is a natural number $n\in \NN$
    depending only on $\mc{P}$ such that for each $(X/Z, D)\in \mc{P}$ and each codimension one point $z\in Z$,
    perhaps after shrinking $Z$ around $z$, the morphism $X\to Z$ factors as a closed immersion
    $X\to \PP_Z^n$ followed by the canonical projection $\PP_Z^n\to Z$.
\end{lemma}

\begin{proof}
    Pick $(X/Z, D)\in \mc{P}$.
    As $\mc{P}$ is generically relatively bounded, there are fixed natural numbers $d, r$ such that 
    $\dim X - \dim Z \le d$, and we can find a very ample$/Z$ divisor $A$ on $X$ with 
    $\deg_{A/Z} A\le r$ and $\deg_{A/Z} D\le r$.  
    Let $R$ be the localisation of $Z$ at $z$.  Since $z$ is a codimension one point of 
    the normal scheme $Z$, $R$ is a DVR, hence we can assume $Z = \spec R$ 
    and that $\hg^0(X, \mc{O}_X(A))$ is a free $R$-module
    of finite rank.  The rest of the proof follows verbatim from 
    the proof of \cite[Lemma 3.4]{birkar2023singularities}.
\end{proof}

%%%%%%%%%%%%%%%%%%%%%%%%%%%%%%%%%%%%%%%%%%%%

\begin{lemma}[cf. \protect{\cite[Lemma 3.5]{birkar2023singularities}}]\label{bir-universal-family}
    Let $\mc{P}$ be a generically relatively bounded family of couples $(X/Z, D)$, where
    $Z$ is a normal variety.  Then there exist finitely many
    couples $(V_i/T_i, C_i)$ satisfying the following.  Assume $(X/Z, D)\in \mc{P}$.
    Then for each codimension one point $z\in Z$, perhaps after shrinking $Z$ around $z$, 
    there exist an $i$ and a morphism $Z\to T_i$ such that 
    \[ X = Z\times_{T_i} V_i \text{ and } D = Z\times_{T_i} C_i. \]
\end{lemma}

\begin{proof}
    By Lemma~\ref{bir-lem-on-Pn}, there is an $n\in \NN$ depending only on $\mc{P}$
    such that perhaps after shrinking $Z$ around $z$ the morphism $X\to Z$
    factors through a closed immersion $X\to \PP_Z^n$.  
    In particular, as $\mc{O}_{Z,z}$ is a DVR, $X\to Z$ can be viewed as a flat family of closed subschemes of $\PP_Z^n$
    with finitely many possible Hilbert polynomials depending only on $\mc{P}$.
    Similarly, $D\to Z$ is locally flat over $z\in Z$, 
    hence it can be viewed as a flat family of closed subschemes of $\PP_Z^n$
    (of one dimension less) again with finitely many possible Hilbert polynomials
    depending only on $\mc{P}$.
    Taking a stratification of $\mc{P}$, we can assume that the Hilbert polynomials 
    are fixed in each case.  The rest of the proof follows identically from 
    the proof of \cite[Lemma 3.5]{birkar2023singularities}.
\end{proof}

%%%%%%%%%%%%%%%%%%%%%%%%%%%%%%%%%%%
%%%%%%%%%%%%%%%%%%%%%%%%%%%%%%%%%%%

\section{Toroidal geometry}\label{toroidal-section}

Part of the proofs of Theorem~\ref{vertical-bnd-intro} relies on 
the proof of the main result of \cite{birkar2024irrationality}, where toroidal geometry
is one of the crucial ingredients.
In this section, we collect the necessary notions from toroidal geometry for clarification.

\subsection{Toroidal couples}\label{toroidal-couples-defn}
We refer the readers to \cite{CLS:toric} for the general theory of toric varieties.
Let $(X,D)$ be a couple.  We say that the couple is \emph{toroidal} at a closed point $x\in X$ 
if there exist a \emph{normal} affine toric variety $W$ and a closed point $w\in W$ such that there is 
a $\mathbb K$-algebra isomorphism 
\[ \widehat{\mathcal{O}}_{{X},{x}}\to \widehat{\mathcal{O}}_{{W},w} \]
of completion of local rings so that the ideal of $D$ is mapped to the ideal of the toric boundary divisor 
$C\subset W$, that is, the reduction of complement of the torus $\mathbb{T}_W$ of $W$;
see \cite[\S 3.10]{birkar2023singularities}. 
We call $\Set{(W,C),w}$ a \emph{local toric model} of $\Set{(X,D),x}$.
We say that the couple $(X,D)$ is \emph{toroidal} if it is toroidal at every closed point;
in this case, we say that $D$ is the \emph{toroidal boundary} of $(X, D)$.

In literature, the open immersion $U_X:=X\setminus \supp D \subseteq X$ is usually called a 
\emph{toroidal embedding}; for example, see \cite[Chap. II, \S 1, page 54]{KKMB}.
In particular, $U_X$ is smooth as $\mc{O}_{X,x}$ is regular if
and only if $\wh{\mc{O}}_{X ,x}$ is regular.  
We will use the notions toroidal couples and toroidal embeddings interchangeably
to be consistent with literature.
Moreover, if the embedding $(U_X\subset X)$, or equivalently the couple $(X, D)$, 
is clear from the context, we just say that $X$ is a \emph{toroidal variety}.

Let $(X, D)$ be a toroidal couple.  If every irreducible component of $D$ is normal,
we call $(X, D)$ a \emph{strict toroidal couple}.  In this case, we say that the corresponding
toroidal embedding $(U_X\subset X)$ is a \emph{strict toroidal embedding} or
a \emph{toroidal embedding without self-intersection}; see \cite[page 57]{KKMB}.
Every log smooth couple is a strict toroidal couple.

%%%%%%%%%%%%%%%%%%%%%%%%%%%%%%%%%%%%%%%%%%%%%%

\subsection{Toroidal morphisms}\label{toroidal-morphisms}

Let $(X,D)$ and $(Y,E)$ be couples, and let $f\colon X\to Y$ be a morphism of varieties. 
Let $x\in X$ be a closed point 
and $y=f(x)$. We say $(X,D)\to (Y,E)$ 
is a \emph{toroidal morphism at $x$} if there exist local 
toric models $\Set{(W,C),w }$ and $\Set{(V,B),v}$ of 
$\Set{(X,D), x}$ and $\Set{(Y,E), y}$ respectively, and a toric morphism $g\colon W\to V$ of normal affine toric varieties 
so that we have a commutative diagram 
\[\xymatrix{
\widehat{\mathcal{O}}_{{X},{x}}\ar[r] & \widehat{\mathcal{O}}_{{W},w}\\
\widehat{\mathcal{O}}_{{Y},{y}} \ar[u] \ar[r] & \widehat{\mathcal{O}}_{{V},v} \ar[u]
}
\]
where the vertical maps are induced by the given morphisms $f$ and $g$ and the horizontal maps are 
isomorphisms induced by the local toric models; see \cite[\S 3.10]{birkar2023singularities}.
We say that the morphism
$f\colon (X, D) \to (Y, E)$ is \emph{toroidal} 
if it is toroidal at every closed point of $X$. 
Equivalently, we call the corresponding morphism 
$f\colon (U_X\subset X) \to (U_Y\subset Y)$
a \emph{toroidal morphism of toroidal embeddings}; cf. \S \ref{toroidal-couples-defn}.

%%%%%%%%%%%%%%%%%%%%%%%%%%%%%%%%%%%%%%%%%%%%%%%%%%%%%

\subsection{Relative toroidalisation}
The following result generalises the base change functoriality of \cite[Theorem 3.12 (ii)]{birkar2024irrationality}
from nonsingular curves to codimension one points of higher-dimensional normal varieties;
the proof of this result follows directly from \cite[Theorem 3.11]{birkar2024irrationality}.
For more details about saturated base change of dominant toroidal morphisms,
see \cite{Q-log-geometry}.

\begin{theorem}[cf. \protect{\cite[Theorem 3.12]{birkar2024irrationality}}]\label{functorial-toroidalisation}
	Let $\Phi \colon \mf{X} \to \mf{B}$ be a projective, 
	surjective morphism of varieties with geometrically integral generic fibre.  
	Let $\mf{Z}_{\mf{X}}\subset \mf{X}$ and $\mf{Z}_{\mf{B}}\subset \mf{B}$ be proper closed subsets
    of $\mf{X}$ and $\mf{B}$ respectively.	
	Then there exists a toroidal morphism $\Phi'$ between toroidal embeddings satisfying the following properties.
	\begin{itemize}
        \setlength\itemsep{0.5em}
		\item [\emph{(i)}] \emph{[Existence]} There is a commutative diagram
	          \[\xymatrix{
	           (U_{\mf{X}'}\subset \mf{X}')\ar[r]^-{m_{\mf{X}}}\ar[d]^{\Phi'} &\mf{X}\ar[d]^{\Phi} \\
	          (U_{\mf{B}'}\subset \mf{B}')\ar[r]^-{m_{\mf{B}}} & \mf{B}
	          }\]
	          where 
              \begin{itemize}
                  \item [\emph{(1)}] $m_{\mf{B}}$ and $m_{\mf{X}}$ are projective birational morphisms,
                  \item [\emph{(2)}] the embeddings on the left are strict toroidal embeddings,
                  \item [\emph{(3)}] $m_{\mf{X}}^{-1}(\mf{Z}_{\mf{X}})$ is contained in $\mf{X}'\setminus U_{\mf{X}'}$,
                  \item [\emph{(4)}] $m_{\mf{B}}^{-1}(\mf{Z}_{\mf{B}})$ is contained in $\mf{B}'\setminus U_{\mf{B}'}$, and
                  \item [\emph{(5)}] $\Phi'$ is a surjective, projective toroidal morphism.
              \end{itemize}
	     \item [\emph{(ii)}] \emph{[Base change functoriality]} For any given open subset $U\subseteq \mf{B}$, there is a nonempty open subset 
          \[U_{\mf{B}}\subset U\cap(\mf{B}\setminus \mf{Z}_{\mf{B}})\]
	         satisfying the following property.  Assume that
             \begin{itemize}
                 \item [\emph{(1)}] $i\colon B\to \mf{B}$ is a morphism from a normal variety $B$,
                 \item [\emph{(2)}] $b\in B$ is a codimension one point of $B$, and
                 \item [\emph{(3)}] the image $i(B)$ is not entirely contained in the closed subset $\mf{B}\setminus U_{\mf{B}}$.
             \end{itemize} 
             Then there exists an open dense subset $B'\subseteq B$ containing $b$
             such that 
             \begin{itemize}
                 \item [\emph{(a)}] $B'$ is a nonsingular variety,
                 \item [\emph{(b)}] the closure $\ol{\Set{b}}$ of the codimension one point $b$ in $B'$ is a nonsingular reduced divisor on $B'$ so that $(B', \ol{\Set{b}})$ is a log smooth couple, 
                 \item [\emph{(c)}] $B'$ admits a morphism $i'\colon B'\to \mf{B}'$ making the following diagram commutative
                 \[\xymatrix{
                 \mf{B}'\ar[r]^-{m_{\mf{B}}} & \mf{B} \\
                 B'\ar[u]_{i'}\ar[r]^-{m_B} & B\ar[u]_i
                 }\]
                 where $m_B$ is the open immersion $B'\hookrightarrow B$, 
                 \item [\emph{(d)}] $(U_{B'}\subset B')$ is a strict toroidal embedding, where $U_{B'}$ is equal to $(i')^{-1}(U_{\mf{B}'})$, and
                 \item [\emph{(e)}] $\mf{X}\times_{\mf{B}} B$ (respectively, $\mf{X}'\times_{\mf{B}'} B'$) has a unique irreducible component dominating $B$ (respectively, $B'$), 
                 which we call the main component.  
             \end{itemize}
             Moreover, denote by $X$ the reduction of the main component of $\mf{X}\times_{\mf{B}} B$ with projection morphism $f\colon X\to B$. 
             Consider the diagram arising from 
	         normalised base change as follows:
	         \[\xymatrix{
	         (U_{X'}\subset X')\ar[ddd]^{f'}\ar[rrr]^{m_X}\ar[rd] & & & X\ar[ddd]^f\ar[ld] \\
	           & (U_{\mf{X}'}\subset \mf{X}')\ar[r]^-{m_{\mf{X}}}\ar[d]^{\Phi'} &\mf{X}\ar[d]^{\Phi} & \\
	           & (U_{\mf{B}'}\subset \mf{B}')\ar[r]^-{m_{\mf{B}}} & \mf{B} & \\
	         (U_{B'}\subset B')\ar[rrr]^{m_B}\ar[ru]^{i'} & & & B\ar[lu]_i
	         }\]
	         where $X'$ is 
	         the normalisation of the main component of $\mf{X}'\times_{\mf{B}'} B'$, and
	         $U_{X'}$ is the inverse image of $U_{\mf{X}'}$ in $X'$.
	         Then 
             \[ f'\colon (U_{X'}\subset X')\to (U_{B'}\subset B') \] 
             is also a toroidal morphism between strict toroidal embeddings 
             satisfying all the properties listed in \emph{(i)}.
	         In particular, $m_X\colon X'\to X$ is a projective birational morphism, and
	         the inverse image of $\mf{Z}_{\mf{B}}$ in $B'$ (respectively, of 
	         $\mf{Z}_{\mf{X}}$ in $X'$) is contained in $B'\setminus U_{B'}$ (respectively, in $X'\setminus U_{X'}$).
	\end{itemize}
\end{theorem}

\begin{proof}
    The result (i) follows verbatim from the proof of \cite[Theorem 3.12 (i)]{birkar2024irrationality}.
    For (ii), let $V_{\mf{B}}\subset \mf{B}$ be the closed subset 
	that is the image of the toroidal boundary $\mf{B}'\setminus U_{\mf{B}'}$.
	As $m_{\mf{B}}$ is projective and birational,
	$V_{\mf{B}}$ avoids the generic point of $\mf{B}$, hence
	$\mf{B}\setminus V_{\mf{B}}$ is a nonempty open subset of $\mf{B}$ contained in $\mf{B}\setminus \mf{Z}_{\mf{B}}$.
    Let $U_{\mf{B}}$ be an open subset of $U\cap (\mf{B}\setminus V_{\mf{B}})$
    such that 
    \begin{itemize}
        \item $m_{\mf{B}}$ is an isomorphism over $U_{\mf{B}}$,
        \item every fibre of $\Phi$ (respectively, of $\Phi'$) over a closed point of $U_{\mf{B}}$ (respectively, of $m_{\mf{B}}^{-1}(U_{\mf{B}})$) is integral, and
        \item for every closed point $c'\in m_{\mf{B}}^{-1}(U_{\mf{B}})$, 
        the morphism $\mf{X}'_{c'}\to \mf{X}_c$ induced by $m_{\mf{X}}$ is birational, where $c := m_{\mf{B}}(c')$.
    \end{itemize}
	Then given a morphism $i\colon B\to \mf{B}$ from a normal variety $B$
    such that $i(B)$ is not entirely contained in $\mf{B}\setminus U_{\mf{B}}$,
    the existence of the morphism $i'\colon B'\to \mf{B}'$
    satisfying the conditions (a), (b) and (c) follows from
    the valuative criterion of properness 
    as $b$ is a codimension one point of the normal variety $B$.
    If $i'(b)$ is contained in $U_{\mf{B}'}$,
    we can assume that the whole $i'(B')$ is contained in $U_{\mf{B}'}$, then
    we give $B'$ the empty toroidal boundary.
    If $i'(b)$ belongs to $\mf{B'}\setminus U_{\mf{B}'}$,
    shrinking $B'$ near $b$ if necessary, we can assume that
    $B'\setminus U_{B'}$ is equal to the closure of $\Set{b}$ in $B'$.
    Thus, $(U_{B'}\subset B')$ is a strict toroidal embedding as $(B', \ol{\Set{b}})$
    is a log smooth couple, hence (d) holds.
    As the irreducible components of $\mf{X}\times_{\mf{B}}B$ dominating $B$
    correspond bijectively to the irreducible components of 
    the generic fibre of $\mf{X}\times_{\mf{B}}B\to B$,
    (e) follows immediately from the construction of $U_{\mf{B}}$.
    Finally, $f'\colon (U_{X'}\subset X')\to (U_{B'}\subset B')$ 
    is a toroidal morphism between strict toroidal embeddings by
	\cite[Theorem 3.11]{birkar2024irrationality}.
\end{proof}

%%%%%%%%%%%%%%%%%%%%%%%%%%%%%%%%%%%
%%%%%%%%%%%%%%%%%%%%%%%%%%%%%%%%%%%

\section{SS-degree of divisors on fibrations of relative dimension one}\label{threefolds-rel-dim-one}

We prove $\mc{V}(1, t)$ in this section, which is the first step
towards the inductive proof of $\mc{V}(d, t)$ on the relative dimension $d$ in \S\ref{pf-V-d-t}.

\subsection{Proof of $\mc{H}(1,t)$}\label{pf-S-1-t}
We first show that $\mc{H}(d, t)$ holds for every $t\in \RR^{>0}$
if a general fibre of the generalised log Calabi-Yau fibration $(X, B + \mb{M})\to Z$ is of dimension one.
	
\begin{lemma}[$\mc{H}(1, t)$]\label{base-case-r-dim=1}
    Let $t\in (0,1]$.
    Let $(X, B + \mb{M})\to Z$ be a generalised log Calabi-Yau fibration of relative dimension one.
    Let $S$ be a horizontal$/Z$ irreducible component of $B$ such that $\coeff_S B \ge t$.
    Then $\sdeg (S^{\nor}/Z)$ is bounded from above depending only on $t$.
\end{lemma}
	
\begin{proof}
	As $\dim (X/Z) = 1$, a general fibre $F$ of $X\to Z$ is an irreducible and nonsingular curve.
   Moreover, both $S^{\nor}\to Z$ and $S\to Z$ are generically finite, hence
   \[ \sdeg (S^{\nor}/Z) = \sdeg (S/Z) = \deg (S/Z). \]
   Write
   \[ K_F + B_F + M_F = (K_X + B + \mb{M}_X)|_F \sim_{\RR} 0, \]
   where $B_F := B|_F$, and $M_F := \mb{M}_X |_F$ is an $\RR$-divisor on $F$ with $\deg M_F \ge 0$.
   As $S|_F$ is effective and nonzero, $\deg K_F < 0$, so $F$ is isomorphic to $\PP^1$.
   Thus,
   \[ \deg (S/Z) = \deg (S|_F) \le \frac{1}{t} \deg (-K_F) = \frac{2}{t}, \]
   which shows that $\sdeg (S^{\nor}/Z)$ is bounded from above depending only on $t$.
\end{proof}

%%%%%%%%%%%%%%%%%%%%%%%%%%%%%%%%%%%%%%%%%%%%%%%%

\subsection{Boundedness of SS-degree on fibrations of relative dimension one}\label{pf-pre-V-1-t}

\begin{lemma}\label{surface-comp-bnd}
    Let $t\in \RR^{>0}$.
    Let $(X, \Delta)$ be a $\QQ$-Gorenstein surface lc pair.
    Let $x\in X$ be an arbitrary closed point.
    Then the number of irreducible components of $\Delta$ 
    passing through $x\in X$ and having coefficients $\ge t$
    is bounded from above depending only on $t$.
\end{lemma}

\begin{proof}
    By strong Lefschetz principle (see \cite{Lefschetz-principle}), we can work over $\CC$.
    If $(x\in X, 0)$ is not log terminal, then $\Delta = 0$ near $x\in X$, so we can assume that
    $(x\in X, 0)$ is log terminal in the rest of the proof.
    In particular, shrinking $X$ around $x$, we can assume that $x$ is the only singularity of $X$
    and that $X$ is $\QQ$-factorial; see \cite[Proposition 4.11]{km98}.  

    \medskip
    
    \emph{Step 1}.
    Let $p\colon (y\in Y)\to (x\in X)$ be the index 1 cover near $x\in X$; see \cite[Definition 5.19]{km98}.
    Then by \cite[Definition 5.19, Corollary 5.21]{km98},
    up to shrinking $X$ near $x\in X$ if necessary,
    \begin{itemize}
        \item $p$ is finite \'etale over $X\setminus \Set{x}$, which implies that $K_Y = p^* K_X$, and
        \item $Y$ has a canonical singularity at $y$ of index 1, i.e., $K_Y$ is Cartier.
    \end{itemize}
    Write
    \[ K_Y + \Delta_Y = p^*(K_X + \Delta), \]
    where $\Delta_Y = p^* \Delta$ is effective.  
    Then $(Y, \Delta_Y)$ is also a surface lc pair by \cite[Proposition 5.20]{km98}.
    Hence, it suffices to show that the number of
    irreducible components of $\Delta_Y$ passing through $y\in Y$
    and having coefficients $\ge t$ is bounded from above depending only on $t$.
    Moreover, it suffices to prove the result in the analytic topology (over $\CC$). 
    By \cite[\S 4.2]{km98}, $(y\in Y)$ belongs to one of the types 
    $A_n$ ($n\ge 1$), $D_n$ ($n\ge 4$), $E_6$, $E_7$, and $E_8$.  

    \medskip

    \emph{Step 2}.  
    Let $\pi\colon W\to Y$ be the minimal resolution of $y\in Y$,
    and let $E' := \Set{ E_1', \dots, E_m' }$ be the collection of exceptional curves of $\pi$. 
    The dual graph of $E'$ is listed in \cite[Theorem 4.22]{km98},
    and $(E_i')^2 = -2$ for every irreducible component $E_i'$; see \cite[page 198]{Matsuki-intro-Mori}.  Let 
    \[ \big( (E_i'\cdot E_j') \big)_{1\le i, j\le m} \] 
    be the intersection matrix of $E'$.
    Let $H$ be the Abelian group with generators $e_1, \dots, e_m$ 
    subject to the relations (see \cite[page 225]{Lipman-ra-sing-UFD})
    \[ \sum_{j=1}^m (E_i'\cdot E_j') e_j = 0 \text{ for } 1\le i\le m. \]
    Then $H$ is a finite group of order 
    \[ \abs{ \det \big( (E_i'\cdot E_j') \big)_{1\le i, j\le m} }.  \]
    Denote by $\cl (\mc{O}_{Y,y})$ the divisor class group of $\mc{O}_{Y,y}$;
    cf. \cite[page 222]{Lipman-ra-sing-UFD} and \cite[page 165]{CA-Mat}.
    By the proof of \cite[Proposition (17.1)]{Lipman-ra-sing-UFD}, there is an exact sequence
    \[ 0\to \cl (\mc{O}_{Y,y}) \to H \to G \to 0, \]
    where $G$ is a finite group defined on \cite[page 225]{Lipman-ra-sing-UFD}.
    By standard linear algebra, it is easy to see that the order $\abs{H}$
    of the finite group $H$ is given by
    \[ \abs{H} =
    \begin{dcases}
        n+1 & \text{ if } (y\in Y) \text{ is of type } A_n\,\,(n\ge 1) \\
        4 & \text{ if } (y\in Y) \text{ is of type } D_n\,\,(n\ge 4) \\
        3 & \text{ if } (y\in Y) \text{ is of type } E_6 \\
        2 & \text{ if } (y\in Y) \text{ is of type } E_7 \\
        1 & \text{ if } (y\in Y) \text{ is of type } E_8.
    \end{dcases}
    \]

    \medskip

    \emph{Step 3}.
    If $(y\in Y)$ is of type $D_n$ ($n\ge 4$), $E_6$, $E_7$, or $E_8$, 
    then (up to shrinking $Y$ around $y$) for any irreducible curve $C$ passing through $y\in Y$,
    the Cartier index $I(C)$ of $C$ satisfies that
    \[ I(C) \le 4. \]
    Assume that $\Delta_Y$ has $r$ irreducible components $F_1, \dots, F_r$ passing through $y\in Y$
    and having coefficients $\ge t$.
    As $(Y, 0)$ is canonical but not terminal at $y$, there exists a birational contraction
    $\pi\colon W\to Y$ such that $\pi$ admits a single exceptional divisor $E$ (over $y\in Y$)
    whose log discrepancy is equal to one.  Thus, we can write
    \[ K_W = \pi^* K_Y. \]
    Moreover, for every $1\le i\le r$, $4F_i$ is a Cartier divisor passing through $y\in Y$, hence
    \[ \pi^* (4 F_i) = 4 F_i^{\sim} + \alpha_i E, \]
    where $F_i^{\sim}$ is the birational transform of $F_i$ to $W$ and $\alpha_i\in \NN$.
    Write
    \[ K_W + \Delta_W = \pi^*(K_Y + \Delta_Y), \]
    where $\Delta_W = \pi^* \Delta_Y$.  Then as $(Y, \Delta_Y)$ is an lc pair, we have
    \[ \frac{rt}{4} \le \sum_{i=1}^r \frac{\alpha_i t}{4} \le \coeff_E \Delta_W \le 1. \]
    Therefore, we can conclude that
    \[ r\le \frac{4}{t}, \]
    which is bounded from above depending only on $t$. 

    \medskip

    \emph{Step 4}.  If $(y\in Y)$ is of type $A_n$ ($n\ge 1$), then $(y\in Y)$ is the analytic quotient of $(0\in \CC^2)$
    by $C_{n+1}$, the cyclic group of order $n+1$; see \cite[Corollary 4-6-16]{Matsuki-intro-Mori}.
    Let $W$ be an Euclidean open subset of $\CC^2$ containing $0\in \CC^2$
    such that $\pi\colon W \to Y$
    is the corresponding analytic quotient morphism 
    defining the singularity $y\in Y$.
    Then $\pi$ is \'etale over $Y\setminus \Set{y}$, and $\deg \pi = n+1$.  Write
    \[ K_W + \wt{\Delta} = \pi^*(K_Y + \Delta_Y), \]
    where $\wt{\Delta} = \pi^* \Delta_Y$, then $(W, \wt{\Delta})$ is also a surface lc pair.
    Assume that $\Delta_Y$ has $r$ irreducible components $F_1, \dots, F_r$ 
    passing through $y\in Y$ and having coefficients $\ge t$.  
    Note that every $\pi^* F_i$ is an effective integral divisor on $W$ 
    as $\pi\colon W\setminus \Set{0}\to Y\setminus \Set{y}$ is \'etale.
    Since $W$ is nonsingular, we have 
    \[ rt\le \sum_{i=1}^r t \cdot\mult_0 (\pi^* F_i) \le \mult_0 \wt{\Delta} \le 2, \]
    hence
    \[ r\le \frac{2}{t}. \]
    Therefore, the number of irreducible components of $\Delta_Y$
    passing through $y\in Y$ and having coefficients $\ge t$
    is bounded from above depending only on $t$.  
\end{proof}

%%%%%%%%%%%%%%%%%%%%%%%%%%%%%%%%%%%%%%%%%%%%%%%%

The following result is a more general version of $\mc{V}(1,t)$.

\begin{theorem}\label{fir-one-u-t}
    Let $u, t\in (0,1]$.
	Let $f\colon (X, B + \mb{M}) \to Z$ be a generalised log Calabi-Yau fibration of relative dimension one such that
	\begin{itemize}
		\item [\emph{(1)}] there exists a horizontal$/Z$ irreducible component $H$ of $B$ with $\coeff_H B \ge u$, 
		\item [\emph{(2)}] $S$ is a vertical$/Z$ irreducible component of $B$ with $\coeff_S B \ge t$, and
        \item [\emph{(3)}] the image $T$ of $S$ in $Z$ is a prime divisor on $Z$. 
	\end{itemize}
	Then $\ssdeg (S^{\nor}/Z)$ is bounded from above depending only on $u,t$.
\end{theorem}

\begin{proof}
    Taking a $\QQ$-factorial g-dlt model, we can assume that $(X, B + \mb{M})$
    is a $\QQ$-factorial g-dlt g-pair.
    Moreover, by taking general Cartier hyperplane sections of $Z$, we can assume that
    $\dim X = 3$, $\dim Z = 2$, and $\dim T = 1$ by the same proof of Lemma~\ref{fano-type-induction}. 

    \medskip
	
	\emph{Step 1}.  
    By Lemma~\ref{base-case-r-dim=1}, 
	$\sdeg (H^{\nor}/Z)$ is bounded from above depending only on $u$.
    By \cite[Lemma 4.4 (1)]{B-Zhang}, we can
	run $\phi\colon X\dashrightarrow X'$ a $(-H)$-MMP over $Z$ that ends with a Mori fibre space
	$f'\colon X'\to Z'$ over $Z$.  Then the pushdown $H'$ of $H$ to $X'$ is ample$/Z'$.
	Note that the morphism $Z'\to Z$ is a birational contraction of
	normal surfaces as $\dim (X/Z) = 1$.  
    Shrinking $Z$ near the generic point of $T$, we can assume that $Z'\to Z$
    is an isomorphism of nonsingular surfaces.
	Denote by $B'$ the pushdown of $B$ to $X'$.
	Then we also have $K_{X'} + B' + \mb{M}_{X'}\sim_{\RR}0/Z'$ and $\coeff_{H'} B' \ge u$;
    see Lemma~\ref{CY-log-discre-no-change}. 

    \medskip

    Note that $S$ may be contracted by $\phi\colon X\dashrightarrow X'/Z$.
    In Steps 2 to 5, we treat the case that $S$ is contracted by $\phi$;
    in Step 6, we treat the case that $S$ is not contracted.

    \medskip
	
	\emph{Step 2}.
	Assume that $S$ is contracted by $\phi$.
	Denote by $C'$ the centre of $S$ on $X'$.  We show that $C'$ 
	is contained in $H'$.
	Write the sequence of $(-H)$-MMP over $Z$ as
	\[ X = X_0\xdashrightarrow{\phi_0} X_1 \xdashrightarrow{\phi_1} X_2 \dashrightarrow \cdots \xdashrightarrow{\phi_{n-1}} X_n = X'. \]
    Denote by $S_{k+1} := (\phi_k)_* S_k$ inductively with $S_0 := S$ for $0\le k \le n-1$.
    Let $H_k$ be the pushdown of $H$ to $X_k$ for every $0\le k \le n$;
    in particular, $H_0 = H$ and $H_n = H'$.
    
	Assume that $S_i$ is a divisor on $X_i$ and that $\phi_i\colon X_i \dashrightarrow X_{i+1}$ contracts
	$S_i$.  Then $\phi_i$ is a divisorial contraction and $\supp S_i$ is equal to the
	exceptional locus of $\phi_i$.
	Denote by $C_{k}$ the centre of $S$ on $X_k$.
	Note that $C_k = \supp S_k$ for $0\le k \le i$, but $S_k = 0$ for $k\ge i+1$.  
	Assume that the generic point $\eta_{C_{i+1}}$ of $C_{i+1}$
	is not contained in $H_{i+1}$.
	Let $F_i$ be a general fibre of $S_i \to C_{i+1}$, then $H_i\cdot F_i = 0$.
    As $\phi_i$ contracts a $(- H_i)$-negative extremal ray, we have
	$(-H_i)\cdot F_i <0$, that is, $H_i\cdot F_i > 0$, which contradicts the assumption that $H_i \cdot F_i = 0$.
	Thus, $C_{i+1}$ is contained in $H_{i+1}$.
	
	If $\phi_{i+1}\colon X_{i+1} \dashrightarrow X_{i+2}$ is a divisorial contraction, 
	then $C_{i+2}$ is also contained in $H_{i+2}$ since $C_{i+1}\subset H_{i+1}$.
	Assume that $\phi_{i+1}$ is a flip and
    that $C_{i+1}$ is a flipping curve;
    otherwise, $\phi_{i+1}$ is an isomorphism near $\eta_{C_{i+1}}$.  Then $C_{i+2}$ is a flipped curve, hence
	\[ H_{i+1}\cdot C_{i+1} > 0 \text{ and } (- H_{i+2})\cdot C_{i+2} >0, \]
	which implies that $H_{i+2}\cdot C_{i+2} < 0$.  Thus, we can conclude that $C_{i+2}$ is contained in $H_{i+2}$.
	Therefore, inductively, $C' = \centre_{X'} S$ is contained in the divisor $H'$.

    \medskip
	
	\emph{Step 3}.  
    As $(X', (B'-(\coeff_{H'}B')H') + \mb{M})$ is also a $\QQ$-factorial g-dlt g-pair, the pair $(X', 0)$ is klt.
    By \cite[Lemma 4.6]{B-Zhang}, let $\pi\colon X''\to X'$ be a birational contraction extracting only 
	the valuation of $S$, where $X''$ is $\QQ$-factorial.
	\[\xymatrix{
	  & X''\ar[d]^{\pi} \\
	X\ar@{-->}[r]^{\phi}\ar[d]_f & X'\ar[dl]^{f'} \\
	Z & 
	}\]
	By Step 2, we can write
	\[ \pi^* H' = a S'' + H'', \]
	where $S'' := \centre_{X''} S$, $H'' := \centre_{X''} H$ (i.e., the birational transform of $H'$ to $X''$),
	and $a$ is a positive rational number.  Then $H''\sim_{\QQ} -a S''/X'$ is ample over $X'$.
 
	On the other hand, consider the morphisms of integral surfaces
	$H''\to H' \to Z$.
    Since $H\dashrightarrow H'$ and $H''\to H'$ are birational, by Lemma~\ref{num-irr-comp}, we see that
    \[ \sdeg((H'')^{\nor}/Z) = \sdeg (H^{\nor}/Z). \]
	Let $U\subseteq Z$ be an open subset for $H''\to Z$ as in Lemma~\ref{finite-fibre}.
	Denote by $\eta_T$ the generic point of $T$, then $\eta_T\in U$ as $\codim(Z\setminus U, Z) \ge 2$.  
    Set $U_T := U\cap T \subseteq T$.
	Then we have
	\[ \mf{n}(H''_z)\le \sdeg (H^{\nor}/Z) \]
	for every closed point $z\in U_T$.

    \medskip
	
	\emph{Step 4}.  Pick a general hyperplane section $L$ of $Z$.
	By Bertini's theorem on irreducibility and smoothness (cf. \cite[Th\'eor\`eme 6.10]{jou-bertini}), 
	$L$ is an irreducible nonsingular curve that intersects $U_T$
	at general points of $T$.
	Denote by $X_L''$ the fibre product $L\times_Z X''$.
	By Bertini's theorem on normality (see \cite[Corollary 3.4.9]{normal-bertini}), 
    $X_L''$ is a normal, irreducible and $\QQ$-Gorenstein surface,
	and $X_L'' \to L$ is a contraction of normal varieties.  
    Write
	\[ K_{X''} + B'' + \mb{M}_{X''}  = \pi^*(K_{X'} + B' + \mb{M}_{X'}). \]
	Denote by $S''_L, H''_L$ the restrictions of
	$S'', H''$ to the surface $X_L''$ respectively.  
    By generalised adjunction (see \cite[Definition 4.7]{B-Zhang}), we can write
	\[ K_{X_L''} + B''_L + \mb{N}_{X_L''} \sim_{\RR} (K_{X''} + B'' + X_L'' + \mb{M}_{X''})|_{X_L''}, \]
    where $(X_L'', B_L'' + \mb{N})$ is also a g-lc g-pair.
    By \cite[Remark 1.15]{Filipazzi-bnd-g-pair-surface},
    $(X_L'', B_L'')$ is a $\QQ$-Gorenstein surface lc pair.
	Note that $H''_L$ is a horizontal$/L$ irreducible component of $B''_L$ by \cite[Th\'eor\`eme 6.10]{jou-bertini}. 

    \medskip

    \emph{Step 5}.
	Pick a closed point $z\in L\cap U_T$ that is a general closed point of $T$.
	Let $S''_z$ be the fibre of $S''\to T$ over $z$.
	Notice that $S''_z$ is a finite union of closed fibres of $S''\to C'$.
	As $H''$ is ample$/X'$, $H''$ intersects every irreducible component of $S_z''$, hence
	$H_L''$ intersects every irreducible component of $S_z''$.  
	Let $H''_{L,z}$ be the fibre of $H''_L\to L$ over $z$.
	Then $\red (H''_{L,z})$ consists of at most $\sdeg (H^{\nor}/Z)$ closed points by Step 3.
	Write
	\[ \red (H''_{L,z}) = \Set{x_1, \dots, x_r} \]
	for a natural number $r\le \sdeg (H^{\nor}/Z)$, 
    which is bounded from above depending only on $u$ by Lemma~\ref{base-case-r-dim=1}; see Step 1.
    Then every irreducible component of $S''_z$
	intersects at least one of the $x_i$'s.
	On the other hand, by Lemma~\ref{CY-log-discre-no-change}, and by the construction in \cite[Definition 4.7]{B-Zhang},
    every irreducible component of $S_z''$ has coefficient $\ge t$ in $B_L''$.
    Fix an arbitrary $x_i$ for $1\le i\le r$.
    By Lemma~\ref{surface-comp-bnd}, the number of irreducible components of $B_L''$
    passing through $x_i$ and having coefficients $\ge t$ is bounded from above depending only on $t$.  Therefore, 
    \[ \ssdeg (S^{\nor}/Z) = \mf{n}(S_z'') \] 
    is bounded from above depending only on $u,t$. 

    \medskip

    \emph{Step 6}.
    Keep the notation in Step 1.  Assume that $S$ is not contracted by $\phi$.
    Pick a general closed point $z\in T$.
    As $H'$ is ample$/Z$, $H'$ intersects every irreducible component of $S_z'$,
    where $S'$ is the pushdown of $S$ to $X'$, and $S'_z$ is the fibre of $S'$ over $z$.
    Then the rest of the proof follows from the same argument in Steps 4 and 5.
\end{proof}

%%%%%%%%%%%%%%%%%%%%%%%%%%%%%%%%%%%%%%%%%%%%%%%

\subsection{Proof of $\mc{V}(1,t)$}\label{pf-V-1-t}

Recall that a \emph{Fano type (generalised log Calabi-Yau) fibration} means 
a generalised log Calabi-Yau fibration $(X, B + \mb{M})\to Z$ whose underlying variety 
$X$ is of Fano type over $Z$.
Now we show that for Fano type fibrations, we can drop the requirement in Theorem~\ref{fir-one-u-t} on the
existence of a horizontal$/Z$ irreducible component whose coefficient is bounded from below;
see Theorem~\ref{fir-one-u-t} (1).

\begin{theorem}[$\mc{V}(1, t)$]\label{vertical-case-fib-dim-one}
    Let $t\in \RR^{>0}$.  
    Let $f\colon (X, B + \mb{M}) \to Z$ be a Fano type fibration
    of relative dimension one.
    Let $S$ be a vertical$/Z$ irreducible component of $B$ such that
    \begin{itemize}
        \item the coefficient of $S$ in $B$ is $\ge t$, and
        \item the image $T$ of $S$ is a prime divisor on $Z$.
    \end{itemize}
    Then $\ssdeg (S^{\nor}/Z)$ is bounded from above depending only on $t$.
\end{theorem}

\begin{proof}
    Shrinking $Z$ near the generic point of $T$, we can assume that $Z$ is nonsingular.
    By Lemma~\ref{fano-type-induction}, we can then assume that $\dim X = 3$, $\dim Z = 2$ and $\dim T = 1$.
    Moreover, by Lemma~\ref{fano-type-at-dlt-model},
    taking a $\QQ$-factorial g-dlt model, we can also assume that
    $(X, B + \mb{M})$ is a $\QQ$-factorial g-dlt g-pair.

    As $X$ is of Fano type over $Z$, we can run $\phi\colon X\dashrightarrow X'$
    a $(-S)$-MMP over $Z$ that ends with a good minimal model $X'$ of $-S$.  
    By negativity lemma, $S$ is not contracted by $\phi$.
    Set $S' := \phi_* S$.
    Let $g\colon X'\to Z'$ be the contraction induced by $-S'$.
    As $S'$ is vertical$/Z$, 
    the contraction $Z'\to Z$ is birational. 
    Let $T'$ be the birational transform of $T$ to $Z'$.
    As $-S'$ is nef$/Z'$, $S'$ is the unique irreducible component of $\supp (g^* T')$.
    Then we can write
    \begin{equation}
        S' = b\cdot g^* T' \label{S-whole-fibre-3}
    \end{equation}
    for some $b\in \QQ^{>0}$.   
    Shrinking $Z$ near the generic point of $T$, we can identify $(Z', T')$ with $(Z, T)$.
    Set $B' := \phi_* B$.
    By Lemma~\ref{CY-log-discre-no-change}, we have 
    \[ \coeff_{S'} B' = \coeff_S B \ge t. \]
    Moreover, by Lemma~\ref{Fano-type-pushdown}, $(X', B' + \mb{M})\to Z$ is also a Fano type fibration.  
    Decreasing $t$ if necessary, we can assume that $t\in \QQ^{>0}$.
    Run $\varphi\colon X'\dashrightarrow X''$ a $(- K_{X'} - t S')$-MMP 
    over $Z$ that ends with a minimal model $X''\to Z$ of $-K_{X'} - t S'$, then $-(K_{X''}+ t S'')$ is nef$/Z$,
    where $S'' := \varphi_* S'$.
    By \eqref{S-whole-fibre-3}, $S'$ is not contracted by $\varphi$. 

    By Lemma~\ref{Fano-type-pushdown}, $X''$ is also of Fano type over $Z$.
    Thus, by \cite[Theorem 1.8]{B-Fano}, up to shrinking $Z$ near the generic point of $T$, 
    $K_{X''} + t S''$ has a monotonic $n$-complement $K_{X''} + B^+$
    for some bounded $n\in \NN$ depending only on $t$, that is,
    \begin{itemize}
        \item $(X'', B^+)$ is an lc pair,
        \item $n(K_{X''} + B^+)\sim 0/Z$, and
        \item $t S''\le B^+$.
    \end{itemize}
    As $-K_{X''}$ is big$/Z$, $B^+$ admits a horizontal$/Z$ irreducible component $H''$ with coefficient $\ge 1/n$.
    Applying Theorem~\ref{fir-one-u-t} to the log Calabi-Yau fibration
    $(X'', B^+)\to Z$, we see that $\ssdeg (S^{\nor}/Z)$ is bounded from above depending only on $t$.
\end{proof}

%%%%%%%%%%%%%%%%%%%%%%%%%%%%%%%%%%%%
%%%%%%%%%%%%%%%%%%%%%%%%%%%%%%%%%%%%

\section{SS-degree of divisors on Fano type fibrations}\label{pf-V-d-t}

We prove $\mc{V}(d,t)$ (=Theorem~\ref{vertical-bnd-intro}) in this section.
The proof proceeds by induction on the relative dimension $d\in \NN$.
By Theorem~\ref{vertical-case-fib-dim-one}, $\mc{V}(1,t)$ holds.
Assuming $\mc{V}(k, t)$ for $1\le k\le d-1$, we prove $\mc{V}(d, t)$ via the methods 
from the proof of the main theorem of \cite{birkar2024irrationality}. 

The following notation denotes a family of fibrations that appears
in the proof of $\mc{V}(d,t)$.  We collect these notation here to simplify statements in the proof.

\begin{definitionnotation}\label{notation-bnd-family}
    Let $d, n\in \NN$ and let $\epsilon\in \RR^{>0}$.
    Denote by $\CY_{d, \epsilon, n}$ the set of data $(X/Z, B)$, where
    \begin{itemize}
        \item $X\to Z$ is a contraction of normal varieties,
        \item the relative dimension of $X\to Z$ is equal to $d$, 
        \item $(X, B)$ is an lc pair, 
        \item $X$ is $\epsilon$-lc over the generic point $\eta_Z$ of $Z$,
        \item the coefficients of $B$ belong to $\mb{T}_n := \Set{m/n}_{1\le m\le n}$, 
        \item $-K_X$ is ample$/Z$, and
        \item $K_X + B \sim_{\QQ} 0/Z$.
    \end{itemize}
\end{definitionnotation}

\begin{definitionnotation}\label{notation-bnd-family-curves}
    Let $d, n\in \NN$ and let $\epsilon\in \RR^{>0}$.
    Denote by $\CY_{d, \epsilon, n}^{\cur}$ the subset of $\CY_{d, \epsilon, n}$,
    where the base $Z$ of every $(X/Z, B)$ in $\CY_{d, \epsilon, n}^{\cur}$
    is a nonsingular curve.
\end{definitionnotation}

%%%%%%%%%%%%%%%%%%%%%%%%%%%%%%%%%%%%%%%%%%%%%

\subsection{Modification to relatively bounded couples}\label{modify-to-bnd-family}

In this subsection, we show that $\CY_{d, \epsilon, n}$ in Notation~\ref{notation-bnd-family}
can be modified to be a relatively bounded family of couples whose general fibres are log smooth; see \S\ref{log-smooth}.
The results in this subsection generalise slightly the lemmas in \cite[\S 4.1]{birkar2024irrationality}
by allowing the bases of fibrations to be higher-dimensional normal varieties,
not just nonsingular curves as in \cite[\S 4.1]{birkar2024irrationality}.

\begin{lemma}[cf. \protect{\cite[Lemma 4.2]{birkar2024irrationality}}]\label{bir-to-rbnd-couples}
    Let $d, n \in \NN$ and let $\epsilon\in \RR^{>0}$.
    There is a relatively bounded family $\mc{P}$ of couples such that for each 
    $f\colon (X, B)\to Z$ in $\CY_{d, \epsilon, n}$,
    we have a couple $(Y/Z_f, D)\in \mc{P}$, where $Z_f$ is an open subset of $Z$ with
    $\codim (Z\setminus Z_f, Z) \ge 2$, and
    a birational map $\phi\colon X'\dashrightarrow Y$ over $Z_f$,
    where $X' := Z_f\times_Z X$, such that
    \begin{itemize}
        \item [\emph{(1)}] the reduced divisor $D$ is horizontal$/Z_f$,
    	\item [\emph{(2)}] $\phi$ is an isomorphism over the generic point of $Z_f$, and
    	\item [\emph{(3)}] $\supp B'$ is mapped isomorphically to $\supp D$ over the generic point of $Z_f$, where $B'$ is the restriction of $B$ to $X'$.
    \end{itemize}
\end{lemma}

\begin{proof}
    The proof works almost the same as \cite[Lemma 4.2]{birkar2024irrationality}.
    For convenience of the readers, we include a sketchy proof of this result
    to indicate necessary modifications to the original proof of \cite[Lemma 4.2]{birkar2024irrationality}.

    \medskip
    
    \emph{Step 1}.
    By \cite[Theorem 1.1]{B-BAB}, there exists an $\ell\in \NN$ depending only on
    $d, \epsilon$ such that 
	for each $f\colon (X, B)\to Z$ in $\CY_{d, \epsilon, n}$,
    there is an open subset $U_f$ of $Z$ satisfying
    \begin{itemize}
        \item every closed fibre of $f$ over $U_f$ is an $\epsilon$-lc Fano variety of dimension $d$,
        \item $A := -\ell K_X$ is very ample$/U_f$, and
        \item the relative degree $\deg_{A/Z} A$ is bounded from above depending only on $d, \epsilon$.
    \end{itemize}
    Denote by $\CY_{d, \epsilon, n}^{\circ}$ the set of all morphisms that are
    restrictions of $f\colon X\to Z$ over $U_f$, where $f$ is the underlying morphism of varieties
    for some $(X/Z, B)\in \CY_{d, \epsilon, n}$.
	Then $\CY_{d, \epsilon, n}^{\circ}$ is a generically relatively bounded family of varieties.

    \medskip
	
	\emph{Step 2}.
	For every $f\colon (X, B)\to Z$ in $\CY_{d, \epsilon, n}$, we have that 
    \[ (A|_{X_s})^{d-1}\cdot (-n K_{X_s})\le r, \]
    where $X_s$ is a general fibre of $X\to Z$ and $r$ is a natural number depending only on 
    $d, \epsilon, n$.  Denote by $C$ the reduced divisor supporting 
    on the horizontal part of $B$.  Since
    the coefficients of $B$ belong to $\mb{T}_n$, $nB-C$
    is an effective integral divisor.  Moreover, as
    \[ C + (n B - C) \sim_{\QQ} -nK_X/Z, \]
    we also have 
    \[ \deg_{A/Z} C =  (A|_{X_s})^{d-1}\cdot (C|_{X_s})\le (A|_{X_s})^{d-1}\cdot \big((-nK_X)|_{X_s}\big)
    = (A|_{X_s})^{d-1}\cdot (- n K_{X_s})\le r. \]
    Denote by $\mc{P}^{\circ}$ the family of all the couples 
    $\big(f^{-1}(U_f)/U_f, C|_{f^{-1}(U_f)}\big)$.  
    Then $\mc{P}^{\circ}$ is a generically relatively bounded family of couples.
    By Lemma~\ref{bir-universal-family}, there are finitely many projective morphisms $V_i\to T_i$ 
	of varieties and reduced divisors $C_i$ on $V_i$ such that each couple in $\mc{P}^{\circ}$
	is a base change of some $(V_i, C_i)\to T_i$ (after shrinking
	the base $U_f$ near its generic point if necessary).  Then for every $f\colon (X, B)\to Z$
	in $\CY_{d, \epsilon, n}$, we can assume that there 
	is a morphism from $U_f$ to some $T_i$
	such that $f^{-1}(U_f)$ (respectively, $C|_{f^{-1}(U_f)}$)
	is equal to the fibre product $U_f\times_{T_i} V_i$
	(respectively, $U_f\times_{T_i} C_i$).

    \medskip
	
	\emph{Step 3}.
	For every $T_i$, let $\wt{T}_i$ be a projective compactification of $T_i$.  
	Embed $V_i\to T_i$ into some $\PP^N\times T_i$; 
	denote by $\wt{V}_i$ (respectively, by $\wt{C}_i$) the scheme-theoretic closure of 
	$V_i$ (respectively, of $C_i$) in $\PP^N\times \wt{T}_i$.  
	Then $(\wt{V}_i, \wt{C}_i)\to \wt{T}_i$ is a projective compactification
	of $(V_i, C_i)\to T_i$, where $\wt{V}_i$, $\wt{C}_i$, and $\wt{T}_i$ are all projective schemes over $\mbb K$.  
	As $Z$ is normal, valuative criterion of properness implies that 
    the morphism $U_f \to T_i$ extends to a morphism $Z_f\to \wt{T}_i$
    for some open subset $Z_f$ of $Z$ satisfying
    \[ \codim (Z\setminus Z_f, Z) \ge 2. \]
	Denote by $\mc{P}$ the set of all couples $(Y/Z_f, D)$ constructed from
	base changes $Z_f\times_{\wt{T}_i} \wt{V}_i$ and $Z_f\times_{\wt{T}_i} \wt{C}_i$ 
	for morphisms $Z_f\to \wt{T}_i$ 
	as in \cite[Lemma 2.14]{birkar2024irrationality}.  Then $\mc{P}$ is generically relatively bounded.  
    By construction, every $f\colon (X, B)\to Z$ in $\CY_{d, \epsilon, n}$ admits
	a birational map $\phi\colon X'\dashrightarrow Y$ over $Z_f$ to some $(Y/Z_f, D)$ in $\mc{P}$,
    where $X' := Z_f\times_Z X$; moreover, $(X'/Z_f, \supp B')$ and $(Y/Z_f, \supp D)$ have the same general fibres,
    where $B'$ is the restriction of $B$ to $X'$.  
    By \S\ref{r-bnd-def}, up to shrinking $Z_f$ further
    (also satisfying that $\codim(Z\setminus Z_f, Z)\ge 2$), we see that $\mcP$ is relatively bounded.
\end{proof}

%%%%%%%%%%%%%%%%%%%

\begin{lemma}\label{bir-to-rlbnd-generic-sm-couples}
    Let $d, n \in \NN$ and let $\epsilon\in \RR^{>0}$.
    There exists a relatively bounded family $\CY_{d, \epsilon, n}^{\text{sm}}$ of couples 
    satisfying the following property.
    For every $f\colon (X, B)\to Z$ in $\CY_{d, \epsilon, n}$, there is
    a couple $(Y/Z_f, D)$ in $\CY_{d, \epsilon, n}^{\text{sm}}$ such that
    \begin{itemize}
        \item [\emph{(1)}] $Z_f$ is an open subset of $Z$ whose complement has codimension $\ge 2$ in $Z$, 
    	\item [\emph{(2)}] there is a birational map $\phi\colon X'\dashrightarrow Y$ over $Z_f$, where $X' := Z_f\times_Z X$, 
        \item [\emph{(3)}] $(Y, D)\to Z_f$ is generically log smooth (see \S\ref{log-smooth}),
       	\item [\emph{(4)}] $D$ is horizontal$/Z_f$, and
    	\item [\emph{(5)}] over the generic point of $Z_f$, we have the following:
    	      \begin{itemize}
    	      \item [\emph{(i)}] $\phi$ does not contract any divisor, and
    	      \item [\emph{(ii)}] the divisor $D$ supports on the union of $\supp \phi_* B'$ and the 
    	                          support of all horizontal$/Z_f$ exceptional divisors of $\phi^{-1}$,
                                  where $B'$ is the restriction of $B$ to $X'$.
    	      \end{itemize}
    \end{itemize}
\end{lemma}

\begin{proof}
    Take the relatively bounded family of couples $\mc{P}$ in Lemma~\ref{bir-to-rbnd-couples}.
    By construction in the proof of Lemma~\ref{bir-to-rbnd-couples},
    there are finitely many projective morphisms $\wt{V}_i\to \wt{T}_i$
    of projective varieties and reduced horizontal$/\wt{T}_i$ divisors $\wt{C}_i$ satisfying that 
    for each $(\wt{Y}/Z_f, \wt{D}) \in \mc{P}$, there is a morphism $Z_f \to \wt{T}_i$
    such that $(\wt{Y}/Z_f, \wt{D})$ is the base change of $(\wt{V}_i, \wt{C}_i) \to \wt{T}_i$ via $Z_f\to \wt{T}_i$.
    Note that $Z_f$ is an open subset of $Z$ for some $f\colon (X, B)\to Z$ in $\CY_{d, \epsilon, n}$ 
    such that the complement of $Z_f$ in $Z$ has codimension $\ge 2$.  
	Take a log resolution $V^{\text{sm}}_i\to \wt{V}_i$ of the couple $(\wt{V}_i, \wt{C}_i)$.  
	Denote by $C^{\text{sm}}_i$ the reduced divisor supporting on the union of the birational 
	transform of $\wt{C}_i$ and all the horizontal$/\wt{T}_i$ exceptional divisors of $V^{\text{sm}}_i\to \wt{V}_i$.

    Denote by $\wt{E}_i$ the reduced exceptional divisor of $V_i^{\text{sm}}\to \wt{V}_i$.
	By generic smoothness and generic flatness, there exists a dense open
	subset $U_i\subset \wt{T}_i$ such that 
    \begin{itemize}
        \item $(V^{\text{sm}}_i, C^{\text{sm}}_i)$ is log smooth over $U_i$,
        \item for every closed point $t\in U_i$, $V^{\text{sm}}_{i,t}\to \wt{V}_{i,t}$ is a birational morphism whose (reduced) exceptional divisor is equal to $\wt{E}_{i,t}$, where $V^{\text{sm}}_{i,t}$, $\wt{V}_{i,t}$ and $\wt{E}_{i,t}$ are the fibres of $V^{\text{sm}}_i$, $\wt{V}_i$ and $\wt{E}_i$ over $t\in U_i$ respectively, and
        \item the images of all vertical$/\wt{T}_i$ irreducible components of $\wt{E}_i$ do not intersect $U_i$.
    \end{itemize}
	If the image of $Z_f\to \wt{T}_i$ is not entirely contained in $S_i := \wt{T}_i\setminus U_i$, 
	we can take the couple $(Y/Z_f, D)$ as the reduction of the main component of $Z_f\times_{\wt{T}_i}V^{\text{sm}}_i$ 
	equipped with the reduced divisor which is the reduction 
	of the horizontal$/Z_f$ part of $Z_f\times_{\wt{T}_i} C^{\text{sm}}_i$. 
    Up to shrinking $Z_f$ further (whose complement also has codimension $\ge 2$ in $Z$), 
    we can assume that $(Y/Z_f, D)$ is a flat couple.  Write $S_i$ as a union of finitely many 
	irreducible schemes, then the result follows by
	making a suitable stratification on $\mc{P}$ and doing induction on $\dim \wt{T}_i$.
\end{proof}

%%%%%%%%%%%%%%%%%%%%%%%%%%%%%%%%%%%%%%%%%%%%%
%%%%%%%%%%%%%%%%%%%%%%%%%%%%%%%%%%%%%%%%%%%%%

\subsection{Proof of $\mc{V}(d, t)$}\label{pf-vertical-bnd}

We are now ready to give the proof of Theorem~\ref{vertical-bnd-intro}.

\begin{theorem}[=Theorem~\ref{vertical-bnd-intro}]\label{vertical-bnd}
    Let $d\in \NN$ and let $t \in (0,1]$.  Let 
    \[ f\colon (X, B + \mb{M})\to Z \]
    be a Fano type generalised log Calabi-Yau fibration
    of relative dimension $d$.  
    Let $S$ be a vertical$/Z$ irreducible component of $B$ such that
    \begin{itemize} 
        \item the coefficient of $S$ in $B$ is $\ge t$, and
        \item the image of $S$ in $Z$ has codimension one in $Z$.
    \end{itemize}
    Then $\ssdeg (S^{\nor}/Z)$ is bounded from above depending only on $d, t$.
\end{theorem}

\begin{proof}
    Decreasing $t$ slightly, we can assume that $t\in (0,1]\cap \QQ$.
    By Theorem~\ref{vertical-case-fib-dim-one}, $\mc{V}(1, t)$ holds.
    We prove the result by induction on the relative dimension $d$, hence we assume that 
    $\mc{V}(k ,t)$ holds for every $1\le k \le d-1$. 
    Taking a $\QQ$-factorial g-dlt model, we can assume that $(X, B + \mb{M})$
    is a $\QQ$-factorial g-dlt g-pair.  By Lemma~\ref{fano-type-at-dlt-model}, the Fano type property 
    of $X\to Z$ is preserved.  Let $T$ be the image of $S$ in $Z$.  
    By taking general Cartier hyperplane sections of $Z$ inductively, we can assume that $\dim Z = 2$ and $\dim T = 1$;
    see Lemma~\ref{fano-type-induction}. 

    \medskip

    \emph{Step 1}.
    In this step, we reduce the proof to an lc pair with standard coefficients.
    Since $X$ is of Fano type over $Z$, and since
    $-S$ is pseudo-effective over $Z$, we can run $\phi\colon X\dashrightarrow X'$ a $(-S)$-MMP  
    over $Z$ that ends with a good minimal model $X'\to Z$ of $-S'$, where $S' := \phi_* S$.    
    Let $g\colon X'\to Z'$ be the contraction induced by $-S'$, and let $T'$ be the birational transform
    of $T$ to $Z'$.  As $-S'$ is nef$/Z'$, $S'$ is the unique irreducible component of $\supp (g^* T')$.
    Then we can write
    \begin{equation}
        S' = a\cdot g^* T' \label{S-whole-fibre}
    \end{equation}
    for some $a\in \QQ^{>0}$.
    Moreover, running $\varphi\colon X'\dashrightarrow X''$ a $(- K_{X'} - t S')$-MMP 
    over $Z'$, we have that $-(K_{X''} + t S'')$
    is nef$/Z'$, where $S'' := \varphi_* S'$.
    By \eqref{S-whole-fibre}, $S'$ is not contracted by $\varphi$.
    Denote by $h\colon X'' \to Z'$ the induced contraction, then by \eqref{S-whole-fibre}, we also have
    \begin{equation}
        S'' = a\cdot h^* T'. \label{S-whole-fibre-1}
    \end{equation}
    By Lemma~\ref{Fano-type-pushdown}, $X''$ is also of Fano type over $Z'$.
    Thus, by \cite[Theorem 1.8]{B-Fano}, over the generic point of $T'$, there exists a monotonic $n$-complement
    $K_{X''} + B^+$ of $K_{X''} + t S''$ for some $n\in \NN$ depending only on $d, t$ such that
    \begin{itemize}
        \item $(X'', B^+)$ is lc,
        \item $n(K_{X''} + B^+) \sim 0/Z'$, and
        \item $t S'' \le B^+$.
    \end{itemize}
    In particular, the coefficients of $B^+$ belong to $\mb{T}_n$. 

    \medskip

    \emph{Step 2}.
    In this step, we reduce the proof to Mori fibre spaces.
    Set $\alpha := \coeff_{S''} B^+\in \mb{T}_n$, and set
    \[ C'' := B^+ - \frac{\alpha}{2} S''.  \]
    Then the coefficients of $C''$ belong to the set of rational numbers $\mb{T}_{2n}$.  Moreover, 
    over the generic point of $T'$, we also have
    \begin{itemize}
        \item $(X'', C'')$ is lc,
        \item $K_{X''} + C''\sim_{\QQ} 0/Z'$, and
        \item $\frac{t}{2} S'' \le C''$.
    \end{itemize}
    As $(X'', B^+)$ is lc, there is no lc centre of $(X'', C'')$ contained in $S''$ by \eqref{S-whole-fibre-1}
    and \cite[Lemma 2.5]{Kol_singularities_of_MMP}.  Let 
    \[ \pi\colon (X''', C''') \to (X'', C'') \] 
    be a $\QQ$-factorial dlt model of $(X'', C'')$.
    Then shrinking $Z'$ near the generic point of $T'$ if necessary, 
    we can assume that every stratum of $\floor{C'''}$ is horizontal$/Z'$.
    Let $F$ be a general fibre of $X'''\to Z'$.  Since
    \[ K_F + C'''|_F \sim_{\QQ}0, \]
    and since the coefficients of horizontal$/Z'$ irreducible components of $C'''$ belong to $\mb{T}_n$, 
    there exists an $\epsilon >0$
    such that $(F, 0)$ is $\epsilon$-lc by \cite[Lemma 2.48]{B-Fano}.
    Thus, up to shrinking $Z'$ near the generic point of $T'$, we can assume that
    $(X''', 0)$ is $\epsilon$-lc over $Z'\setminus T'$.  
    
    As $X'''$ is of Fano type over $Z'$ (by Lemma~\ref{fano-type-at-dlt-model}), 
    $K_{X'''}$ is non-pseudo-effective$/Z'$.
    Run $\psi\colon X'''\dashrightarrow \wt{X}$ a $K_{X'''}$-MMP over $Z'$
    that ends with a Mori fibre space $\wt{f}\colon \wt{X} \to \wt{Z}$ over $Z'$;
    in particular, $-K_{\wt{X}}$ is ample$/\wt{Z}$.  
    \[\xymatrix{
     & & X'''\ar[d]^{\pi}\ar@{-->}[r]^{\psi} & \wt{X}\ar[dd]^{\wt{f}} \\
    X\ar@{-->}[r]^{\phi}\ar[d]_f & X'\ar@{-->}[r]^{\varphi}\ar[d]_g & X''\ar[ld]^h & \\
    Z & Z'\ar[l] & & \wt{Z}\ar[ll]
    }\]
    Let $S'''$ be the birational transform of $S''$ to $X'''$.  
    Since every exceptional divisor of $\pi$ is horizontal$/Z'$, we also have that
    (up to shrinking $Z'$ near the generic point of $T'$)
    \begin{equation}
        S''' = a\cdot (h\circ \pi)^* T' \label{S-whole-fibre-2}
    \end{equation}
    for the $a\in \QQ^{>0}$ as in \eqref{S-whole-fibre-1}.  
    Thus, $S'''$ is not contracted by $\psi$.  Set $\wt{S} := \psi_* S'''$. 

    \medskip

    \emph{Step 3}.
    In this step, we show that the diagram obtained in Step 2 gives a tower of two Fano type fibrations, and then
    apply induction hypothesis to these Fano type fibrations.
    Let $\wt{T}$ be the image of $\wt{S}$ on $\wt{Z}$.  
    By \eqref{S-whole-fibre-2}, $\wt{T}$ is also of codimension one in $\wt{Z}$.
    Let $\wt{C}$ be the pushdown of $C'''$ to $\wt{X}$, then
    \[ K_{\wt{X}} + \wt{C} \sim_{\QQ} 0/Z'. \]
    Assume that $\dim (\wt{Z}/Z') \ge 1$.
    By adjunction for fibre spaces (see \cite[\S 3.4]{B-Fano}), we can write
    \[ K_{\wt{X}} + \wt{C} \sim_{\QQ} \wt{f}^* (K_{\wt{Z}} + C_{\wt{Z}} + \mb{N}_{\wt{Z}}), \]
    where $(\wt{Z}, C_{\wt{Z}} + \mb{N}) \to Z'$ is a Fano type generalised log Calabi-Yau fibration
    by Lemma~\ref{fano-type-on-base}.
    Moreover, by \cite[Lemma 3.7]{B-Fano}, we also have that
    \[ \coeff_{\wt{T}} C_{\wt{Z}} \ge \frac{t}{2}. \]
    On the other hand, by Lemma~\ref{fano-type-upper}, 
    $(\wt{X}, \wt{C}) \to \wt{Z}$ is also a Fano type (generalised) log Calabi-Yau fibration
    with $\coeff_{\wt{S}} \wt{C} \ge \frac{t}{2}$.
    Thus, by induction hypothesis $\mc{V}(k, t)$ for $1\le k\le d-1$, 
    \[ \sdeg \big( (\wt{S})^{\nor}/\wt{T} \big) \text{ and }\sdeg \big((\wt{T})^{\nor}/T'\big) \]
    are bounded from above depending only on $d, t$, hence
    \[ \ssdeg (S^{\nor}/Z) = \sdeg \big( (\wt{S})^{\nor}/T' \big)
    \le \sdeg \big( (\wt{S})^{\nor}/\wt{T} \big)\cdot \sdeg \big( (\wt{T})^{\nor}/T' \big) \]
    is also bounded from above depending only on $d,t$ by Lemma~\ref{mul-g-fib}.
    Therefore, we can assume that $\dim (\wt{Z}/Z') = 0$, 
    that is, the contraction $\wt{Z}\to Z'$ is birational.
    Shrinking $Z'$ near the generic point of $T'$ if necessary, we can assume further
    that $\wt{Z} \to Z'$ is an isomorphism of nonsingular surfaces.
    In particular, $-K_{\wt{X}}$ is ample$/Z'$.
    To simplify the notation, we identify $(\wt{Z}, \wt{T})$ and $(Z', T')$
    with $(Z, T)$ in the rest of the proof. 

    \medskip
    
    \emph{Step 4}.
    In this step, we modify the Fano fibration obtained in Step 3 to a relatively bounded family of couples.
    Recall that $(X''', 0)$ is $\epsilon$-lc over $Z\setminus T$.
    By \cite[Lemma 3.38]{km98}, $(\wt{X}, 0)$ is also $\epsilon$-lc over $Z\setminus T$.  
    Then $\wt{f}\colon (\wt{X}, \wt{C})\to Z$
    belongs to $\CY_{d, \epsilon, 2n}$; see Notation~\ref{notation-bnd-family}.
    Let $\CY_{d, \epsilon, 2n}^{\text{sm}}$ be the relatively bounded family 
    of couples associated to $\CY_{d, \epsilon, 2n}$
    as in Lemma~\ref{bir-to-rlbnd-generic-sm-couples}. 
    Without loss of generality, we can assume that 
    \[ \Phi\colon (\mf{X}, \mf{D}) \to \mf{B} \]
    is the universal family of $\CY_{d, \epsilon, 2n}^{\text{sm}}$, where $\mf{X}$ and $\mf{B}$
    are varieties, $\mf{D}$ is a reduced horizontal$/\mf{B}$ divisor on $\mf{X}$, 
    and $\Phi$ is a projective morphism with geometrically integral generic fibre. 
    
    By construction in the proof of Lemma~\ref{bir-to-rlbnd-generic-sm-couples},
    there exists a maximal open subset $\mf{B}^{\circ}\subseteq \mf{B}$
    such that $(\mf{X}, \mf{D})$ is log smooth over $\mf{B}^{\circ}$.  
    Moreover, by Lemma~\ref{bir-to-rlbnd-generic-sm-couples}, shrinking
    $Z$ near the generic point of $T$ if necessary, we can assume that 
    there is a morphism 
    \[ i\colon Z\to \mf{B}, \] 
    whose image intersects $\mf{B}^{\circ}$.
    Also, Lemma~\ref{bir-to-rlbnd-generic-sm-couples} shows that there is a commutative diagram
    as follows for some $(Y/Z, D)\in \CY_{d, \epsilon, 2n}^{\text{sm}}$
    satisfying all the properties listed in Lemma~\ref{bir-to-rlbnd-generic-sm-couples}.
    \[\xymatrix{
    \wt{X}\ar[rd]_{\wt{f}}\ar@{-->}[r]^{\wt{\phi}} & Y\ar[r]\ar[d] & (\mf{X}, \mf{D})\ar[d]^{\Phi} \\
     & Z\ar[r]^i & \mf{B}
    }\]
    Note that $\centre_Y \wt{S}$ may not be a divisor on $Y$, that is, $\wt{S}$ can be contracted by $\wt{\phi}$.
    If $\centre_Y \wt{S}$ is a divisor on $Y$, 
    then the centre of $\wt{S}$ on $\mf{X}$ is an irreducible component of the inverse image of $\Phi$ over $i(T)$. 
    In this case, $\ssdeg \big( (\wt{S})^{\nor}/Z\big)$
    is bounded from above depending only on $\CY_{d, \epsilon, 2n}^{\text{sm}}$, 
    hence depending only on $d, t$.
    Therefore, in the rest of the proof, we can assume that $\centre_Y \wt{S}$ is not a divisor on $Y$. 

    \medskip

    Denote by $\eta_T$ the generic point of $T$.  
    In Step 5, we treat the case when the image $i(\eta_T)$ is contained in $\mf{B}^{\circ}$;
    In Steps 6 to 9, we treat the case when $i(\eta_T)\not\in \mf{B}^{\circ}$.

    \medskip

    \emph{Step 5}.
    Assume that $i(\eta_T)\in \mf{B}^{\circ}$.
    Shrinking $Z$ near $\eta_{T}$ if necessary,
    we can assume that the whole image $i(Z)$ is contained in $\mf{B}^{\circ}$. 
    By construction of $\mf{B}^{\circ}$, $(Y, D)$ is log smooth over $Z$;
    in particular, $(Y, D)$ is a strict toroidal couple.  Write
    \[ K_Y + C_Y = (\wt{\phi}^{-1})^*(K_{\wt{X}} + \wt{C}), \]
    that is, pull back $K_{\wt{X}} + \wt{C}$ to a common resolution of $Y$ and $\wt{X}$, then push it down to $Y$.
    By construction in the proof of Lemma~\ref{bir-to-rlbnd-generic-sm-couples}, we can assume that $C_Y\le D$ 
    (up to shrinking $\mf{B}^{\circ}$ if necessary).  
    Thus, by \cite[Lemma 2.4]{birkar2024irrationality}, 
    \[ 0\le a(\wt{S}, Y, D) \le a(\wt{S}, Y, C_Y) = a(\wt{S}, \wt{X}, \wt{C}) <1, \]
    which implies that $a(\wt{S}, Y, D) = 0$ by \cite[Proposition 3.5]{birkar2024irrationality}.
    Then $\centre_Y \wt{S}$ is an lc centre of $(Y, D)$ that is horizontal$/T$.  
    As $(Y/Z, D)$ is log smooth and relatively bounded, the number of irreducible components 
    of a general fibre of $\centre_Y \wt{S} \to T$ is bounded from above depending only on $d,t$.
    Let $S_Y$ be a divisor over $Y$ extracting the valuation of $\wt{S}$.
    Then by \cite[Proposition 3.5]{birkar2024irrationality}, a general fibre of 
    \[ S_Y \to \centre_Y \wt{S} \] 
    is irreducible, so is a general fibre of $(S_Y)^{\nor} \to \centre_Y \wt{S}$.
    Since $S^{\nor}\to T$ factors as
    \[ S^{\nor} \dashrightarrow (\wt{S})^{\nor} \dashrightarrow (S_Y)^{\nor} \to \centre_Y \wt{S} \to T, \]
    where $S^{\nor} \dashrightarrow (\wt{S})^{\nor}$
    and $(\wt{S})^{\nor} \dashrightarrow (S_Y)^{\nor}$ are birational, we can conclude that 
    \[ \ssdeg (S^{\nor}/Z) = \sdeg (S^{\nor}/T) = \mf{n}\big((S_Y)^{\nor}/T\big)
    = \mf{n}\big((\centre_Y\wt{S})/T\big) \]
    is bounded from above depending only on $d,t$ by Lemmas~\ref{mul-g-fib} and \ref{num-irr-comp}.
    Denote by $\mf{Z}_{\mf{B}} \subset \mf{B}$ the proper closed subset $\mf{B}\setminus \mf{B}^{\circ}$.
    From now on, we assume that $i(\eta_T)$ is contained in $\mf{Z}_{\mf{B}}$. 

    \medskip

    \emph{Step 6}.  
    In this step, we modify the relatively bounded couple $(Y, D)\to Z$
    to a toroidal morphism with bounded general fibres.
    By Theorem~\ref{functorial-toroidalisation}, there is a commutative diagram
    \[\xymatrix{
	  (U_{\mf{X}'}\subset \mf{X}')\ar[r]^-{m_{\mf{X}}}\ar[d]^{\Phi'} & \mf{X}\ar[d]^{\Phi} \\
	  (U_{\mf{B}'}\subset \mf{B}')\ar[r]^-{m_{\mf{B}}} & \mf{B}
    }\]
    such that $m_{\mf{B}}$ and $m_{\mf{X}}$ are projective birational morphisms, 
    the embeddings on the left are strict toroidal embeddings, 
    $m_{\mf{X}}^{-1}(\mf{D})$ is contained in the toroidal boundary $\mf{D}'$ that is the reduction of 
    $\mf{X}'\setminus U_{\mf{X}'}$, $m_{\mf{B}}^{-1}(\mf{Z}_{\mf{B}})$ is contained in $\mf{B}'\setminus U_{\mf{B}'}$,
    and $\Phi'$ is a surjective, projective toroidal morphism of strict toroidal embeddings.
    Moreover, let $U$ be an open subset of $\mf{B}$ such that for every closed point $c\in U$
    and for every closed point $c'\in m_{\mf{B}}^{-1}(U)$,
    \begin{itemize}
        \item the fibre $\mf{X}_c$ of $\Phi$ over $c\in U$ and the fibre $\mf{X}'_{c'}$ of $\Phi'$ over $c'\in m_{\mf{B}}^{-1}(U)$ are irreducible and normal, and that
        \item the fibre $\mf{D}'_{c'}$ of $\mf{D}'$ over $c'\in m_{\mf{B}}^{-1}(U)$ is reduced of pure dimension $\dim \mf{X}'_{c'}-1$.
    \end{itemize}
    Then we can assume that the open subset $U_{\mf{B}} \subset \mf{B}$
    in Theorem~\ref{functorial-toroidalisation} (ii) is contained in $U\cap \mf{B}^{\circ}$.
    We can also assume that $i(Z)$ intersects the open subset $U_{\mf{B}}\subset \mf{B}$;
    otherwise, we make a suitable stratification of $\CY_{d, \epsilon, 2n}^{\text{sm}}$ 
    and do induction on $\dim \mf{B}$.

    Denote by $\mu$ the morphism $Y\to Z$.
    By Theorem~\ref{functorial-toroidalisation} (ii), up to shrinking $Z$
    near $\eta_{T}$, there is an induced morphism $i'\colon Z\to \mf{B}'$.
    Let $Y'$ be the normalisation of the main component of $\mf{X}'\times_{\mf{B}'} Z$ 
    (where the fibre product is formed via $i'$),
    and let $U_{Y'}\subset Y'$ be the inverse image of $U_{\mf{X}'}$ in $Y'$.
    Similarly, let $U_{Z} \subset Z$ be the inverse image of $U_{\mf{B}'}$ in $Z$ under $i'$.

    Let $m_Y$ be the induced birational morphism $Y'\to Y$; cf. notation in Theorem~\ref{functorial-toroidalisation} (ii).
    Denote by $\nu_{Y'}$ the induced morphism $Y'\to \mf{X}'$ and by 
    $\mu'$ the morphism $Y'\to Z$.
    By \cite[Lemma 3.7]{birkar2024irrationality}, we have
    \[ U_{Y'} = \nu_{Y'}^{-1}(U_{\mf{X}'}) = \nu_{Y'}^{-1}(U_{\mf{X}'}) \cap (\mu')^{-1}(U_{Z}). \]
    By Theorem~\ref{functorial-toroidalisation} (ii), 
    the induced morphism 
    \[ \mu'\colon (U_{Y'}\subset Y') \to (U_{Z}\subset Z) \]
    is also a surjective, projective toroidal morphism between strict toroidal embeddings.  
    We include the commutative diagram here for convenience.
    \[\xymatrix{
    (U_{Y'}\subset Y')\ar[rrr]^-{m_{Y}}\ar[ddd]_{\mu'}\ar[rd]^{\nu_{Y'}} & & & Y\ar[ddd]^{\mu}\ar[ld]  \\
     & (U_{\mf{X}'}\subset \mf{X}')\ar[d]_{\Phi'}\ar[r]^-{m_{\mf{X}}} & \mf{X}\ar[d]^{\Phi} &  \\
     & (U_{\mf{B}'}\subset \mf{B}')\ar[r]^-{m_{\mf{B}}} & \mf{B} &  \\
    (U_Z \subset Z)\ar[rrr]^-{m_Z}\ar[ru]^{i'} & & & Z\ar[lu]_i 
    }\]
    Recall that there is a birational map $\wt{\phi}\colon \wt{X}\dashrightarrow Y$ over $Z$
    by Lemma~\ref{bir-to-rlbnd-generic-sm-couples} such that
    $\wt{\phi}$ does not contract any horizontal$/Z$ divisors,
    and all the horizontal$/Z$ exceptional divisors of $\wt{\phi}^{-1}$ are contained in $\supp D$.
    Furthermore, by Theorem~\ref{functorial-toroidalisation} (ii),
    $m_Y^{-1}(\supp D)$ is contained in $Y'\setminus U_{Y'}$. 
    
    \medskip

    \emph{Step 7}.
    In this step, we show that the centre of $\wt{S}$ on $Y'$ is an lc centre of 
    the strict toroidal couple $(Y', D')$.
    By the proof of Theorem~\ref{functorial-toroidalisation} (ii), 
    we can assume that $T$ is equal to the toroidal boundary $Z\setminus U_Z$;
    more specifically, $Z$ is a nonsingular surface, and $T$ is a nonsingular prime divisor on $Z$.
    Moreover, by \cite[Lemma 3.7]{birkar2024irrationality},
    the support of the inverse image of $\mu'$ over $T$
    is contained in the toroidal boundary of $(U_{Y'}\subset Y')$.  
    Denote by $D'$ the reduced divisor that is the reduction of $Y'\setminus U_{Y'}$.
    Define the divisor $D_{Y'}$ by 
    \begin{equation}
        K_{Y'} + D_{Y'} = (\wt{\phi}^{-1}\circ m_Y)^* (K_{\wt{X}} + \wt{C}). \label{pullback-boundary-1}
    \end{equation}
    Similarly, denote by $D_Y$ the divisor defined by 
    \begin{equation}
        K_Y + D_Y = (\wt{\phi}^{-1})^*(K_{\wt{X}} + \wt{C}). \label{pullback-boundary-2}
    \end{equation}
    Note that $D_{Y'}$ and $D_Y$ may have components with negative coefficients.
    Since $Y$ is generically smooth over $Z$, over the generic point of $Z$, we can write
    \begin{equation}
        K_{Y'} + R = m_Y^* K_Y, \label{pullback-canonical-divisor}
    \end{equation}
    where $R\le 0$ supports in the exceptional locus of $m_Y$.
    By \eqref{pullback-boundary-1}, \eqref{pullback-boundary-2} and \eqref{pullback-canonical-divisor}, we can write
    \[ D_{Y'} = m_Y^* D_Y + R\le m_Y^* D_Y \]
    over the generic point of $Z$.  
    
    By construction of the divisor $D$ on $Y$ in Lemma~\ref{bir-to-rlbnd-generic-sm-couples}, 
    $\supp D_Y$ is contained in $\supp D$ over the generic point of $Z$.
    Moreover, since $m_Y^{-1}(\supp D)$ is contained in $\supp D'$ by construction, we can conclude that 
    $\supp (m_Y^* D_Y)$ is contained in $\supp D'$ over the generic point of $Z$.
    Then as $D_{Y'}\le 1$, we have $D_{Y'}\le D'$ over the generic point of $Z$.
    
    On the other hand, as the support of the inverse image of $\mu'$ over $T$ is contained in the toroidal boundary $D'$,
    shrinking $Z$ around $\eta_{T}$ if necessary, we can assume that
    $D_{Y'}\le D'$ over the whole $Z$.
    Then by \cite[Lemma 2.4]{birkar2024irrationality}, we have the relation of log discrepancies
    \[ 0\le a(\wt{S}, Y', D')\le a(\wt{S}, Y', D_{Y'}) = a(\wt{S}, \wt{X}, \wt{C}) < 1. \]
    Since $(Y', D')$ is a strict toroidal couple, 
    by \cite[Proposition 3.5]{birkar2024irrationality}, we have
    \[ a(\wt{S}, Y', D') = 0. \]
    That is, $\centre_{Y'} (\wt{S})$ is a horizontal$/T$ lc centre of 
    the strict toroidal couple $(Y', D')$. 

    \medskip

    \emph{Step 8}.
    In this step, we show that up to shrinking $Z$ near the generic point of $T$,
    for every closed point $z\in Z$, the number of irreducible components of the fibre
    $Y_z'$ is bounded from above depending only on $d,t$.
    Take a very ample$/\mf{B}'$ divisor $\mc{A}$ on $\mf{X}'$ so that
    $K_{\mf{X}'} + \mf{D}' + \mc{A}$ is ample$/\mf{B}'$.
    Denote by $A$ the pullback of $\mc{A}$ to $Y'$ which is an ample$/Z$ Cartier divisor on $Y'$.
    By the cone theorem (cf. \cite[Theorem 1.19]{Fujino-slc}),
    replacing $\mc{A}$ by $3(d+2)\mc{A}$ if necessary, we can assume that 
    the Cartier divisor $K_{Y'} + D' + A$ is also relatively ample over $Z$.
    Recall that every fibre of $\Phi'$ over $m_{\mf{B}}^{-1}(U_{\mf{B}})$ is irreducible and normal, 
    hence over the generic point of $Z$, the normal scheme $Y'$
    is isomorphic to $\mf{X}'\times_{\mf{B}'} Z$ (which is the fibre product via $i'\colon Z\to \mf{B}'$).  Thus,
    the relative degree of $A$ over $Z$ with respect to $A$,
    \[ \deg_{A/Z} A, \]
    is bounded from above depending only on $\CY_{d, \epsilon, 2n}^{\text{sm}}$,
    whence only on $d,t$.  
    
    Shrinking $Z$ near $\eta_{T}$ if necessary,
    we can assume that $\mu'\colon Y'\to Z$ is a flat morphism of relative dimension $d$.
    By \cite[Corollary I.3.15]{Kol-rc}, 
    \[ (\mu'\colon [Y']\to Z) \]
    is a well-defined family of algebraic cycles over $Z$,
    where $[Y']$ is the fundamental cycle of $Y'$.
    Pick an arbitrary closed point $z\in Z$.
    Let $(\mu')^{[-1]}(z)$ be the cycle theoretic fibre of $\mu'\colon Y'\to Z$;
    see \cite[Definitions I.3.8, I.3.9 and I.3.10]{Kol-rc}.  Write
    \[ (\mu')^{[-1]}(z) = n_1 [Y_1] + \cdots + n_s [Y_s], \]
    where $n_i\in \NN$ and $\Set{Y_i}_{i=1}^s$ is the set of irreducible components of the scheme-theoretic fibre $(\mu')^{-1}(z)$.
    By \cite[Proposition I.3.12]{Kol-rc}, the intersection number (see \cite[(3.1.5)]{Kol-rc})
    \[ \big((\mu')^{[-1]}(z) \cdot A^d\big) \]
    is constant for all closed points $z\in Z$.  
    Moreover, let $q$ be a general closed point of $Z$, and let $Y_q'$
    be the fibre of $Y'\to Z$ over $q$ which is a normal projective variety.  It is clear that
    \[ \big((\mu')^{[-1]}(q) \cdot A^d\big) = ([Y_q']\cdot A^d) = (A|_{Y_q'})^d = \deg_{A/Z} A, \]
    which is bounded from above depending only on $d,t$.  Thus, the intersection number
    \[ \big((\mu')^{[-1]}(z)\cdot A^d\big) = n_1 ([Y_1]\cdot A^d) + \cdots + n_s ([Y_s]\cdot A^d) \]
    is bounded from above depending only on $d,t$.  
    Recall that $A$ is an ample$/Z$ Cartier divisor on $Y'$, hence
    every $([Y_i]\cdot A^d)$ is a positive integer.  Therefore, the natural number $s$
    is bounded from above depending only on $d,t$, that is, $\mf{n}(Y_z')$ is bounded from above
    depending only on $d,t$ for every closed point $z\in Z$. 

    \medskip

    \emph{Step 9}.  
    Now we finish the proof for the case when $i(\eta_T)\not\in \mf{B}^{\circ}$.
    Let $C$ be the centre (equipped with reduced scheme structure) of $\wt{S}$ on $Y'$, which is horizontal$/T$.
    Let $L$ be a general hyperplane section of $Z$, which is a nonsingular curve.  
    Shrinking $Z, L$ if necessary, we can assume that $L$ intersects $T$
    at a single general closed point $z$ of $T$.  Denote by
    $Y_L'$ the fibre product $L\times_{Z} Y'$, which is also normal (and irreducible) by Bertini's theorem on normality;
    see \cite[Corollary 3.4.9]{normal-bertini}.  
    Note that $Y_L'$ is a Cartier divisor on $Y'$.
    By \cite[Proposition 4.5]{Kol_singularities_of_MMP}, we can write
    \[ K_{Y_L'} + D_L' \sim (K_{Y'} + D' + Y_L')|_{Y_L'}. \]
    In particular, $D_L'$ is a reduced divisor, $(Y_L', D_L')$ is an lc pair, and
    $K_{Y_L'} + D_L'$ is also a Cartier divisor on $Y_L'$.
    
    Let $C_z$ be the fibre of $C$ over $z$, then $C_z$
    is a union of lc centres of $(Y_L', D_L')$ over $z$.
    Let $V'$ be an irreducible component of $C_z$,
    and let $G'$ be the reduction of an irreducible component of the fibre $Y_z'$ containing $V'$.
    Let $G$ be the normalisation of $G'$.
    By \cite[\S 4.1]{Kol_singularities_of_MMP}, we can write
    \[ K_G + D_G \sim (K_{Y_L'} + D_L')|_G, \]
    which is a Cartier divisor.
    Note that $(G, D_G)$ is also an lc pair; cf. \cite[Theorem 1.1]{FH-adjunction}.
    
    Denote by $A_G$ the pullback of the ample, base point free, Cartier divisor $A$ to $G$.
    By the same argument in Step 8, the volume of the ample Cartier divisor $K_G + D_G + A_G$
    is bounded from above depending only on $d,t$ (cf. \cite[Proposition VI.2.10]{Kol-rc}),
    so the volume of $K_G + D_G + A_G$ takes only finitely many values.
    Changing the ample, base point free, Cartier divisor $A_G$ linearly, we can assume that $(G, D_G + A_G)$ is lc.
    Then \cite[Theorem 1.1]{HMX14} shows that $(G, D_G + A_G)$
    belongs to a bounded set of pairs (depending only on $d,t$). 
    In particular, the number of lc centres of $(G, D_G + A_G)$ is bounded from above
    depending only on $d,t$,
    so is the number of lc centres of $(G, D_G)$.

    On the other hand, take a general Cartier divisor $P\ge 0$ on $Y_L'$ containing $V'$ and avoiding other 
    lc centres of $(Y_L', D_L')$.
    For any small $\beta >0$, $(Y_L', D_L'+ \beta P)$ is not lc near $V'$, hence
    applying \cite[Theorem 1.1]{FH-adjunction}, we see that 
    $(G, D_G)$ has an lc centre $V$ mapping onto $V'$.  Thus,
    the number of irreducible components of $C_z$ contained in $G'$ is bounded from above
    depending only on $d,t$.
    As $\mf{n}(Y_z')$ is also bounded from above depending only on $d,t$ by Step 8, 
    we see that $\mf{n}(C_z)$ is bounded from above depending only on $d,t$.

    Let $S_{Y'}$ be a divisor over $Y'$ extracting the valuation of $\wt{S}$.
    Then by \cite[Proposition 3.5]{birkar2024irrationality}, a general fibre of
    $S_{Y'}\to C$ is irreducible, so is a general fibre of $(S_{Y'})^{\nor}\to C$.
    Since $S^{\nor}\to T$ has the factorisation
    \[ S^{\nor} \dashrightarrow (\wt{S})^{\nor} \dashrightarrow (S_{Y'})^{\nor} \to C \to T, \]
    where the first two maps are birational, we can conclude that (by Lemmas~\ref{mul-g-fib} and \ref{num-irr-comp})
    \[ \ssdeg (S^{\nor}/Z) = \sdeg \big( (S_{Y'})^{\nor}/T \big) = \mf{n}(C_z). \]
    Therefore, $\ssdeg( S^{\nor}/Z )$ is bounded from above depending only on $d,t$.
\end{proof}

%%%%%%%%%%%%%%%%%%%%%%%%%%%%%%%%%%%%
%%%%%%%%%%%%%%%%%%%%%%%%%%%%%%%%%%%%

After finishing this paper, we figured out an alternative argument for Steps 4 to 9 in the proof of 
Theorem~\ref{vertical-bnd}, which does not depend on toroidal geometry.
We sketch the main ideas of this alternative proof as following.
Note that replacing the proof involving toroidal geometry in Theorem~\ref{vertical-bnd}
by a detailed explanation of this alternative will not reduce the length of this paper.

\begin{theorem}
    Let $\epsilon, t\in \RR$ and $n,d\in \NN$.
    Let $(X, B)\to Z$ be a Fano type fibration of relative dimension $d$ satisfying that
    \begin{itemize}
        \item $-K_X$ is ample$/Z$,
        \item $n(K_X + B)\sim 0/Z$, and
        \item $X$ is $\epsilon$-lc over the generic point of $Z$.        
    \end{itemize}
    Let $S$ be a prime divisor on $X$ such that
    \begin{itemize}
        \item the coefficient of $S$ in $B$ is $\ge t$, and
        \item the image of $S$ in $Z$ is a prime divisor.
    \end{itemize}
    Assume $\mc{H}(k,t)$ for every $1\le k\le d-1$.
    Then $\ssdeg(S/Z)$ is bounded from above depending only on $n,d,\epsilon,t$.
\end{theorem}

\begin{proof}
    Denote by $T$ the image of $S$ in $Z$.
    By Bertini's theorem, and by shrinking $Z$ around the generic point of $T$, 
    we can assume that $Z$ is a nonsingular surface and that 
    $T$ is a nonsingular curve.  

    \medskip

    \emph{Step 1}.
    Let $A$ be a prime divisor on $X$ that is very ample over the generic point of $Z$.
    Let $\pi\colon W\to X$ be a log resolution of $(X, B + A)$, 
    and let $h\colon W\to Z$ be the induced morphism.
    Let $B_h$ be the horizontal$/Z$ part of $B$.  
    Let $B_h^{\sim}$ (respectively, $A^{\sim}$) be the birational transform of $B_h$
    (respectively, of $A$) to $W$.
    Denote by $\ex(\pi)$ the reduced exceptional divisor of $\pi$.  Let
    \[ \Delta_W := B_h^{\sim} + \ex(\pi) + \supp h^*T. \]
    By \cite{BCHM}, we can run $\phi\colon W\dashrightarrow Y'$
    an MMP$/Z$ on the divisor $(K_W + \Delta_W + (1/2)A^{\sim})$ that ends with a good minimal model
    $Y'\to Z$ for $(K_W + \Delta_W + (1/2)A^{\sim})$.  Let
    \[ \Delta_{Y'} := \phi_*\Delta_W \,\text{ and }\, A' := \phi_*A^{\sim}. \]
    As $K_{Y'} + \Delta_{Y'} + (1/2)A'$ is nef and big over $Z$, 
    there is a birational contraction $\varphi\colon Y'\to Y$ over $Z$ such that 
    $K_{Y'} + \Delta_{Y'} + (1/2)A'$ is pullback of an ample$/Z$ divisor on $Y$.  Let
    \[ \Delta := \varphi_*\Delta_{Y'} \,\text{ and }\, A_Y := \varphi_*A'. \]
    Then $(Y, \Delta + (1/2)A_Y)$ is an lc pair such that $K_Y + \Delta + (1/2)A_Y$ is ample$/Z$, and
    the pairs $(X, B + (1/2)A)$ and $(Y, \Delta + (1/2)A_Y)$ are isomorphic over the generic point of $Z$.
    Shrinking $Z$ near the generic point of $T$, we can assume that $K_Y+\Delta\sim_{\QQ}0/Z$.
    Hence $A_Y$ is an ample$/Z$ prime divisor on $Y$.  Moreover, 
    we can deduce that $Y$ is of Fano type over $Z$.
    By construction, the coefficients of $\Delta$ belong to the finite set $\mb T_n$.
    Then by boundedness of complements \cite[Theorem 1.8]{B-Fano}, up to shrinking $Z$ near the generic point
    of $T$, there is a bounded $\ell\in \NN$ such that $\ell (K_Y + \Delta)\sim 0/Z$.    

    \medskip

    \emph{Step 2}.
    Denote by $g$ the contraction $Y\to Z$.  By construction, $\red (g^* T) \subseteq \floor{\Delta_Y}$.
    We claim that for every horizontal$/T$ irreducible component $P$ of $g^*T$, $\sdeg(P^{\nor}/T)$
    is bounded from above, where $P^{\nor}$ is the normalisation of $P$.
    Shrinking $Z$ near the generic point of $T$, we can assume that $g\colon Y\to Z$ is flat, hence
    for every $z\in Z$, $\vol(A_Y|_{Y_z})$ is equal to $\vol(A|_{X_u})$,
    where $u\in Z$ is a general closed point.  Thus, $\vol(A_Y|_{Y_z})$ belongs to a finite subset of natural numbers.
    Let $H\subset Z$ be a general hyperplane section and $G := g^* H$.
    By Bertini's Theorem, $G$ is a normal variety.  Let
    \[ \Delta_G := \Delta|_G \,\text{ and }\, A_G := A_Y|_G. \]
    As $(Y, \Delta_Y + (1/2)A_Y)$ is lc, 
    $\Set{(G, \Delta_G), A_G}$ is an lc stable Calabi-Yau fibration over $H$; cf. \cite{B-moduli}.
    By construction, $\ell (K_G + \Delta_G)\sim 0/H$.
    
    Let $z$ be one of the closed points $T\cap H$.  Taking $H$ generally,
    we can assume that $z$ is a general closed point of $T$.
    Denote by $R$ the reduction of the fibre $G_z = Y_z$.
    Then $R$ is contained in $\floor{\Delta_G}$.  
    Denote by $\nu\colon R^{\nor}\to R$ the normalisation of $R$.  Write
    \[ K_{R^{\nor}} + \Delta_{R^{\nor}} = \nu^*(K_G + \Delta_G), \]
    where the coefficients of $\Delta_{R^{\nor}}$ belong to 
    a finite set of rational numbers depending only on $\ell$.  Let
    \[ A_{R^{\nor}} := \nu^* A_G. \]
    By adjunction, the coefficients of irreducible components of $A_{R^{\nor}}$
    in $\Delta_{R^{\nor}}$ also belong to a finite set of rational numbers depending only on $\ell$.
    
    Let $R_1,\dots, R_r$ be irreducible components of $R$,
    and let $R_i^{\nor}$ be the normalisation of $R_i$.
    Let $\Delta_{R_i^{\nor}}$ (respectively, $A_{R_i^{\nor}}$) be the restriction of $\Delta_{R^{\nor}}$
    (respectively, of $A_{R^{\nor}}$) to $R_i^{\nor}$.
    Then every $(R_i^{\nor}, \Delta_{R_i^{\nor}}), A_{R_i^{\nor}}$ 
    is a stable log Calabi-Yau pair.
    By \cite[Theorem 1.3]{HMX14-ACC}, $\vol(A_{R_i^{\nor}})$ belongs to a DCC set of rational numbers, 
    and there exists a $\delta>0$ such that
    \[ \vol(A_{R_i^{\nor}}) \ge \delta. \]
    Write $g^*z = \sum_{i=1}^r m_i R_i$, where $m_i\in \NN$.  Then we get
    \[ \vol(A)_{/Z} = \vol(A_Y|_{Y_z}) = \sum_{i=1}^r m_i\vol(A_{R_i}) = \sum_{i=1}^r m_i\vol(A_{R_i^{\nor}}). \]
    Hence $r$ is bounded from above depending only on $\ell, d$.
    Then $\sdeg (P^{\nor}/T)$ is bounded from above depending only on $n,d$.

    \medskip

    \emph{Step 3}.
    We can assume that $\centre_Y S$ is not a divisor on $Y$.
    Let $P$ be an irreducible component of $g^*T$ that contains $\centre_Y S$.
    Let $Y''\to Y$ be a birational contraction that extracts only the geometric valuation of $S$, and write
    $K_{Y''} + \Delta_{Y''}$ as the pullback of $K_Y + \Delta_Y$.
    Denote by $P'', S''$ the centres of $P, S$ on $Y''$.
    It is easy to see that $P''\cap S''$ has a horizontal$/T$ irreducible component $J''$ such that 
    $J''$ is of codimension one in $P''$ and $S''$. 
    Since $(Y, \Delta_Y)$ is dlt, $P$ is normal.  Write
    \[ K_P + \Delta_P = (K_Y + \Delta_Y)|_P. \]
    Then $J''$ is a non-canonical place of $(P, \Delta_P)$ such that $a(J'', P, \Delta_P)\le 1-t$.
    Let $P\to T''\to T$ be the Stein factorisation of $P\to T$.
    By Step 2, $\deg(T''/T) = \sdeg(P/T)\le \sdeg(P^{\nor}/T)$ is bounded from above depending only on $n,d$.
    Moreover, $(P, \Delta_P)\to T''$ is also a log Calabi-Yau fibration.
    By assumption $\mc{H}(d-1, t)$, we see that $\sdeg((J'')^{\nor}/T'')$ is bounded from above depending
    only on $d, t$.  By Lemma~\ref{mul-g-fib},
    we can conclude that $\sdeg ((J'')^{\nor}/T)$ is bounded from above depending only on $n,d,t$.

    \medskip

    \emph{Step 4}.
    Let $\pi\colon (S'')^{\nor}\to S''$ be the normalisation of $S''$.  
    As $\coeff_{S''}\Delta_{Y''}\ge t$, it can be shown that
    there is a closed subset $U''\subset S''$ satisfying $\codim (U'', S'')\ge 2$
    such that for every closed point $s\in S''\setminus U''$, 
    the cardinality of $\pi^{-1}(s)$ is bounded from above depending only on $t$.  
    Let $I\subset (S'')^{\nor}$ be a prime divisor mapped onto $J''$ under $\pi$.
    Then $\deg (I/J'')$ is bounded from above depending only on $t$.  
    Denote by $I^{\nor}\to I$ the normalisation of $I$.  
    By Lemma~\ref{mul-g-fib},
    $\sdeg (I^{\nor}/T)$ is bounded from above depending only on $n,d,t$.
    By the same argument of \cite[Lemma 3.3]{B-moduli}, 
    we can conclude that $\sdeg (S^{\nor}/T) = \sdeg ((S'')^{\nor}/T)$ 
    is bounded from above depending only on $n, d, \epsilon, t$.
\end{proof}

%%%%%%%%%%%%%%%%%%%%%%%%%%%%%%%%%%%%
%%%%%%%%%%%%%%%%%%%%%%%%%%%%%%%%%%%%

\section{Boundedness of S-degree of divisors on log Calabi-Yau fibrations}\label{pf-main-results}

We prove our main result Theorem~\ref{dim-d-t} in this section. 

\subsection{Reduction to klt fibrations}\label{red-to-klt-S-d-t}

In this subsection, we reduce the proof of $\mc{H}(d,t)$
to the study of generalised log Calabi-Yau fibrations $(X/Z, B + \mb{M})$,
where $(X, B + \mb{M})$ is a g-klt g-pair.
This enables us to reduced the proof of $\mc{H}(d,t)$ further
to the consideration of Fano type fibrations; see \S\ref{main-proof-S-d-t}.

To make our statements more concise, we introduce several notation as follows.

\begin{definitionnotation}
	Let $d\in \NN$ and $t\in \RR^{>0}$.  
    \begin{itemize}
        \item [(1)] Denote by $\CY(d, t)$ 
	                the set of data
                	\[ \Set{ (X/Z, B + \mb{M}), S }, \]
	                where
	                \begin{itemize}
	                	\item [(a)] $X\to Z$ is a contraction with $\dim (X/Z) = d$,
	                	\item [(b)] $(X, B + \mb{M})\to Z$ is a generalised log Calabi-Yau fibration, and
	                	\item [(c)] $S$ is a horizontal$/Z$ irreducible component of $B$ with $\coeff_S B \ge t$.
	                \end{itemize}
        \item [(2)] Let $\CY_{\cur} (d, t)$ be the subset of $\CY (d, t)$ such that for every $\Set{(X/Z, B + \mb{M}), S}$ in $\CY_{\cur}(d, t)$, the base $Z$ is a nonsingular curve.
        \item [(3)] Let $\Spair (d, t)$ be the subset of $\CY_{\cur}(d, t)$
	                such that for every $\Set{(X/Z, B + \mb{M}), S}$ in $\Spair (d, t)$, the g-pair $(X, B + \mb{M})$ is a \emph{$\QQ$-factorial} g-klt g-pair.
    \end{itemize}    
    In particular, we have the inclusions 
    \[ \Spair (d,t) \subset \CY_{\cur} (d,t) \subset \CY (d,t). \]
\end{definitionnotation}

%%%%%%%%%%%%%%%%%%%%%%%%%%%%%%%%%%%%%%%%%

First we show that to prove $\mc{H}(d, t)$, we can always assume that 
the base variety of every fibration in $\CY (d, t)$ is a smooth curve.

\begin{lemma}\label{red-to-base-curve}
    $\mc{H}(d, t)$ holds for $\CY (d, t)$ if and only if $\mc{H}(d, t)$
    holds for $\CY_{\cur}(d, t)$.
\end{lemma}

\begin{proof}
    Pick an arbitrary $\Set{ (X/Z, B + \mb{M}), S }$ in $\CY (d, t)$.
    Assume $\dim Z\ge 2$, then the result follows exactly from
    the same proof of Lemma~\ref{fano-type-induction}. 
\end{proof}

%%%%%%%%%%%%%%%%%%%%%%%%%%%%%%%%%%%%%%%%%%%

The following result shows that for Theorem~\ref{dim-d-t},
it suffices to prove $\mc{H}(d, t)$ for $\Spair (d, t)$.

\begin{lemma}\label{red-to-Q-fac-klt}
    Assume that $\mc{H} (k, t)$ holds for $\CY(k, t)$ for every $1\le k \le d-1$.
    Then $\mc{H}(d, t)$ holds for $\CY (d, t)$ if and only if $\mc{H}(d, t)$ holds
    for $\Spair (d, t)$.
\end{lemma}

\begin{proof} 
    By Lemma~\ref{red-to-base-curve}, it suffices to consider $\mc{H}(d, t)$
    for $\CY_{\cur}(d,t)$.  Assuming $\mc{H}(d, t)$ holds for $\Spair (d, t)$,
    we show that $\mc{H}(d, t)$ holds for $\CY_{\cur}(d, t)$. 

    \medskip
    
    \emph{Step 1}.
    Let $\Set{(X/Z, B+\mb{M}), S}$ be an arbitrary member in $\CY_{\cur} (d, t)$.  
    Set $\alpha := \coeff_S B$.
    Taking a $\QQ$-factorial g-dlt model of $(X, B + \mb{M})$, we can assume that $(X, B + \mb{M})$
    is a $\QQ$-factorial g-dlt g-pair.  In particular, $X$ has klt singularities.
    By \cite[Lemma 4.4]{B-Zhang}, as $-S$ is non-pseudo-effective$/Z$, 
    we can run $\phi\colon X\dashrightarrow X'/Z$ an MMP over $Z$
    on the divisor
    \[ K_X + (B - \alpha S) + \mb{M}_X \sim_{\RR} -\alpha S/Z, \]
    which ends with a Mori fibre space $g\colon X' \to Z'$ of $S$, so $S':= \phi_*S$ is ample$/Z'$.
    By Lemma~\ref{CY-log-discre-no-change}, 
    the coefficient of $S'$ in $B' := \phi_* B$ is also equal to $\alpha$.

    If $\dim (X'/Z')\le d-1$, then $\sdeg \big((S')^{\nor}/Z'\big)$ is bounded from above depending only on $d, t$
    by assumption, hence by Lemmas~\ref{mul-g-fib} and \ref{num-irr-comp}, 
    \[ \sdeg (S^{\nor}/Z) = \sdeg \big( (S')^{\nor}/Z \big) \le \sdeg \big( (S')^{\nor}/Z' \big). \]
    Thus, we can assume that $Z'\to Z$ is an isomorphism of normal curves. 
    Then $S'$ is ample$/Z$. 

    \medskip

    \emph{Step 2}.
    Since $(X, (B - \alpha S) + \mb{M})$ is a $\QQ$-factorial g-dlt g-pair, 
    \[ (X', (B' - \alpha S') + \mb{M}) \]
    is also a $\QQ$-factorial g-dlt g-pair; see \cite[\S2.13 (2)]{B-Fano}.  
    Pick $0< \epsilon < 1$ and write
    \[ B' = (1-\epsilon)(B' - \alpha S') + \bigg(\epsilon (B' - \alpha S') + \frac{\alpha}{2} S'\bigg) + \frac{\alpha}{2} S'.  \]
    For sufficiently small $0< \epsilon \ll 1$, the divisor 
    $\epsilon (B' - \alpha S') + (\alpha/2)S'$
    is also ample$/Z$.  Let $H'$ be a general effective ample$/Z$ divisor such that
    \begin{equation}
        \epsilon (B' - \alpha S') + \frac{\alpha}{2}S' \sim_{\RR} H'/Z. \label{linear-H}
    \end{equation}
    Set
    \[ B_{\epsilon}' := (1-\epsilon) (B' - \alpha S'). \]
    As $(X', (B' - \alpha S') + \mb{M})$ is $\QQ$-factorial and g-dlt, $(X', B_{\epsilon}' + \mb{M})$
    is a $\QQ$-factorial g-klt g-pair.  On the other hand, since
    $(X', B_{\epsilon}' + \epsilon (B'-\alpha S') + \alpha S' + \mb{M})$ is g-lc 
    by Lemma~\ref{CY-log-discre-no-change}, the g-pair $(X', B_{\epsilon}' + \alpha S' + \mb{M})$
    is also g-lc, hence
    \[ (X', B_{\epsilon}' + \frac{\alpha}{2} S' + \mb{M}) \]
    is a $\QQ$-factorial g-klt g-pair; cf. \cite[Corollary 2.35 (5)]{km98}.
    Thus, as $H'$ is general, 
    \[ (X', C' + \mb{M}) := (X', B_{\epsilon}' + H' + \frac{\alpha}{2} S' + \mb{M}) \]
    is also a $\QQ$-factorial g-klt g-pair.  Moreover, by~\eqref{linear-H}, 
    $(X', C' + \mb{M}) \to Z$ is also a generalised log Calabi-Yau fibration.
    Therefore, up to replacing $t$ by $\frac{t}{2}$, it suffices to 
    consider the data $\Set{(X'/Z, C' + \mb{M}), S'}$,
    which belongs to $\Spair (d, t)$.
\end{proof}

%%%%%%%%%%%%%%%%%%%%%%%%%%%%%%%%%%%%%%%%%%%%%%%%%%%%%%%%%%

\subsection{Boundedness of S-degree on Fano type fibrations}\label{bnd-fano-type-S-d-t}

We show in this subsection that Theorem~\ref{dim-d-t} holds for horizontal divisors
on Fano type fibrations.

\begin{theorem}\label{Fano-type-lCY-fib}
    Let $d\in \NN$ and let $t\in \RR^{>0}$.  Let 
    \[ f\colon (X, B + \mb{M})\to Z \] 
    be a Fano type generalised log Calabi-Yau fibration 
	of relative dimension $d$.  Let $S$ be a horizontal$/Z$ irreducible component of $B$ 
	whose coefficient in $B$ is $\ge t$.
	Then $\sdeg (S^{\nor}/Z)$ is bounded from above depending only on $d, t$,
    where $S^{\nor}$ is the normalisation of $S$.
\end{theorem}

\begin{proof}
    By Lemma~\ref{fano-type-induction}, we can assume that $Z$ is a smooth curve.
    By Lemma~\ref{base-case-r-dim=1}, the result holds for $d=1$.
    We prove the result by induction on $d$, that is, 
    we assume that the result holds for every Fano type fibration of relative dimension $\le d-1$. 

    \medskip
    
    \emph{Step 1}.  
    By Lemma~\ref{fano-type-at-dlt-model},
    we can assume that $(X, B + \mb{M})$
    is a $\QQ$-factorial g-dlt g-pair.
    Set $\alpha := \coeff_S B \ge t$.  Write
    \[ K_X + (B - \alpha S) + \mb{M}_X \sim_{\RR} - \alpha S/Z. \]
    By \cite[Lemma 4.4]{B-Zhang}, as $-S$ is non-pseudo-effective$/Z$, 
    we can run $\phi\colon X\dashrightarrow X'$ a $(-S)$-MMP over $Z$
    that ends with a Mori fibre space $X'\to Z'$ of $S$; in particular,
    $S':=\phi_* S$ is ample$/Z'$, hence $S'$ is horizontal$/Z'$.
    Moreover, by Lemma~\ref{Fano-type-pushdown} and Lemma~\ref{fano-type-upper}, 
    $X'$ is also of Fano type over $Z$, hence of Fano type over $Z'$.
    Thus, $(X', B' + \mb{M})$ is a Fano type fibration over $Z'$, where $B' := \phi_* B$.

    Note that $Z'\to Z$ is a contraction of normal varieties, so $\mf{n}(Z'/Z) = 1$.
    If $\dim (Z'/Z) \ge 1$, then 
    by induction hypothesis on the relative dimension $\dim (X'/Z')$, 
    \[ \sdeg \big((S')^{\nor}/Z'\big) \]
    is bounded from above depending only on $d,t$.  By Lemmas~\ref{mul-g-fib} and \ref{num-irr-comp},
    \[ \sdeg (S^{\nor}/Z) = \sdeg \big( (S')^{\nor}/Z \big) \le \mf{n}(Z'/Z)\cdot 
    \mf{n}\big( (S')^{\nor}/Z' \big) = \sdeg \big( (S')^{\nor}/Z' \big) \]
    is also bounded from above depending only on $d,t$.  Thus, 
    we can assume that the contraction $Z'\to Z$ is an isomorphism of smooth curves.

    \medskip

    \emph{Step 2}.  Decreasing $t$ if necessary, we can assume that $t\in \QQ^{>0}$.
    Since $X'$ is of Fano type over $Z$, by \cite{BCHM}, we can run an MMP$/Z$ on the divisor 
    \[ -K_{X'} - t S' \sim_{\RR} (B'- t S') + \mb{M}_X, \]
    which is pseudo-effective$/Z$.  Denote by $\varphi\colon X'\dashrightarrow X''$
    the $(-K_{X'} - tS')$-MMP over $Z$ that ends with a minimal model $X''\to Z$ 
    of the divisor $-K_{X'} - tS'$.  Set $S'' := \varphi_* S'$, then
    $- (K_{X''} + t S'')$ is nef$/Z$.  Note that $S''$ is a nonzero prime divisor on $X''$
    as $S'$ is ample$/Z$; see Lemma~\ref{no-van-big-divisor}.  Moreover,
    by Lemma~\ref{Fano-type-pushdown}, $X''$ is also of Fano type over $Z$.

    By \cite[Theorem 1.8]{B-Fano}, up to shrinking
    $Z$ near its generic point, there is a monotonic $n$-complement $K_{X''} + B^+$ of $K_{X''} + t S''$
    for some $n\in \NN$ depending on $d,t$, that is,
    \begin{itemize}
        \item $(X'', B^+)$ is lc,
        \item $n(K_{X''} + B^+) \sim 0/Z$, and
        \item $t S'' \le B^+$.
    \end{itemize}
    In particular, the coefficients of $B^+$ belong to $\mathbf{T}_n$.  
    Set $\beta := \coeff_{S''} B^+ \in \mathbf{T}_n$. 

    \medskip

    \emph{Step 3}.  Let $(X''', B''')$ be a $\QQ$-factorial dlt model of $(X'', B^+)$.
    The coefficients of $B'''$ also belong to the set $\mathbf{T}_n$.
    Denote by $S'''$ the birational transform of $S''$ to $X'''$.
    Note that $S'''$ may not be big$/Z$.  
    By \cite[Lemma 2.48]{B-Fano}, shrinking $Z$ near its generic point if necessary,
    we can assume that $(X''', 0)$ is $\epsilon$-lc for some $\epsilon>0$ depending only on $d, n$, whence only on $d,t$.
    Run $\psi\colon X''' \dashrightarrow \wt{X}$ a $K_{X'''}$-MMP over $Z$
    that ends with a Mori fibre space $\wt{f}\colon \wt{X}\to \wt{Z}$ of $K_{X'''}$. 
    \[\xymatrix{
     & & X'''\ar[d]\ar@{-->}[r]^{\psi} & \wt{X}\ar[dd]^{\wt{f}} \\
    X\ar@{-->}[r]^{\phi}\ar[d]_f & X'\ar@{-->}[r]^{\varphi} & X'' &  \\
    Z & & &  \wt{Z}\ar[lll]
    }\]
    Then $-K_{\wt{X}}$ is ample$/\wt{Z}$.
    By \cite[Lemma 3.38]{km98}, $(\wt{X}, 0)$ is also $\epsilon$-lc. 

    \medskip

    \emph{Step 4}.
    Assume that $\dim (\wt{Z}/Z) \ge 1$.  By \cite[Corollary 1.4.3]{BCHM},
    let $\pi\colon X^{\circ} \to \wt{X}$ be a birational contraction extracting only the valuation of $S$,
    where $X^{\circ}$ is $\QQ$-factorial;
    if the centre of the valuation of $S$ on $\wt{X}$ is a divisor, then $\pi$ is the identity map.
    Let $\wt{B}$ be the pushdown of $B'''$ to $\wt{X}$.  Write
    \[ K_{X^{\circ}} + B^{\circ} = \pi^*(K_{\wt{X}} + \wt{B}). \]
    By Lemma~\ref{fano-type-at-dlt-model} and Lemma~\ref{fano-type-upper}, 
    $(X^{\circ}, B^{\circ}) \to Z$ is a Fano type log Calabi-Yau fibration,
    hence $(X^{\circ}, B^{\circ}) \to \wt{Z}$ is also a Fano type fibration.
    
    Let $S^{\circ}$ be the centre of $S$ on $X^{\circ}$.
    If $S^{\circ}$ is horizontal$/\wt{Z}$, then
    \[ \sdeg \big( (S^{\circ})^{\nor}/\wt{Z}\big) \]
    is bounded from above depending only on $d,n$ (hence depending only on $d,t$)
    by applying induction hypothesis 
    to the Fano type fibration $(X^{\circ}, B^{\circ})\to \wt{Z}$.
    Again, by Lemmas~\ref{mul-g-fib} and \ref{num-irr-comp}, we have
    \[ \sdeg (S^{\nor}/Z) = \sdeg \big( (S^{\circ})^{\nor}/Z \big) \le \sdeg \big( (S^{\circ})^{\nor}/\wt{Z}\big), \]
    which shows that $\sdeg (S^{\nor}/Z)$ is bounded from above depending only on $d,t$.
    Thus, we can assume that $S^{\circ}$ is vertical$/\wt{Z}$.

    Recall that $X^{\circ}$ is also of Fano type over $\wt{Z}$.
    Run $\upsilon\colon X^{\circ}\dashrightarrow X^{\triangle}$ a $(-S^{\circ})$-MMP over $\wt{Z}$
    that ends with a good minimal model $X^{\triangle}$ of $-S^{\circ}$.
    By negativity lemma, $S^{\circ}$ is not contracted by $\upsilon$.
    Denote by $f^{\triangle}\colon X^{\triangle}\to Z^{\triangle}$ the contraction$/\wt{Z}$ induced by $-S^{\triangle}$,
    where $S^{\triangle} := \upsilon_* S^{\circ}$.
    As $S^{\triangle}$ is vertical$/\wt{Z}$, the contraction $Z^{\triangle}\to \wt{Z}$ is birational.
    \[\xymatrix{
    X^{\circ}\ar@{-->}[r]^{\upsilon}\ar[d] & X^{\triangle}\ar[d]^{f^{\triangle}} \\
    \wt{Z}\ar[d] & Z^{\triangle}\ar[l]\ar[ld] \\
    Z &
    }\]
    Denote by $T^{\triangle}$ the image of $S^{\triangle}$ in $Z^{\triangle}$,
    which is a prime divisor on $Z^{\triangle}$.  
    As $-S^{\triangle}$ is nef$/Z^{\triangle}$, $S^{\triangle}$ is 
    the unique irreducible component of $\supp ((f^{\triangle})^* T^{\triangle})$.
    Then we can write
    \[ S^{\triangle} = b\cdot (f^{\triangle})^* T^{\triangle} \]
    for some $b\in \QQ^{>0}$.  Let $B^{\triangle}$ be the pushdown of $B^{\circ}$ to $X^{\triangle}$.
    We can write
    \[ K_{X^{\triangle}} + B^{\triangle} \sim_{\RR} (f^{\triangle})^*(K_{Z^{\triangle}} + B_{Z^{\triangle}} + \mb{N}_{Z^{\triangle}}), \]
    where $(Z^{\triangle}, B_{Z^{\triangle}} + \mb{N})$ is a g-lc g-pair.
    By Lemma~\ref{fano-type-on-base}, $Z^{\triangle}$ is of Fano type over $Z$.  Then 
    \[ (Z^{\triangle}, B_{Z^{\triangle}} + \mb{N}) \to Z \]
    is also a Fano type fibration of relative dimension 
    \[ \dim (Z^{\triangle}/Z) = \dim (\wt{Z}/Z) \le d-1. \]
    By \cite[Lemma 3.7]{B-Fano}, the coefficient of $T^{\triangle}$ in $B_{Z^{\triangle}}$
    is $\ge \beta$, where $\beta = \coeff_{S''} B^+$ is $\ge t$ as in Step 2.
    By induction hypothesis, 
    \[ \sdeg \big( (T^{\triangle})^{\nor}/Z\big) \]
    is bounded from above depending only on $d,t$.  Moreover, by Theorem~\ref{vertical-bnd},
    \[ \sdeg \big( (S^{\triangle})^{\nor}/T^{\triangle} \big)  \]
    is also bounded from above depending only on $d,t$.  As $S^{\triangle}\to Z$ has the factorisation
    \[ S^{\triangle} \to T^{\triangle} \to Z, \]
    we can conclude that
    \[ \sdeg (S^{\nor}/Z) = \sdeg \big((S^{\triangle})^{\nor}/Z\big) 
    \le \sdeg \big( (S^{\triangle})^{\nor}/T^{\triangle} \big) \cdot \sdeg \big( (T^{\triangle})^{\nor}/Z\big) \]
    is bounded from above depending only on $d,t$ by Lemmas~\ref{mul-g-fib} and \ref{num-irr-comp}.
    Therefore, we can assume that $\wt{Z}\to Z$ is an isomorphism of nonsingular curves. 

    \medskip

    \emph{Step 5}.  By construction in Step 3, we have that
    \begin{itemize}
        \item $(\wt{X}, 0)$ is $\epsilon$-lc,
        \item $-K_{\wt{X}}$ is ample$/Z$,
        \item $K_{\wt{X}} + \wt{B} \sim_{\QQ} 0/Z$, and
        \item the coefficients of $\wt{B}$ belong to $\mathbf{T}_n$.
    \end{itemize}
    Note that $S'''$ may be contracted by $\psi$.  
    The fibration $(\wt{X}, \wt{B}) \to Z$ belongs to the subset $\CY_{d, \epsilon, n}^{\cur}$ 
    of $\CY_{d, \epsilon, n}$;
    see Notations~\ref{notation-bnd-family} and \ref{notation-bnd-family-curves}.
    By Lemma~\ref{bir-to-rlbnd-generic-sm-couples}, there is a relatively bounded family of
    couples $\CY_{d, \epsilon, n}^{\text{sm}}$ associated to $\CY_{d, \epsilon, n}$.
    Denote by $\mc{Q}$ the subset of $\CY_{d, \epsilon, n}^{\text{sm}}$ associated to
    the subset $\CY_{d, \epsilon, n}^{\cur}$ of $\CY_{d, \epsilon, n}$.  Then $\mc{Q}$ is also a relatively bounded
    family of couples (over nonsingular curves).
    We can assume that $\mc{Q}$ admits a single universal family
    \[ \Phi\colon (\mf{X}, \mf{D}) \to \mf{B}, \]
    where $\mf{X}$ and $\mf{B}$ are varieties, $\mf{D}$ is a reduced divisor on $\mf{X}$, and
    $\Phi$ is a projective morphism.  
    By construction in the proof of Lemma~\ref{bir-to-rlbnd-generic-sm-couples},
    there exists a maximal open subset $\mf{B}^{\circ}\subseteq \mf{B}$
    such that $(\mf{X}, \mf{D})$ is log smooth over $\mf{B}^{\circ}$; see \S\ref{log-smooth}.

    For $\wt{f}\colon (\wt{X}, \wt{B}) \to Z$, there is a couple $\mu\colon (Y, D)\to Z$ in $\mc{Q}$
    admitting a rational map $\wt{\phi}\colon \wt{X}\dashrightarrow Y$ over $Z$ 
    that satisfies all the properties listed in 
    Lemma~\ref{bir-to-rlbnd-generic-sm-couples} (up to shrinking $Z$ near its generic point).
    By Lemma~\ref{bir-universal-family}, we can assume there is a morphism $i\colon Z\to \mf{B}$.
    \[\xymatrix{
    \wt{X}\ar[rd]_{\wt{f}}\ar@{-->}[r]^-{\wt{\phi}} & Y\ar[d]^{\mu}\ar[r] & \mf{X}\ar[d]^{\Phi} \\
     & Z\ar[r]^i & \mf{B}
    }\]
    By induction on $\dim \mf{B}$, we can assume that $i(Z)$ is not entirely contained in $\mf{B}\setminus \mf{B}^{\circ}$.
    Thus, up to shrinking $Z$ near its generic point further, 
    we can assume that the whole $i(Z)$ is contained in $\mf{B}^{\circ}$.
    Then $(Y, D)\to Z$ is log smooth.  Write
    \[ K_Y + B_Y = (\wt{\phi}^{-1})^*(K_{\wt{X}} + \wt{B}), \]
    where $B_Y$ may have negative coefficients.  By Lemma~\ref{bir-to-rlbnd-generic-sm-couples}, $B_Y \le D$.
    Then by \cite[Lemma 2.4]{birkar2024irrationality}, we have
    \[ 0\le a(\wt{S}, Y, D) \le a(\wt{S}, Y, B_Y) = a(\wt{S}, \wt{X}, \wt{B}) < 1. \]
    As $(Y, D)$ is a log smooth couple, $a(\wt{S}, Y, D) = 0$ by \cite[Proposition 3.5]{birkar2024irrationality}, hence
    \[ C := \centre_Y \wt{S} \] 
    is a horizontal$/Z$ lc centre of the couple $(Y, D)$.
    Then $C$ is a horizontal$/Z$ stratum of $\floor{D}$.  Since $(Y, D)\to Z$ is relatively bounded,
    $C\to Z$ is also relatively bounded, so $\mf{n}(C/Z)$ is bounded from above 
    depending only on $\mc{Q}$, whence only on $d,t$.  
    Let $S_Y$ be a divisor over $Y$ extracting the valuation of $S$.
    Since $(Y, D)$ is log smooth, a general fibre of $S_Y \to C$ is irreducible.
    As $S^{\nor}\to Z$ has the factorisation
    \[ S^{\nor} \dashrightarrow (S_Y)^{\nor} \to C \to Z, \]
    where the first map is birational, then by Lemmas~\ref{mul-g-fib} and \ref{num-irr-comp}, we can conclude that 
    $\sdeg (S^{\nor}/Z) = \mf{n}(C/Z)$ is bounded from above depending only on $d,t$.
\end{proof}

%%%%%%%%%%%%%%%%%%%%%%%%%%%%%%%%%%%%%%%%%

\subsection{Proof of $\mc{H}(d, t)$}\label{main-proof-S-d-t}

We are now able to give the proof of Theorem~\ref{dim-d-t}.

\begin{theorem}[=Theorem~\ref{dim-d-t}]\label{pf-dim-d-t}
    Let $d\in \NN$ and let $t\in \RR^{>0}$.  Then $\mc{H}(d, t)$ holds for $\CY (d, t)$.
\end{theorem}

\begin{proof}
    We argue by induction on $d$.  If $\dim (X/Z) = 1$, then $\mc{H}(1, t)$ holds by Lemma~\ref{base-case-r-dim=1}.
    Thus, we can assume that $\mc{H}(k, t)$ holds for $\CY (k, t)$ for every $1\le k\le d-1$.
    By Lemma~\ref{red-to-base-curve}, it suffices to consider $\Set{(X/Z, B + \mb{M}), S}\in \CY (d, t)$,
    where $Z$ is a smooth curve.  Moreover, by Lemma~\ref{red-to-Q-fac-klt}, we can assume further
    that $(X, B + \mb{M})$ is a $\QQ$-factorial g-klt g-pair.

    By \cite[Lemma 4.4]{B-Zhang}, as $-S$ is non-pseudo-effective$/Z$, we can run $\phi\colon X\dashrightarrow X'$
    a $(-S)$-MMP over $Z$ that ends with a Mori fibre space $g\colon X'\to Z'$ over $Z$.
    Set $S' := \phi_* S$, then $S'$ is ample$/Z'$.  
    If $\dim (Z'/Z) \ge 1$, we can conclude that
    \[ \sdeg (S^{\nor}/Z) = \sdeg \big( (S')^{\nor}/Z \big)\le \sdeg \big( (S')^{\nor}/Z' \big)\cdot \mf{n}(Z'/Z)
    = \sdeg \big( (S')^{\nor}/Z' \big)\]
    is bounded from above depending only on $d,t$ 
    by induction hypothesis, Lemma~\ref{mul-g-fib} and Lemma~\ref{num-irr-comp}.
    Thus, we can assume that $Z'\to Z$ is an isomorphism of smooth curves.  Then $S'$ is ample$/Z$.

    Set $B' := \phi_* B$.  As $(X, B + \mb{M})$ is $\QQ$-factorial and g-klt, 
    $(X', B' + \mb{M})$ is also a $\QQ$-factorial g-klt g-pair by Lemma~\ref{CY-log-discre-no-change}.
    Since $S'$ is ample$/Z$, and since
    \[ K_{X'} + B' + \mb{M}_{X'} \sim_{\RR} 0/Z, \]
    we can conclude that $X'$ is of Fano type over $Z$.
    Therefore, by Theorem~\ref{Fano-type-lCY-fib}, $\sdeg (S^{\nor}/Z)$
    is bounded from above depending only on $d,t$.
\end{proof}

%%%%%%%%%%%%%%%%%%%%%%%%%%%%%%%%%%%%%%%%%%%%
%%%%%%%%%%%%%%%%%%%%%%%%%%%%%%%%%%%%%%%%%%%%

\medskip

\printbibliography
 
\vspace{1em}
 
\noindent\small{Caucher Birkar} 

\noindent\small{\textsc{Yau Mathematical Sciences Center, Tsinghua University, Beijing, China} }

\noindent\small{Email: \texttt{birkar@mail.tsinghua.edu.cn}}

\vspace{1em}
 
\noindent\small{Santai Qu} 

\noindent\small{\textsc{Institute of Geometry and Physics, University of Science and Technology of China, Hefei, Anhui Province, China} }

\noindent\small{Email: \texttt{santaiqu@ustc.edu.cn}}

\end{document}